\documentclass[final,onefignum,onetabnum]{siamart250211}

\usepackage{setspace}
\usepackage{microtype}
\usepackage[numbers]{natbib}
\usepackage{algorithm}
\usepackage{stmaryrd}
\usepackage{mathtools}
\usepackage{amsmath} 
\usepackage{amssymb}
\usepackage{amsfonts}
\usepackage{mathrsfs}
\usepackage{enumitem}
\usepackage{comment}
\usepackage{diagbox}
\usepackage{capt-of}
\usepackage{bm}
\usepackage{xcolor}
\usepackage{soul}
\usepackage{cancel}
\usepackage{oplotsymbl}
\usepackage{multirow}
\usepackage{algpseudocode}
\usepackage{tabto}
\usepackage[usestackEOL]{stackengine}[2013-10-15]
\usepackage{subcaption}

\newcommand{\bias}{\mathrm{Bias}}

\DeclareMathOperator\supp{supp}

\newsiamremark{remark}{Remark}

\definecolor{red1}{HTML}{D53E4F}
\definecolor{orange1}{HTML}{FDAE61}
\definecolor{paleorange1}{HTML}{FDAE61}
\definecolor{paleyellow-orange1}{HTML}{FEE08B}
\definecolor{paleyellow-green1}{HTML}{E6F598}
\definecolor{palegreen1}{HTML}{ABDDA4}
\definecolor{green1}{HTML}{66C2A5}
\definecolor{blue1}{HTML}{3288BD}
\definecolor{purple1}{HTML}{8073AC}
\usepackage{tikz}
\usetikzlibrary{patterns,positioning}

\newcommand{\meshdots}[5]{%
  \begin{scope}[shift={(#1,#2)}]
    \pgfmathtruncatemacro{\ma}{2^(#3)}
    \pgfmathtruncatemacro{\mb}{2^(#4)}
    \foreach \p in {1,...,\ma}{
      \foreach \q in {1,...,\mb}{
        \fill[black] ({#5*(\p-0.5)/\ma},{#5*(\q-0.5)/\mb}) circle (0.027);
      }
    }
  \end{scope}
}

\newcommand{\fullmeshdots}[5]{%
  \begin{scope}[shift={(#1,#2)}]
    \pgfmathtruncatemacro{\ma}{2^(#3)}
    \pgfmathtruncatemacro{\mb}{2^(#4)}
    \foreach \p in {1,...,\ma}{
      \foreach \q in {1,...,\mb}{
        \fill[black] ({#5*(\p-0.5)/\ma},{#5*(\q-0.5)/\mb}) circle (0.020);
      }
    }
  \end{scope}
}
\newcommand{\cle}[1]{#1}
\newcommand{\cleparagraph}{}

\def\x{\bm{x}}

\newenvironment{fabrice}%
  {\par}%
  {\par}%
\newcommand{\Fabrice}[1]{#1}

\newcommand{\IMT}{Universit\'e de Toulouse; UPS, INSA, UT1, UTM, Institut de Math\'ematiques de Toulouse,
CNRS, Institut de Math\'ematiques de Toulouse UMR 5219, F-31062 Toulouse, France.}
\newcommand{\SOR}{{Inria, Concace joint team between Airbus CR\&T, Cerfacs and Inria}, {Talence}, {France}, \email{clement.guillet@inria.fr}}
\newcommand{\cor}{Corresponding author}

\headers{Adaptive hierarchical sparse-grid PIC method}{F. Deluzet, C. Guillet, J. Narski, P. Pace}
\title{Hierarchical sparse-grid particle-in-cell method with locally adaptive mesh refinement}



\author{
Fabrice Deluzet\footnotemark[3]
\and Cl\'ement Guillet\footnotemark[2]\;\,$^*$\and Jacek Narski\footnotemark[3] \and Paul Pace\footnotemark[3]
}

\begin{document}

\maketitle

\footnotetext[2]{\SOR}
\footnotetext[3]{\IMT}
\footnotetext[1]{\cor}

{\cleparagraph
\begin{abstract}	
In this paper, we introduce new approximation spaces and a locally adaptive refinement strategy for the hierarchical sparse-grid PIC (HSG-PIC) method to improve the bias while preserving the noise-reduction properties of sparse-grid methods. We first propose an energy-based approximation space, which optimizes the relation between the $\mathrm{H}^1$-norm error and the number of degrees of freedom, together with a family of generalized sparse-grid spaces that continuously connects classical sparse-grid and full-grid approximations. We then develop a locally adaptive approximation strategy based on hierarchical surpluses, combined with an efficient incremental refinement algorithm that avoids solving the Galerkin problem on the complete generalized space.
Numerical experiments demonstrate that the proposed adaptive HSG-PIC method substantially improves the approximation of solutions with localized structures while maintaining the statistical advantages of sparse-grid discretizations. Compared with a standard full-grid PIC method, the adaptive approach achieves comparable or higher accuracy with a significantly reduced number of mesh nodes and particles. These results demonstrate the potential of adaptive sparse-grid PIC methods for efficient simulations of kinetic plasmas with complex solution structures.
	\end{abstract}}



\section{Introduction}
Particle-in-cell (PIC) methods are among the most widely used numerical
approaches for the simulation of kinetic plasmas governed by the
Vlasov--Poisson system~\cite{hockney88,birdsall18}. They combine a
Lagrangian discretization of the Vlasov equation, in which the
distribution function is represented by a finite collection of
macro-particles whose trajectories are integrated in time, with a
mesh-based resolution of the Poisson equation for the self-consistent
electrostatic field. This particle representation introduces an
irreducible statistical error: sampling the distribution function with a
finite number of particles generates noise in all mesh-based quantities.

To analyze this error rigorously, it
is standard to decompose the spatial error into two components
of fundamentally different nature. The first component is the
deterministic bias, or grid-based error, which measures the accuracy with
which the mean values of the grid quantities approximate their exact counterparts. It is governed by the mesh resolution and the order of the
interpolation scheme.
The second component is the statistical noise, a random
fluctuation around this mean value whose magnitude is determined by the mean
number of particles per cell. For a standard PIC scheme on a uniform
grid of size $h$ in $d$ dimensions, the variance of the
statistical noise on the charge density scales as
\[
\mathcal{O}\!\left((\sqrt{N h^d})^{-1/2}\right)\]
in the $\mathrm{L}^\infty$-norm~\cite{ricketson17,deluzet22}, where $N$ is the total number of
particles.  This slow convergence makes noise
reduction \Fabrice{critical} in three-dimensional simulations, where the
statistical noise rapidly dominates the total error. This limitation has motivated numerous noise-reduction techniques, including the $\delta f$
method~\cite{aydemir94,denton95,sydora99},  filtering in the Fourier and
wavelet domains~\cite{birdsall18,gassama07}, and micro--macro
decompositions~\cite{crestetto18}.

Sparse-grid methods provide an alternative approach to noise reduction,
exploiting the structure of the approximation space rather than
post-processing the particle data. Two distinct frameworks have been
investigated in the context of PIC methods. In the sparse-grid combination technique (SGCT) PIC method~\cite{ricketson17,deluzet22,guillet24}, the
charge density is reconstructed by linearly combining partial
solutions computed on a hierarchy of coarse full component grids, with carefully chosen
combination coefficients~\cite{griebel90,bungartz04}. 
In the hierarchical sparse-grid (HSG) PIC method~\cite{guillet25,deluzet26}, 
the approximation space is instead built from a truncated
tensor product of one-dimensional multiresolution analyses~\cite{bungartz91}. In both cases, the mean number of particles per
cell on the underlying grid is substantially larger than that
on a full fine grid, resulting in a significant reduction in statistical noise at comparable computational cost. Specifically, for a
sparse-grid-, i.e., SGCT- or HSG-PIC scheme, the statistical noise scales
as
\[
\mathcal{O}\!\left({|\log h|^{(d-1)/2}}(\sqrt{N \Fabrice{h}})^{-1/2}\right)
\]
in the $\mathrm{L}^1$- or $\mathrm{L}^\infty$-norms~\cite{deluzet22,deluzet26}, compared \Fabrice{to 
\[
\mathcal{O}\!\left((\sqrt{N h^d})^{-1/2}\right)
\] of standard PIC methods} the noise reduction is substantial, growing with
the dimension $d$, and becomes decisive in three-dimensional
configurations~\cite{deluzet22-1,deluzet23,garrigues24,garrigues24-1}.

This noise reduction comes at the price of a modified bias. The sparse-grid space approximates well
functions whose mixed derivatives of high order are bounded, i.e.,
functions in the mixed Sobolev space
$\mathrm{H}^{p+1}_{\mathrm{mix}}(\Omega)$, for B-splines of degree $p$.
Under this regularity assumption, the bias of the charge density scales
as \[
\mathcal{O}\!\left(h^{p+1}|\log h|^{d-1}\right)
\]
in the $\mathrm{L}^2$- or $\mathrm{L}^\infty$-norm~\cite{deluzet25,deluzet26}, compared to $\mathcal{O}(h^{p+1})$ for the
full-grid method: the bias is slightly degraded by the polylogarithmic
factor $|\log h|^{d-1}$, but remains of the same polynomial order
in the mesh size. However, the requirement of bounded mixed derivatives is a
strong assumption, which fails precisely for the solutions of great
physical interest: distributions presenting sharp gradients, thin
filaments, or strongly anisotropic structures, as observed in diocotron
instabilities~\cite{deluzet22,garrigues21}, gradient-driven
modes~\cite{garrigues24,garrigues24-1}, or filamentation phenomena.
For such solutions, the bias of sparse-grid methods \Fabrice{deteriorates}
as \Fabrice{evidenced} by numerical experiments~\cite{deluzet25,deluzet26}.

Several strategies have been proposed to improve the bias of SGCT-PIC methods while preserving their noise-reduction properties. The
truncated method~\cite{muralikrishnan21} shifts the grid hierarchy
toward finer resolutions by removing coarse components, thereby reducing the bias at the expense of increased statistical noise
and a larger approximation space. Another approach introduces
high-order shape functions~\cite{deluzet25} to increase the
polynomial convergence order of the bias. While this strategy preserves
the sparse-grid noise reduction and improves the bias for smooth
solutions, its benefit is limited for distributions with localized sharp
features, where the required mixed-derivative regularity is lost.

{\cleparagraph
This fundamental limitation motivates the present work, whose contributions are twofold. First, building upon the HSG-PIC method introduced in~\cite{deluzet26}, we introduce two new families of approximation spaces: an energy-based sparse-grid space, designed to minimize the $\mathrm{H}^1$-norm error with respect to the number of degrees of freedom, and a family of generalized sparse-grid spaces that continuously bridges the traditional sparse-grid and full-grid approximation spaces. The latter provides the building block underlying our second contribution. 

The second contribution is a locally adaptive sparse-grid approximation space together with an efficient mesh refinement algorithm. Rather than improving the approximation uniformly over the computational domain, the proposed strategy concentrates the degrees of freedom where the solution requires them the most. This approach naturally exploits the hierarchical structure of the HSG-PIC method: the hierarchical surpluses, i.e., the coefficients associated with the basis functions at each level of the hierarchy, provide reliable local error indicators that identify regions where the approximation is insufficient. Accordingly, only basis functions whose hierarchical surpluses satisfy a prescribed refinement criterion are retained in the adaptive approximation space. Adaptive refinement strategies based on hierarchical surpluses have been successfully employed in a wide range of sparse-grid applications~\cite{griebel99,pfluger10,bokanowski13}. To construct this space efficiently, we develop an incremental refinement algorithm that explores the additional hierarchical subspaces of the generalized sparse-grid space while maintaining small linear systems through Schur complement eliminations, thereby avoiding the resolution of the linear system on the complete generalized approximation space. In contrast, incorporating local adaptivity into the SGCT-PIC method is considerably more challenging~\cite{obersteiner21,obersteiner21-1}, since the component grids are globally defined and cannot be refined locally without destroying the combination formula. Numerical experiments demonstrate that the proposed adaptive strategy substantially improves the bias for solutions with localized structures while preserving the statistical noise reduction and computational efficiency of sparse-grid PIC methods.
}

The remainder of the paper is organized as follows. In
\cref{sec:pic}, we introduce the Vlasov--Poisson system, the standard
PIC algorithm, and the SGCT-PIC method. In~\cref{sec:hsgpic}, we present the
HSG-PIC method, its different approximation space variants, and the
proposed locally adaptive mesh refinement strategy. Numerical
experiments assessing the accuracy and computational efficiency of the
method are reported in \cref{sec:num}. Finally,
\cref{sec:conclusion} summarizes the main findings and discusses
directions for future work.

Throughout the paper, the symbol $\lesssim$ denotes an inequality holding up to a positive multiplicative constant independent of the discretization parameters, and we write $\log := \log_2$.

\section{Continuous model and PIC methods}
\label{sec:pic}
\subsection{Continuous model}
Let $d$ denote the dimensionality of the problem, and $\Omega\subset \mathbb{R}^d$ be the spatial domain of interest, which, in this paper, is the \Fabrice{$d$}-dimensional torus $\Omega=\mathbb{T}^d$. 

The ions are treated as a neutralizing background and all variables are expressed in dimensionless form, with characteristic scales being the Debye length and the plasma period $\omega_P^{-1}$, given by
\[
\lambda_D = \sqrt{\frac{\varepsilon_0 T_e}{q_e n_0}},\qquad
\omega_p =\left(\sqrt{\frac{q_e n_0}{m_e \varepsilon_0}}\right).
\]
In this normalization setting, the electron mass $m_e$, characteristic temperature $T_e$, charge $q_e$,  typical electron density $n_0$ and vacuum permittivity $\varepsilon_0$ are set to unity.

\Fabrice{Denoting $f$ the phase-space distribution function of the electron species, t}he noncollisional, nonrelativistic, dimensionless Vlasov--Poisson system with an external magnetic field $\bm{B}$ is considered. \Fabrice{Let $\bm{E}$ and $\Phi$ be the electric field and potential and $\rho_0$ the ion density, the system writes:}
\begin{equation*}
    \left\{\begin{aligned}
   & \frac{\partial f}{\partial t}
+ \bm{v} \cdot \bm{\nabla}_{\bm{x}} f
- \left( \bm{E} + \bm{v} \times \bm{B} \right) \cdot \bm{\nabla}_{\bm{v}} f = 0, \\ 
&\nabla \cdot \bm{E}=\Fabrice{\rho_0}-\rho, \quad \bm{E} = -\bm{\nabla}\Phi,
\end{aligned}
\right.
\end{equation*}
where the electron charge density is defined by 
\[
\rho = \int_{\mathbb{R}^d} fd \bm{v}.
\]
The notation $\bm{\nabla}_*$ denotes the gradient operator with respect to the variable $*$, and we use the shorthand notation $\bm{\nabla} := \bm{\nabla}_{\bm{x}}$.

\subsection{STD-PIC method}
\label{sec:stdpic}
Particle-in-cell (PIC) approximations are reference particle methods applied to plasma physics. Let us briefly introduce in this section the key ingredients of these methods for electrostatic models. We refer the reader to the books~\cite{hockney88,birdsall18} for a more detailed presentation.
In the following the time dependency is omitted where unambiguous. 

\paragraph{Particle approximation}
The particle distribution is represented by a collection of $N$ numerical particles whose positions and velocities are denoted $\left(\bm{x}_{p}(t), \bm{v}_{p}(t)\right)$, for $p=1, \ldots, N$, and is approximated by
\begin{equation*}
  f_{N}(\bm{x}, \bm{v}):=\sum_{p=1}^{N} w_{p} \delta_{\bm{x}_{p}}(\bm{x}) \delta_{\bm{v}_{p}}(\bm{v}),  
\end{equation*}
where $\delta_{\bm{x}_{p}}, \delta_{\bm{v}_{p}}$ are the Dirac delta functions centered either at the particle position or velocity and $\omega_{p}$ is the weight of a numerical particle.  

\paragraph{Charge density approximation} The electron charge density is approximated by
\begin{equation}
\label{eq:9}
    \begin{aligned}
            \rho_N(\bm{x}) &:= \int_{\mathbb{R}^{d}} f_N(\bm{x}, \bm{v}) d \bm{v} = \sum_{p=1}^{N} w_{p} \delta_{\bm{x}_{p}}(\bm{x}).
    \end{aligned}
\end{equation}
We refer to $\rho_N$ as the raw Monte-Carlo density estimator: it is a sum of Dirac masses, and therefore a distribution rather than an $\mathrm{L}^2(\Omega)$ function, so that it cannot be evaluated pointwise or represented on a mesh without a further regularization or projection step. The standard PIC (STD-PIC) and SGCT-PIC schemes described in this section rely on a regularized estimator, obtained by convolving the raw Monte-Carlo estimator with a smooth kernel \cite{cottet84}, referred to as shape function \cite{birdsall18} and denoted $W$. One of the most widely used shape functions is the hat function, defined by
\begin{equation*}
    W_{\bm{h}}(\bm{x}):=\bigotimes_{i=1}^{d} W_{h_i}(x_i), \qquad
    W_{h_i}(x_i):={\frac{1}{h_i}}\max \left(1-\frac{|x_i|}{h_i}, 0\right),
\end{equation*}
where $\bm{h}=(h_1,\ldots,h_d)\in \mathbb{N}^d$ is the mesh size of the grid used for the discretization of the field equation.
The components of the mesh size define the size of the grid cells along each dimension. A uniform Cartesian grid is considered, parameterized by $h$, the size of the grid cells along any dimension, and denoted $\Omega_h$. Let $\bm{j}\in \mathbb{N}^d$ be a multi-index, the uniform Cartesian grid, also called full grid, is defined by 
\begin{align}
\label{eq:omega_h}
\Omega_{h}:=\left\{\bm{x}_{\bm{j}}:=\bm{j} h \mid \bm{j} \in I_{h}\right\}, \quad I_{h}:=\llbracket 0, h^{-1}-1 \rrbracket^d \subset \mathbb{N}^{d}.
\end{align}
This yields the definition of the regularized density estimator,
\begin{equation}
\label{eq:rho_N_hat}
    {\rho}_{h,N}(\x) := (W_{\bm{h}} * \rho_N)(\x) = \sum_{p=1}^N w_p W_{\bm{h}}(\x - \x_p),
\end{equation}
which, unlike the raw density estimator, is a continuous function and can therefore be evaluated pointwise at the mesh nodes.

\paragraph{Electric field computation}
The electric field approximation is carried out on the full grid thanks to the discrete systems
\begin{equation*}
 \hat{\bm{E}}_{h,N}= - \nabla_h \hat{\Phi}_{h,N}, \quad  -\Delta_h \hat{\Phi}_{h,N}=\rho_0 - \hat{\rho}_{h,N}, 
\end{equation*}
where 
\[
\hat{\rho}_{h,N}:= \left({\rho}_{h,N}(\x_{\bm{j}})\right)_{\bm{j}\in I_h}
\]  is the approximation of the charge density on the full grid and $\hat{\bm{E}}_{h,N}$ is the vector of the nodal values of the electric field carried by the nodes of the full grid, denoted by  $\bm{E}_{\bm{j}}$, for $\bm{j}\in I_h$, while $\Delta_h$ and $\nabla_h$ are finite difference approximations of the differential operators.

\paragraph{Interpolation of electric field}
Finally, an interpolant of the electric field is constructed from the nodal values using the shape functions, following 
\begin{equation*}
    \bm{E}_{h,N} (\bm{x})= \sum_{\bm{j}\in I_h} \bm{E}_{\bm{j}} W_h(\bm{x}_{\bm{j}}-\bm{x}).
\end{equation*}

\paragraph{Evolution of particles}
The particle properties are updated by integration of Newton's law using a leapfrog scheme with a time step $\Delta t$:
\begin{equation*}
    \begin{aligned}
        \bm{v}_p^{\kappa+1/2} = \bm{v}_p^{\kappa-1/2} +\Delta t \,  \bm{E}_{h,N}^\kappa (\bm{x}_p^{\kappa}),\qquad 
        \bm{x}_p^{\kappa+1} = \bm{x}_p^\kappa + \Delta t \, \bm{v}_p^{\kappa+1/2},
    \end{aligned}
\end{equation*}
for all $\kappa=0,\ldots,T/\Delta t$, where $T$ is the final time of the simulation.

\subsection{SGCT-PIC method}
\label{sec:sgct}
The philosophy of the sparse-grid combination technique-PIC
method, hereafter referred to as the SGCT-PIC method, is to keep the algorithmic structure of standard PIC entirely unchanged replacing the single full grid on which it operates by a small family of independent, anisotropic Cartesian grids, referred to as component grids, each coarser than the full grid in at least one direction. Relative to STD-PIC, only three steps genuinely differ. First, the density is deposited independently on every component grid, so that the field equation can be solved on each of them. Then, rather than directly using the field obtained from a single grid, SGCT-PIC recombines the fields computed independently on each component grid into a single approximation, through a linear combination with explicit, dimension-dependent coefficients~\cite{ricketson17,deluzet22,garrigues21}. It is this recombination, detailed below, that reconstructs an accuracy comparable to that of the fine full grid from a collection of much coarser and cheaper component-grid solutions. We briefly recall the method here and refer to~\cite{deluzet22,deluzet25} for further details\Fabrice{.}
\paragraph{Component grids}
For a multi-index $\bm{\ell}=(\ell_1,\ldots,\ell_d)\in\mathbb{N}^d$, and a mesh size $h_{\bm{\ell}} := \left(2^{-\ell_1},\ldots,2^{-\ell_d}\right)$, 
we denote $\Omega_{h_{\bm{\ell}}}$ the anisotropic Cartesian grid defined as in~\cref{eq:omega_h}, with a mesh size $2^{-\ell_i}$ along dimension $i$. For a target resolution $n\in\mathbb{N}^*$, the combination technique selects the family of component grids
\begin{equation}
\begin{aligned}
\label{eq:sgct:levels}
\mathscr{L}_h &:= \left\{ \bm{\ell}\in\mathbb{N}^d \;\middle|\; n \leq |\bm{\ell}|_1 \leq n+\Fabrice{\mathcal{D}}-1\right\}, \qquad n= |\log h|,  
\end{aligned}
\end{equation}
i.e., the component grids lying on the $\Fabrice{\mathcal{D}}$ outer diagonals of level $n,\ldots,n+\Fabrice{\mathcal{D}}-1$ of the level lattice $\mathbb{N}^d$ as depicted on \cref{fig:sgct}. Within this family, the mesh resolution is traded off between directions: a component grid may resolve one direction as finely as, or even more finely than, the isotropic full grid of mesh
size $h=2^{-n}$, at the price of a much coarser resolution in the remaining directions, the total resolution $|\bm{\ell}|_1$ being held between $n$ and $n+\Fabrice{\mathcal{D}}-1$ across the family.

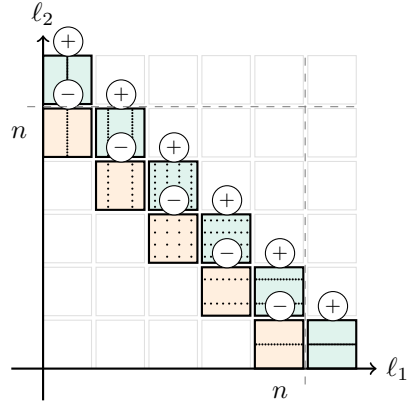
\begin{figure}[h!]
\centering
\centering
\begin{tikzpicture}[scale=.7]
  \def\n{4}
  \def\cellsize{0.915}
  \foreach \i in {0,...,5}{
    \foreach \j in {0,...,5}{
      \draw[draw=black!12,thin] (\i,\j) rectangle ++(\cellsize,\cellsize);
    }
  }
  \foreach \i/\j in {0/4,1/3,2/2,3/1,4/0}{
    \draw[fill=orange1!20,draw=black,thick] (\i,\j) rectangle ++(\cellsize,\cellsize);
    \fullmeshdots{\i}{\j}{\i}{\j}{\cellsize}
  }
  \foreach \i/\j in {0/5,1/4,2/3,3/2,4/1,5/0}{
    \draw[fill=green1!20,draw=black,thick] (\i,\j) rectangle ++(\cellsize,\cellsize);
    \fullmeshdots{\i}{\j}{\i}{\j}{\cellsize}
  }
  \foreach \i/\j in {0/4,1/3,2/2,3/1,4/0}{
    \node[circle,draw,fill=white,inner sep=0pt,minimum size=11pt] at (\i+0.475,\j+1.18) {\scriptsize $-$};
  }
  \foreach \i/\j in {0/5,1/4,2/3,3/2,4/1,5/0}{
    \node[circle,draw,fill=white,inner sep=0pt,minimum size=11pt] at (\i+0.475,\j+1.18) {\scriptsize $+$};
  }
  \draw[->,thick] (-0.6,0) -- (6.3,0) node[right] {$\ell_1$};
  \draw[->,thick] (0,-0.6) -- (0,6.3) node[above] {$\ell_2$};
  \draw[dashed,black!50] (4.95,-0.3) -- (4.95,5.95);
  \draw[dashed,black!50] (-0.3,4.95) -- (5.95,4.95);
  \node at (4.475,-0.45) {$n$};
  \node at (-0.45,4.475) {$n$};
\end{tikzpicture}
\caption{Schematic representation, for \Fabrice{ a depth $\mathcal{D}=2$}, of the component grids \Fabrice{$\mathscr{L}_h$} defined by \cref{eq:sgct:levels}. Each cell of the level lattice $(\ell_1,\ell_2)$ shows the full nodal mesh of the corresponding anisotropic grid, of resolution $h_{\bm{\ell}}=(2^{-\ell_1},2^{-\ell_2})$, i.e., $2^{\ell_1}\times 2^{\ell_2}$ nodes. The component grids on the diagonal $|\bm{\ell}|_1=n+1$ (green, coefficient $+1$) and on the diagonal $|\bm{\ell}|_1=n$ (orange, coefficient $-1$) are recombined according to \cref{eq:sgct:combi}.}
\label{fig:sgct}
\end{figure}

\paragraph{Charge density approximation and field computation}
For each component grid, an approximation of the
electric field is computed by performing these three operations successively:
\begin{enumerate}[label=(\roman*)]
\item the regularized Monte-Carlo estimator is
defined on the component grid as in \cref{eq:rho_N_hat}, using the
anisotropic shape function $W_{h_{\bm{\ell}}}$; 
\item the discrete Poisson equation \Fabrice{$-\Delta_{h_{\bm{\ell}}}\hat{\Phi}_{h_{\bm{\ell}},N} = \rho_0 - \hat{\rho}_{h_{\bm{\ell}},N}$}
is solved on
the component grid by finite differences;
\item the electric field is computed at the nodes
of the component grid:
\[
\hat{\bm{E}}_{h_{\bm{\ell}},N} :=
-\nabla_{h_{\bm{\ell}}}\hat{\Phi}_{h_{\bm{\ell}},N}
\] and interpolated at the particle positions, defining ${\bm{E}}_{h_{\bm{\ell}},N}$,
using the same anisotropic shape function.
\end{enumerate}

\paragraph{Combination}
The component-grid fields are then recombined at the particle positions according to the combination formula:
\begin{equation}
\label{eq:sgct:combi}
\bm{E}_{h,N}^{(C)}(\bm{x}_p) :=
\sum_{i=0}^{\Fabrice{\mathcal{D}}-1}(-1)^i\binom{\Fabrice{\mathcal{D}}-1}{i}
\sum_{|\bm{\ell}|_1 = n+(\Fabrice{\mathcal{D}}-1)-i} \bm{E}_{h_{\bm{\ell}},N}(\bm{x}_p),
\end{equation}
which combines the fields obtained on each of the $d$ diagonals of \cref{eq:sgct:levels} with alternating signs and binomial weights, as depicted by~\cref{fig:sgct}. The combination formula~\cref{eq:sgct:combi} also defines an approximation of the charge density, upon replacing the electric field interpolants by the component-grid charge density approximations. 

\begin{fabrice}\paragraph{Benefits and limitations}
The sparse-grid combination technique significantly reduces the computational cost of PIC simulations. For a fixed particle-per-cell ratio, the total number of mesh nodes and particles decreases from $\mathcal{O}(h^{-d})$ to $\mathcal{O}\big(h^{-1}|\log h|^{d-1}\big)$. Moreover, the method replaces a single large linear system with $\mathcal{O}\big(|\log h|^{d-1}\big)$ small and independent sparse systems defined on anisotropic component grids. 

A fundamental limitation lies in the restrictive regularity assumptions: the combination technique requires $\mathrm{C}^{q}_{\mathrm{mix}}(\Omega)$ regularity ($\mathrm{H}^{q}_{\mathrm{mix}}(\Omega)$ space is sufficient for Galerkin formulations~\cite{griebel90,bungartz04}, see \cref{sec:hsgpic}). Consequently, solutions exhibiting large mixed derivatives (e.g., sharp gradients, thin filaments, or vortices) may suffer from a significant deterministic bias~\cite{deluzet22,deluzet22-1}, thereby motivating adaptive refinement. However, the global tensor-product structure of the component grids does not easily accommodate local refinement. Although adaptive variants exist~\cite{obersteiner21}, they come at the cost of substantial implementation complexity. Finally, the Cartesian structure of the component grids restricts the standard combination technique to rectangular-like geometries.

These limitations motivate the hierarchical Galerkin reformulation of sparse-grid-PIC methods introduced in~\cite{deluzet26} and recalled in~\cref{sec:hsgpic}.
\end{fabrice}

\section{HSG-PIC methods}
\label{sec:hsgpic}
Whereas STD-PIC discretizes the field equation on a single full grid, and SGCT-PIC recombines independent solutions computed on a collection of full, anisotropic component grids, the hierarchical sparse-grid-PIC (HSG-PIC) method~\cite{deluzet26} departs from this combination-of-full-grids paradigm altogether. Instead, it builds a single approximation space directly as a truncated hierarchical sum of small, local subspaces, ordered by their contribution to the approximation error, and formulates the field equation and the density estimator variationally on this space.
This shift from combining independent full-grid solutions to a single Galerkin problem on a hierarchical basis is what brings more flexibility: the hierarchical structure provides local error indicators, naturally compatible with adaptive refinement. We first introduce,  the hierarchical approximation spaces themselves, before turning to the discretization of the density estimator and of the field equation.
\subsection{Sparse-grid truncations and approximation spaces}
The HSG-PIC method relies on a hierarchical decomposition of a tensor-product B-spline space, and a truncation rule in which the subspaces that are assumed to contribute the least to the approximation according to regularity assumptions are discarded.
\paragraph{Univariate B-spline spaces}
For $p\in\mathbb{N}$ and a level $\ell\in\mathbb{N}$ with associated mesh size $h:=2^{-\ell}$, we introduce 
\[
S_{h}([0,1]) := \left\{ u\in \mathrm{H}^{p}([0,1]) \;\middle|\; u|_{(hj,\,h(j+1))}\in \mathbb{P}^p, \ \forall j\in I_h\right\},
\]
the space of univariate B-splines of degree $p$ and regularity $\mathrm{C}^{p-1}$, associated with the uniform partition of the unit interval of step $h$. In the above, $\mathbb{P}^p$ denotes the space of polynomials of degree at most $p$. This space admits a nodal basis $\{\varphi^p_{h,j}\}_{j\in I_h}$, with index set $I_h$ defined as in \cref{sec:stdpic}.

\paragraph{Tensor-product spaces}
For a multi-index $\bm{\ell}=(\ell_1,\ldots,\ell_d)\in\mathbb{N}^d$ and $h_{\bm{\ell}}=(h_{\ell_1},\ldots,h_{\ell_d})$, with $h_{\ell_i}=2^{-\ell_i}$, the corresponding anisotropic tensor-product B-spline space on the $d$-dimensional unit ball $\Omega$ is defined as \Fabrice{$S_{h_{\bm{\ell}}}(\Omega) := \bigotimes_{i=1}^d S_{h_{\ell_i}}([0,1])$.}

Note that the same polynomial degree is used in every dimension, while the mesh size varies from one dimension to another.

\paragraph{Hierarchical subspaces}
The tensor-product spaces are nested, i.e., $S_{h_{\bm{k}}}(\Omega)\subset S_{h_{\bm{\ell}}}(\Omega)$ whenever $\bm{k}\leq\bm{\ell}$, where the inequality holds componentwise, so that the following hierarchical decomposition holds
\begin{equation}
\label{eq:hier_decomp}
S_{h_{\bm{\ell}}}(\Omega) = \bigoplus_{\bm{k}\leq \bm{\ell}} W_{h_{\bm{k}}}(\Omega),
\qquad
W_{h_{\bm{k}}}(\Omega) := S_{h_{\bm{k}}}(\Omega) \Big/ \bigoplus_{i=1}^d S_{h_{\bm{k}-\bm{e}_i}}(\Omega),
\end{equation}
with the convention that $S_{h_{\bm{k}}}(\Omega):=\{0\}$ as soon as one component of $\bm{k}$ is negative, and where $\bm{e}_i\in\mathbb{N}^d$ denotes the $i$-th canonical unit vector. The space $W_{h_{\bm{k}}}(\Omega)$, referred to as the hierarchical increment, or hierarchical subspace of level $\bm{k}$, collects the degrees of freedom of $S_{h_{\bm{k}}}(\Omega)$ that are not already contained in any coarser tensor-product space $S_{h_{\bm{k}-\bm{e}_i}}(\Omega)$. It is spanned by a subset of the hierarchical B-spline basis functions associated with the resolution $\bm{k}$:
\begin{align*}
  W_{h_{\bm{k}}}(\Omega) = \mathrm{span}\left\{\varphi^p_{h_{\bm{k}},\bm{j}} \;\middle|\; \bm{j}\in J_{h_{\bm{k}}}\right\}, \qquad  \varphi^p_{h_{\bm{k}},\bm{j}}(\bm{x}):=\bigotimes_{i=1}^d\varphi^p_{h_{k_i},j_i}(x_i),
\end{align*}
where the hierarchical index set is defined by
\begin{align}
\label{eq:hier_set}
J_{h_{\bm{k}}} := J_{h_{k_1}}\times\cdots\times J_{h_{k_d}}\subset I_{h_{\bm{k}}},\qquad
J_{h_{k_i}} := \{j\in I_{h_{k_i}} \mid j \text{ odd}\}.
\end{align}
The set $J_{h_{k_i}}$ collects the newly created, i.e., odd, nodal indices at level $k_i$.


\paragraph{Hierarchical decomposition}
For an arbitrary spline degree $p \in \mathbb{N}$, any function $u \in \mathrm{H}^{q}_{\mathrm{mix}}(\Omega)$, $q \leq p+1$, decomposes in the hierarchical basis as
\begin{equation}
\label{eq:hier_decomp_general}
u = \sum_{\bm{\ell}\in\mathbb{N}^d} u_{\bm{\ell}}, \qquad
u_{\bm{\ell}} := \sum_{\bm{j}\in J_{h_{\bm{\ell}}}}
\alpha_{\bm{\ell},\bm{j}}\, \varphi^p_{h_{\bm{\ell}},\bm{j}}
\;\in\; W_{h_{\bm{\ell}}}(\Omega) ,
\end{equation}
where $u_{\bm{\ell}}$ is the hierarchical component of the function at level $\bm{\ell}$, and $\alpha_{\bm{\ell},\bm{j}}$ is the hierarchical surplus associated with the level and node index $\bm{\ell},\bm{j}$. 

The \Fabrice{two} lemmas below specialize this decomposition to the linear
case $p=1$, for which the surplus $\alpha_{\bm{\ell},\bm{j}}$ admits the
explicit, closed-form characterization of \cref{lem:surplus_formula}.

\paragraph{Notation}
For a multi-index $\bm{\alpha} = (\alpha_1,\ldots,\alpha_d)\in
\mathbb{N}^d$, we denote by $D^{\bm{\alpha}}u$ the corresponding weak
mixed derivative of $u$, and by $\bm{p}\in\mathbb{N}^d$
the multi-index with all components equal to $p$, so that
\begin{equation*}
D^{\bm{\alpha}} u(\bm{x})
:= \frac{\partial^{|\bm{\alpha}|_1} u(\bm{x})}
{\partial x_1^{\alpha_1}\cdots\, \partial x_d^{\alpha_d}}\,,
\qquad
D^{\bm{p}} u(\bm{x})
:= \frac{\partial^{dp} u(\bm{x})}
{\partial x_1^{p}\cdots\, \partial x_d^{p}}\,,
\end{equation*}
where $|\bm{\alpha}|_1 := \alpha_1+\cdots+\alpha_d$. In particular,
$D^{\bm{2}}u$ denotes the mixed second derivative of $u$ in every
direction, of total order $2d$, which will appear throughout the
estimates below.

We also introduce the space of $\mathrm{L}^2$-integrable functions with vanishing mean
\[
\mathrm{L}^2_0(\Omega) := \left\{ u \in \mathrm{L}^2(\Omega) \;\middle|\;
\frac{1}{\mu(\Omega)} \int_{\Omega} u \, d\bm{x} = 0
\right\}, \qquad \mu(\Omega) = \int_{\Omega} d\bm{x}.
\]

\begin{lemma}[Closed-form expression of the hierarchical surplus]
\label{lem:surplus_formula}
Let $p=1$, and fix a hierarchical level $\bm{\ell}\in\mathbb{N}^d$ together with a node index $\bm{j}\in J_{h_{\bm{\ell}}}$. We define the surplus kernel associated with this level and node as the tensor product, over each direction $i=1,\ldots,d$, of rescaled basis function at level $\ell_i$:
\[
\Psi_{\bm{\ell},\bm{j}}(\bm{x}) :=
\bigotimes_{i=1}^d \psi_{\ell_i,j_i}(x_i)
= \bigotimes_{i=1}^d \Big({-2^{-(\ell_i+1)}}\,
\varphi^1_{h_{\ell_i},j_i}(x_i)\Big) .
\]
Then, for any $u \in \mathrm{H}^{2}_{\mathrm{mix}}(\Omega)$ with the hierarchical decomposition~\eqref{eq:hier_decomp_general}, the hierarchical surplus at level $\bm{\ell}$ admits the closed-form expression
\begin{equation}
\label{eq:surplus_closed_form}
\alpha_{\bm{\ell},\bm{j}}
= \int_\Omega \Psi_{\bm{\ell},\bm{j}}(\bm{x})\,
D^{\bm{2}} u(\bm{x}) \, d\bm{x} .
\end{equation}
\end{lemma}

\begin{lemma}[Estimates of the hierarchical surplus and components]
\label{lem:surplus_estimate}
Let $p=1$, $u \in \mathrm{H}^{2}_{\mathrm{mix}}(\Omega)$, with the
hierarchical decomposition~\eqref{eq:hier_decomp_general}. For every
level $\bm{\ell}\in\mathbb{N}^d$ and node $\bm{j}\in
J_{h_{\bm{\ell}}}$, the hierarchical surplus and the hierarchical
component satisfy, for
$r\in\{\infty,2\}$,
\begin{align}
\label{eq:surplus_bound_r}
\big|\alpha_{\bm{\ell},\bm{j}}\big|
&\;\lesssim\; 2^{-2|\bm{\ell}|_1}\,
\big\|D^{\bm{2}}u\big\|_{\mathrm{L}^r(\Omega)} , \\
\label{eq:component_bound_r}
\|u_{\bm{\ell}}\|_{\mathrm{L}^r(\Omega)}
&\;\lesssim\; 2^{-2|\bm{\ell}|_1}\,
\big\|D^{\bm{2}}u\big\|_{\mathrm{L}^r(\Omega)} , \\
\label{eq:component_bound_E}
\|u_{\bm{\ell}}\|_{E}
&\;\lesssim\; 2^{-2|\bm{\ell}|_1}\,
\Big(\sum_{i=1}^d 2^{2\ell_i}\Big)^{1/2}
\big\|D^{\bm{2}}u\big\|_{\mathrm{L}^r(\Omega)} ,
\end{align}
where $\|\cdot\|_E := \|\bm{\nabla}\cdot\|_{\mathrm{L}^2(\Omega)}= |\cdot|_{\mathrm{H}^1(\Omega)}$ is
the energy seminorm, with
implicit constants depending only on $d$ and $r$.
\end{lemma}

\begin{proof}[Proof of \cref{lem:surplus_formula,lem:surplus_estimate}]
See~\cite[Lemmas~3--5]{bungartz99}. Although originally formulated for homogeneous Dirichlet boundary conditions, these bounds translate directly to the periodic torus $\Omega = \mathbb{T}^d$ under a zero-mean constraint, imposed by $\mathrm{L}^2_0(\Omega)$.
Indeed, the proof relies on the
closed-form expression~\eqref{eq:surplus_closed_form} of
\cref{lem:surplus_formula} for the surplus, integrated against the
explicit norms of the hierarchical basis functions to obtain
\eqref{eq:surplus_bound_r}--\eqref{eq:component_bound_E}; it is local
and does not invoke the boundary condition, and therefore transfers
unchanged to the periodic setting and to the
zero-mean constraint considered here.
\end{proof}

\paragraph{Truncated hierarchical approximation spaces}
The hierarchical decomposition~\eqref{eq:hier_decomp} naturally suggests a family of approximation spaces, obtained by retaining only the hierarchical subspaces indexed by a prescribed subset of multi-indices.

\begin{definition}[Truncated hierarchical approximation space]
\label{def:Vstar}
For a set of hierarchical indices $\mathscr{H}_{h}^{(\star)} \subset \mathbb{N}^d$, the corresponding truncated hierarchical approximation space is defined as
\begin{equation*}
V_{h}^{(\star)}(\Omega)
:= \bigg\{
u = \sum_{\bm{\ell} \in \mathscr{H}_{h}^{(\star)}} w_{h_{\bm{\ell}}},
\quad \text{where } w_{h_{\bm{\ell}}} \in W_{h_{\bm{\ell}}}(\Omega)
\bigg\}.
\end{equation*}
Several particular choices of the index set $\mathscr{H}_{h}^{(\star)}$, some of which being represented in~\cref{fig:hsg_spaces}, are of interest in the following.
\begin{enumerate}[label=(\roman*)]
\item \label{def:1}
The $\mathrm{L}^2$-based sparse-grid approximation space, denoted by $V_{h}^{(1)}(\Omega)$, is obtained by taking the $\mathrm{L}^2$-based index set
\[
\mathscr{H}_{h}^{(1)}
:= \bigg\{
\bm{\ell} \in \mathbb{N}^d
\;\Big|\;
|\bm{\ell}|_1 \leq |\log h|,\,  \ell_i \geq 1
\bigg\}.
\]
\item \label{def:2}
The energy-, or $\mathrm{H}^1$-based, sparse-grid approximation space, denoted by $V_{h}^{(\mathrm{E})}(\Omega)$, is obtained by taking the energy-based index set
\[
\mathscr{H}_{h}^{(\mathrm{E})}
:= \bigg\{
\bm{\ell} \in \mathbb{N}^d
\;\Big|\;
|\bm{\ell}|_1 -\frac{1}{5}\log\Big(\sum_{i=1}^d4^{\ell_{i}}\Big) \leq |\log h|  -\frac{1}{5}|\log(h^2)|,\,  \ell_i \geq 1
\bigg\}.
\]
\item \label{def:inf}
The full-grid approximation space, denoted by $V_{h}^{(\infty)}(\Omega)$, is obtained by taking the full-grid index set
\[
\mathscr{H}_{h}^{(\infty)}
:= \bigg\{
\bm{\ell} \in \mathbb{N}^d
\;\Big|\;
|\bm{\ell}|_\infty \leq |\log h|,\,  \ell_i \geq 1
\bigg\}.
\]
This space is constructed using the full tensor-product rule of each unidimensional subspace, without proceeding to any truncation.
\cle{\item \label{def:offset}
The generalized\, $\mathrm{L}^2$-based, or sparse-grid, approximation space, denoted by $V_{h,\sigma}^{(1)}(\Omega)$, is obtained by taking the $\mathrm{L}^2$-based index set, augmented with $\sigma \in \mathbb{N}$ additional diagonals, \Fabrice{with $\sigma \in \{0, \ldots, |\log h|\}$}
\[
\mathscr{H}_{h,\sigma}^{(1)}
:=
\bigg\{
\bm{\ell} \in \mathbb{N}^d
\;\Big|\;
|\bm{\ell}|_1 \leq |\log h|+\sigma,\;
1 \leq \ell_i \leq |\log h|
\bigg\}.
\]
The term\textquotedblleft\,generalized\textquotedblright\, reflects that this definition yields a family of approximation spaces ranging from the $\mathrm{L}^2$-based sparse-grid space, obtained for $\sigma=0$, to the full-grid space, obtained for $\sigma=|\log h|$.}
\end{enumerate}
\end{definition}

\begin{remark}[Treatment of the trivial level $\bm{\ell}=\bm{0}$]
\label{rem:trivial_level}
On the periodic torus, the level $\bm{\ell}=\bm{0}$ corresponds to the single constant function on the whole domain. Consistently with the zero-mean constraint imposed by $\mathrm{L}^2_0(\Omega)$, this constant component is removed from the hierarchical decomposition, so that the indices $\bm{\ell}$ entering $\mathscr{H}_{h}^{(\star)}$ effectively satisfy $\ell_i \geq 1$ for at least one direction $i = 1,\ldots,d$. This is the periodic counterpart of the convention adopted for homogeneous Dirichlet boundary conditions in the original sparse-grid construction~\cite{bungartz99, bungartz04}.
\end{remark}

\begin{figure}[h!]
\centering
\begin{subfigure}[b]{0.31\textwidth}
\centering
\begin{tikzpicture}[scale=0.62]
  \def\n{4}
  \def\cellsize{0.893}
  \foreach \i in {0,...,\n}{
    \foreach \j in {0,...,\n}{
      \draw[fill=blue1!18,draw=black,thick] (\i,\j) rectangle ++(\cellsize,\cellsize);
      \meshdots{\i}{\j}{\i}{\j}{\cellsize}
    }
  }
  \draw[->,thick] (-0.6,0) -- (\n+1.3,0) node[right] {\scriptsize$\ell_1$};
  \draw[->,thick] (0,-0.6) -- (0,\n+1.3) node[above] {\scriptsize$\ell_2$};
  \draw[dashed,black!50] (\n+\cellsize,-0.3) -- (\n+\cellsize,\n+\cellsize);
  \draw[dashed,black!50] (-0.3,\n+\cellsize) -- (\n+\cellsize,\n+\cellsize);
  \node at (\n+0.45,-0.5) {\scriptsize $n$};
  \node at (-0.5,\n+0.45) {\scriptsize $n$};
\end{tikzpicture}
\caption{$V_{h}^{(\infty)}(\Omega)$: $|\bm{\ell}|_\infty\leq |\log h|$.}
\label{fig:Vinf}
\end{subfigure}
\hfill
\begin{subfigure}[b]{0.31\textwidth}
\centering
\begin{tikzpicture}[scale=0.62]
  \def\n{4}
  \def\cellsize{0.893}
  \foreach \i in {0,...,\n}{
    \foreach \j in {0,...,\n}{
      \pgfmathtruncatemacro{\s}{\i+\j}
      \ifnum\s>\n
        \draw[draw=black!25,thin] (\i,\j) rectangle ++(\cellsize,\cellsize);
      \else
        \draw[fill=green1!22,draw=black,thick] (\i,\j) rectangle ++(\cellsize,\cellsize);
        \meshdots{\i}{\j}{\i}{\j}{\cellsize}
      \fi
    }
  }
  \draw[->,thick] (-0.6,0) -- (\n+1.3,0) node[right] {\scriptsize$\ell_1$};
  \draw[->,thick] (0,-0.6) -- (0,\n+1.3) node[above] {\scriptsize$\ell_2$};
  \draw[dashed,black!50] (\n+\cellsize,-0.3) -- (\n+\cellsize,\n+\cellsize);
  \draw[dashed,black!50] (-0.3,\n+\cellsize) -- (\n+\cellsize,\n+\cellsize);
  \node at (\n+0.45,-0.5) {\scriptsize $n$};
  \node at (-0.5,\n+0.45) {\scriptsize $n$};
\end{tikzpicture}
\caption{$V_{h}^{(1)}(\Omega)$: $|\bm{\ell}|_1\leq |\log h|$.}
\label{fig:V1}
\end{subfigure}
\hfill
\begin{subfigure}[b]{0.31\textwidth}
\centering
\begin{tikzpicture}[scale=0.62]
  \def\n{4}
  \def\cellsize{0.893}
  \foreach \i in {0,...,\n}{
    \foreach \j in {0,...,\n}{
      \draw[draw=black!25,thin] (\i,\j) rectangle ++(\cellsize,\cellsize);
    }
  }
  \foreach \i/\j in {0/0,0/1,0/2,0/3,0/4,1/0,1/1,1/2,2/0,2/1,3/0,4/0}{
    \draw[fill=purple1!22,draw=black,thick] (\i,\j) rectangle ++(\cellsize,\cellsize);
    \meshdots{\i}{\j}{\i}{\j}{\cellsize}
  }
  \draw[->,thick] (-0.6,0) -- (\n+1.3,0) node[right] {\scriptsize$\ell_1$};
  \draw[->,thick] (0,-0.6) -- (0,\n+1.3) node[above] {\scriptsize$\ell_2$};
  \draw[dashed,black!50] (\n+\cellsize,-0.3) -- (\n+\cellsize,\n+\cellsize);
  \draw[dashed,black!50] (-0.3,\n+\cellsize) -- (\n+\cellsize,\n+\cellsize);
  \node at (\n+0.45,-0.5) {\scriptsize $n$};
  \node at (-0.5,\n+0.45) {\scriptsize $n$};
\end{tikzpicture}
\caption{$V_{h}^{(\mathrm{E})}(\Omega)$: \cref{def:2}.}
\label{fig:VE}
\end{subfigure}
\\[-2ex]  
\begin{subfigure}[c]{0.31\textwidth}   
\centering
\begin{tikzpicture}[scale=0.62]
  \def\n{4}
  \def\sigma{2}
  \def\cellsize{0.893}
  \foreach \i in {0,...,\n}{
    \foreach \j in {0,...,\n}{
      \pgfmathtruncatemacro{\s}{\i+\j}
      \ifnum\s>\numexpr\n+\sigma\relax
        \draw[draw=black!25,thin]
          (\i,\j) rectangle ++(\cellsize,\cellsize);
      \else
        \draw[fill=orange!25,draw=black,thick]
          (\i,\j) rectangle ++(\cellsize,\cellsize);
        \meshdots{\i}{\j}{\i}{\j}{\cellsize}
      \fi
    }
  }
  \draw[->,thick] (-0.6,0) -- (\n+1.3,0)
      node[right] {\scriptsize$\ell_1$};
  \draw[->,thick] (0,-0.6) -- (0,\n+1.3)
      node[above] {\scriptsize$\ell_2$};
  \draw[dashed,black!50]
      (\n+\cellsize,-0.3) -- (\n+\cellsize,\n+\cellsize);
  \draw[dashed,black!50]
      (-0.3,\n+\cellsize) -- (\n+\cellsize,\n+\cellsize);
  \node at (\n+0.45,-0.5) {\scriptsize $n$};
  \node at (-0.5,\n+0.45) {\scriptsize $n$};
\end{tikzpicture}
\caption{$V_{h,\sigma}^{(1)}(\Omega)$ with $\sigma=2$.}
\label{fig:V1sigma}
\end{subfigure}
\caption{
\Fabrice{Two-dimensional schematic of the hierarchical subspaces $W_{h_{\bm{\ell}}}(\Omega)$ selected in each approximation space for a fixed target level $n$. Each filled cell of the lattice $(\ell_1,\ell_2)$ marks the mesh points added by the corresponding increment, whose count $m(\ell_i)=2^{\ell_i-1}$ ($m(0)=1$) per direction grows with the level. (a) Full-grid: all levels up to $n$. (b) $\mathrm{L}^2$-based: the lower-diagonal levels $|\bm{\ell}|_1\leq n$. (c) $\mathrm{H}^1$-based: additionally drops the near-isotropic combinations near the diagonal (e.g.\ $(2,2)$, $(3,1)$, $(1,3)$, absent here but present in (b)). (d) Generalized: augments the $\mathrm{L}^2$-based sparse-grid index set with the next $\sigma$ diagonals (here $\sigma=2$).}}
\label{fig:hsg_spaces}
\end{figure}
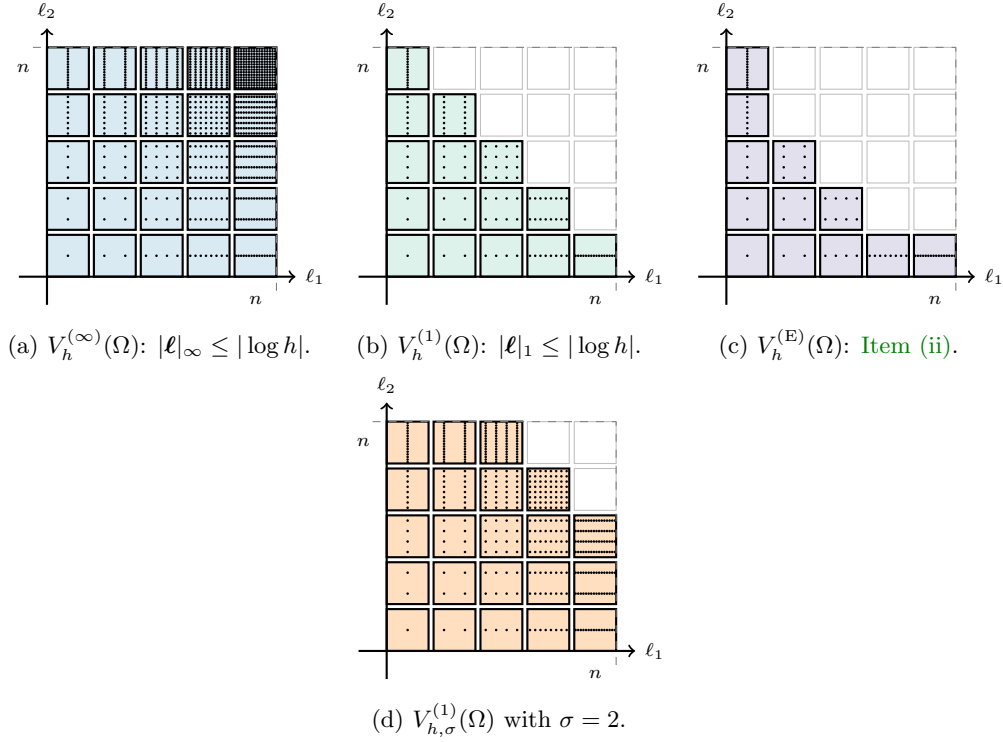

Both the $\mathrm{L}^2$- and $\mathrm{H}^1$-based index sets answer the same underlying problem: given a fixed number of degrees of freedom, which truncation of the hierarchical decomposition~\eqref{eq:hier_decomp} minimizes the resulting approximation error, measured in a prescribed norm? The $\mathrm{L}^2$-norm-based sparse-grid space is constructed so as to minimize, at a fixed number of degrees of freedom, the $\mathrm{L}^2$ norm of the approximation error~\cite{bungartz04}: among all possible truncations of the hierarchical basis, the criterion $|\bm{\ell}|_1\leq |\log h|$ of \cref{def:1} is the one that retains, for a given budget of unknowns, the hierarchical subspaces contributing the most to the reduction of the $\mathrm{L}^2$ error.
The energy-based sparse-grid space is constructed so as to minimize, at a fixed number of degrees of freedom, the $\mathrm{H}^1$ norm of the error~\cite{bungartz04}. The $\mathrm{H}^1$ norm is a relevant metric to control the accuracy of PIC methods since it is the gradient of the potential that actually governs the dynamic of the particle pusher.

\begin{proposition}
\label{prop:3}
The energy-based, generalized $\mathrm{L}^2$-based, and full-grid approximation spaces are nested as
\begin{equation}
\label{eq:embedding}
V_{h}^{(\mathrm{E})}(\Omega) \subset V_{h}^{(1)}(\Omega) \subset V_{h,\sigma}^{(1)}(\Omega) \subset V_{h}^{(\infty)}(\Omega), \quad \sigma=1,\ldots,|\log h|-1
\end{equation}
and the dimension of the spaces, i.e. their number of mesh nodes, are given by
\begin{align*}
\dim (V_{h}^{(\mathrm{E})}) =  \mathcal{O}(h^{-1}), \quad
 \dim(V_{h}^{(1)})= \mathcal{O}\!\left(|\log h|^{d-1}h^{-1}\right),  \quad
    \dim(V_{h}^{(\infty)})= \mathcal{O}(h^{-d}).
\end{align*}
\end{proposition}
\begin{proof}[Proof of \cref{prop:3}]
The embedding~\eqref{eq:embedding} follows from~\cite[lemma 3.9]{bungartz04}; the complexity estimates follow from~\cite[lemma 3.6 and equation (3.31)]{bungartz04}.
\end{proof}
 
\begin{remark}[The economic and analytical paradox of energy-based truncation]
\label{rem:paradox_truncation}
The $\mathrm{L}^2$-based and energy-based index sets select fundamentally different hierarchical subspaces, which accounts for the difference in grid complexity. Understanding why the energy-based space is strictly smaller than the $\mathrm{L}^2$-based space may seem paradoxical at first glance, since controlling the $\mathrm{H}^1$ norm is a strictly stronger requirement than controlling the $\mathrm{L}^2$ norm. The resolution lies in evaluating which hierarchical subspaces deliver the best accuracy per degree of freedom ratio for each respective norm. This paradox is directly reflected in the analytical error estimates provided by \cref{lem:surplus_estimate}.

   \paragraph{The $\mathrm{L}^2$ viewpoint (Multiplicative profit)} According to the standard error bounds \cref{eq:component_bound_r}, the contribution of a hierarchical subspace to the global $\mathrm{L}^2$ error depends solely on the total level $|\bm{\ell}|_1$, regardless of its distribution. 
 Consequently, near-isotropic multi-index ($\ell_1 \approx \cdots \approx \ell_d$) subspaces and highly anisotropic ones are treated indiscriminately in the truncation, which explains the complexity of the $\mathrm{L}^2$-based approximation space inherits the dimension-dependent factor $\mathcal{O}(|\log h|^{d-1})$.

\paragraph{The $\mathrm{H}^1$ viewpoint (Additive penalty)}  
Conversely, evaluating fields under the energy norm introduces the gradient operator, which acts additively across coordinate directions via the partial derivatives ($\sum_{i=1}^d \|\partial_{x_i} \cdot\|_{\mathrm{L}^2}^2$). As established in the error estimate~\cref{eq:component_bound_E}, the bounds for a subspace are penalized by the Euclidean norm of the directional derivative scales, namely $(\sum_{i=1}^d 2^{2\ell_i})^{1/2} = (\sum_{i=1}^d 4^{\ell_i})^{1/2}$. This penalty completely destroys the multiplicative advantage of isotropic refinements. To incorporate this accuracy benefit degradation into an applicable selection criterion, the fractional cost-benefit ratio is linearized by taking its logarithm in base 2. This mathematical transformation directly yields the corrective term $\frac{1}{5}\log_2\big(\sum_{i=1}^d 4^{\ell_i}\big)$ featured in the definition of the energy-based index set. Since near-isotropic blocks maximize this logarithmic penalty, they fail the threshold inequality and are aggressively filtered out, successfully compressing the final spatial complexity to a dimension-independent $\mathcal{O}(h^{-1})$.
\end{remark}

\subsection{Resolution of field equations}
\subsubsection{Variational formulation of the Poisson equation}
\label{sec:poisson_variational}
We now turn to the resolution of the Poisson equation for determining the electric
potential.

\paragraph{Continuous problem}
A variational formulation of the Poisson equation is considered within the space of functions with vanishing mean, it writes as:
\begin{equation}
\label{thm:3:eq:2}
\left\{\begin{aligned}
  &\text{Find $\Phi\in V:= \mathrm{H}^1(\Omega) \cap \mathrm{L}^2_0(\Omega),\, $ such that}\quad a(\Phi,\varphi) = l(\varphi), \qquad \forall \varphi \in V , \\
&\text{ with } \quad  a(\psi,\varphi) := \big(\bm{\nabla} \psi , \bm{\nabla}
\varphi\big)_{\mathrm{L}^2(\Omega)}, 
\quad
l(\varphi) := \int_{\Omega} \rho \, \varphi \, d\bm{x}\, \quad \forall (\psi,\varphi)
\in V\times V \,.
\end{aligned}\right.
\end{equation}

\paragraph{Discrete problem}
\label{sec:hsgpic_density_projection}
Problem~\eqref{thm:3:eq:2} is discretized by a Galerkin method on the truncated hierarchical space $V_{h,0}^{(\star)}(\Omega) := V_h^{(\star)}(\Omega) \cap \mathrm{L}^2_0(\Omega)$ introduced in \cref{def:Vstar}, for any of the choices $\star = 1, E, \infty$, and $\sigma=0,\ldots,|\log h|$. Unlike STD- and SGCT-PIC schemes, presented in~\cref{sec:stdpic,sec:sgct}, which discretize the regularized estimator of~\cref{eq:rho_N_hat}, the hierarchical sparse-grid method acts directly on the raw estimator introduced in~\cref{eq:9}, through its Galerkin projection onto the approximation space. 
Since the raw estimator is a distribution and not an $\mathrm{L}^2(\Omega)$ function, this projection is characterized through the duality pairing extending the $\mathrm{L}^2(\Omega)$ inner product to distributions: $\Pi_h^{(\star)} \rho_N$ is the unique element of $V_h^{(\star)}(\Omega)$ such that
\begin{equation}
\label{eq:density_projection}
\big(\Pi_h^{(\star)} \rho_N,\, v_h\big)_{\mathrm{L}^2(\Omega)}
= \big\langle \rho_N, v_h \big\rangle_{\mathrm{H}^{-1}_{\mathrm{mix}},\,
\mathrm{H}^1_{\mathrm{mix}}},
\qquad \forall v_h \in V_h^{(\star)}(\Omega),
\end{equation}
where $\langle \cdot,\cdot\rangle_{\mathrm{H}^{-1}_{\mathrm{mix}},\,
\mathrm{H}^1_{\mathrm{mix}}}$ denotes the duality pairing between
$\mathrm{H}^{-1}_{\mathrm{mix}}(\Omega)$ and its dual
$\mathrm{H}^1_{\mathrm{mix}}(\Omega)$, well-defined since the Dirac delta distribution belongs to $\mathrm{H}^{-1}_{\mathrm{mix}}(\Omega)$. This
pairing only requires $v_h \in \mathrm{H}^1_{\mathrm{mix}}(\Omega)$, and
is therefore meaningful for the B-splines of any degree $p \geq 1$ used as basis functions of the
approximation spaces defined in~\cref{def:Vstar}. Indeed, such splines belong
to $\mathrm{H}^p_{\mathrm{mix}}(\Omega) \subset \mathrm{H}^1_{\mathrm{mix}}
(\Omega)$ for every $p \geq 1$.

The construction \eqref{eq:density_projection} therefore plays, for the hierarchical method, the role that the convolution with the shape
function plays in \cref{eq:rho_N_hat} for the STD-
and SGCT-PIC schemes: both produce, from the raw density estimator, an
element of the relevant approximation space that can be used as the
source term of the discrete Poisson equation. 

Specifically, the discrete problem consists in finding $\Phi_{h,N}^{(\star)} \in V_{h,0}^{(\star)}(\Omega)$
such that
\begin{equation}
\label{eq:5}
a\big(\Phi_{h,N}^{(\star)}, \varphi_h\big) = l_{h}(\varphi_h),
\qquad \forall \varphi_h \in V_{h,0}^{(\star)}(\Omega),
\end{equation}
where the discrete linear form is built from the Galerkin projection defined above in~\cref{eq:density_projection},
\[
l_{h}(\varphi_h) := \big(\Pi_h^{(\star)} \rho_N, \varphi_h
\big)_{\mathrm{L}^2(\Omega)}.
\]
The choice of the approximation space governs the trade-off, as discussed in
Remark~\ref{rem:paradox_truncation}, between the accuracy of the discrete potential and the number of degrees of freedom involved in
\eqref{eq:5}.

\subsubsection{Resulting linear system}
\label{sec:hierarchical_linear_systems}
We now derive the linear system involved in the computation of the electric field at each step of the HSG-PIC algorithm. 

\paragraph{Resolution of Poisson equation}
Reordering the basis functions of the approximation space with vanishing mean into a single sequence 
\begin{align}
\label{eq:base_2}
\{\varphi_1,\ldots,\varphi_{N_{h,0}^{(\star)}}\}, \qquad N_{h,0}^{(\star)} := \dim (V_{h,0}^{(\star)}(\Omega)),
\end{align}
the hierarchical representation~\cref{eq:hier_decomp_general} of any function of this space can be rewritten 
\[u_h = \sum_{j=1}^{N_{h,0}^{(\star)}} \alpha_j \varphi_j,
\]
where the vector $\bm{\alpha}$ collects the hierarchical surpluses in this reordering. The discrete variational problem~\eqref{eq:5} is tested against every basis functions $\varphi_{i}$, yielding
\begin{align*}
a\big(\Phi_h^{(\star)},\, \varphi_{i}\big)
= l_h\big(\varphi_{i}\big)
= \big(\Pi_h^{(\star)} \rho_N,\, \varphi_{i}
\big)_{\mathrm{L}^2(\Omega)}, \qquad 1\leq i \leq N_{h,0}^{(\star)}.
\end{align*}
Decomposing both the electric potential and the Galerkin projection of the density into the base~\cref{eq:base_2}, we obtain the stiffness linear system
\begin{align}
\label{eq:stiff_system}
\mathbf{K} \bm{\beta} = \bf{b},
\end{align}
where the vector $\bm{\beta}$ collects the hierarchical surpluses of the electric potential, and where the stiffness matrix is defined by
\begin{align*}
\mathbf{K}_{i,j}=  \int_{\Omega} \bm{\nabla}\varphi_i \cdot \bm{\nabla}\varphi_j \, d\bm{x},  \qquad 1\leq i,j \leq  N_{h,0}^{(\star)}.
\end{align*}
The right-hand side $\bf{b}$ is assembled directly from the
particle positions, owing to the sifting property of the Dirac $\delta$ function under the duality pairing, as
\begin{align*}
\Fabrice{\big(\Pi_h^{(\star)} \rho_N,\, \varphi_{i}
\big)_{\mathrm{L}^2(\Omega)}=}\big\langle \rho_N,\varphi_i\big\rangle_{\mathrm{H}^{-1}_{\mathrm{mix}},\,\mathrm{H}^1_{\mathrm{mix}}} &= \sum_{s=1}^N w_s\, \big\langle \delta(\cdot-\bm{x}_s), \varphi_i\big\rangle_{\mathrm{H}^{-1}_{\mathrm{mix}},\, \mathrm{H}^1_{\mathrm{mix}}}\\
 &= \sum_{s=1}^N w_s\, \varphi_i(\bm{x}_s):= \mathbf{b}_i,  \Fabrice{\qquad 1\leq i \leq  N_{h,0}^{(\star)}}.
\end{align*}

The hierarchical structure of the basis ensures that
the stiffness matrix is globally sparse: a basis function associated with a level overlaps only with
neighboring basis functions of comparable or coarser levels,  so that the
number of nonzero entries per row of the matrix
remains bounded independently of the number of mesh nodes.

\paragraph{Evaluation of the electric field}
The electric field is evaluated at the particles' position
by direct differentiation of the discrete potential. It reads
\begin{equation}
\label{eq:field_eval}
\bm{E}_{h,N}^{(\star)}(\bm{x}_s)
= -\bm{\nabla}\Phi_{h,N}^{(\star)}(\bm{x}_s)
= -\sum_{\bm{\ell}\in\mathscr{H}_h^{(\star)}}
\sum_{\bm{j}\in J_{h_{\bm{\ell}}}}
\beta_{\bm{\ell},\bm{j}}\,
\bm{\nabla}\varphi_{h_{\bm{\ell}},\bm{j}}^p(\bm{x}_s),
\end{equation}
and does not require any additional linear solve: the gradient is computed by
differentiating each hierarchical basis function analytically and
summing its contribution, weighted by the surplus obtained in~\eqref{eq:stiff_system}. This choice keeps the electric field approximation exactly
equal to the gradient of the computed potential, with no further
approximation, and is consistent with the error estimates
of~\cite{deluzet26}, which bound the quantity $\bias[\bm{\nabla}\Phi_{h,N}^{(\star)}]$
directly rather than that of a separately projected field.

Nonetheless, for linear B-splines, the gradient of the basis functions is piecewise constant and discontinuous across cell
interfaces, so that \cref{eq:field_eval} is not defined when
the particle's position coincides exactly with a mesh node; in
practice this is a measure-zero event that does not affect the particle
dynamics. For $p\geq 2$, the gradient of the basis functions is
continuous and \cref{eq:field_eval} is well-defined everywhere on
the domain.

{\cleparagraph
\subsection{Locally adaptive mesh refinement}
The objective of mesh refinement is to improve the approximation of the solution while keeping the dimension of the approximation space as small as possible. Ideally, only those hierarchical contributions that significantly reduce the approximation error should be added. \Fabrice{Since our primary goal is to improve the approximation of functions featuring large mixed derivatives, we enrich the sparse-grid space by reintroducing selected hierarchical subspaces. These subspaces belong to the full-grid space, but are originally discarded by the sparse-grid truncation that defines $V^{(1)}_h$.}

The hierarchical surpluses measure the contribution of each basis function to the approximation. For sufficiently smooth functions, they decay rapidly as the hierarchical level increases, as established in~\cref{lem:surplus_estimate,eq:surplus_bound_r}. This decay justifies both the static truncation criterion introduced in~\cref{def:Vstar} and the use of hierarchical surpluses as \emph{a posteriori} local error indicators for adaptive refinement.

Applying H\"older's inequality to the closed-form expression of the hierarchical surplus~\cref{eq:surplus_closed_form} yields
\begin{align*}
|\alpha_{\bm{\ell},\bm{j}}|
\lesssim
h_{\ell_1}\cdots h_{\ell_d}
\left\|
D^{\bm{2}}u
\right\|_{\mathrm{L}^1(\supp(\varphi_{\bm{\ell},\bm{j}}))},
\end{align*}
for linear B-spline approximation functions.
Hence, a large hierarchical surplus indicates the presence of large second-order derivatives of the solution over the support of the corresponding basis function. 
Motivated by this estimate, we retain only the basis functions whose normalized hierarchical surpluses exceed a prescribed threshold. More precisely, for a tolerance $\epsilon>0$, we define the locally adaptive, or $\epsilon$-based, hierarchical index set by
\begin{align}
\label{eq:hier_index}
J_{h_{\bm{\ell}},\epsilon}
=
J_{h_{\bm{\ell}}}
\cap
\left\{
\bm{j}
~
\middle|
~
|\alpha_{\bm{\ell},\bm{j}}|\,
h_{\ell_1}^{-1}\cdots h_{\ell_d}^{-1}
>
\epsilon
\right\}.
\end{align}

Every approximation space introduced in~\cref{def:Vstar} is completely characterized by two index sets: a level index set, which specifies the retained hierarchical subspaces, and a hierarchical index set associated with each retained subspace, which specifies the selected basis functions. In all previously introduced approximation spaces, the hierarchical index set is complete, i.e., every basis function belonging to a retained subspace is included. We now introduce a new family of approximation spaces, parameterized by $\sigma$ and $\epsilon$, in which the generalized $\mathrm{L}^2$-based level index set is combined with the locally adaptive hierarchical index set defined in~\cref{eq:hier_index}. These locally adaptive sparse-grid approximation spaces are 
defined by extending~\cref{def:Vstar},~\cref{def:offset}, to the 
$\epsilon$-based hierarchical index set $J_{h_{\bm{\ell}},\epsilon}$.

\Fabrice{\begin{definition}[Locally adaptive sparse-grid approximation space]\label{def:Aspace} Given a target function and refinement tolerance $\epsilon>0$, let $J_{h_{\bm{\ell}},\epsilon}$ be the corresponding set of active spatial indices at level $\bm{\ell}$, determined according to \cref{eq:hier_index} by the hierarchical surplus criterion. The locally adaptive sparse-grid approximation space is then defined by
\begin{equation*}
V_{h,\sigma,\epsilon}^{(1)}(\Omega)
:= \mathrm{span} \left\{ \varphi^p_{h_{\bm{\ell}},\bm{j}} \;\middle|\; \bm{\ell} \in \mathscr{H}_{h,\sigma}^{(1)}, \quad \bm{j}\in J_{h_{\bm{\ell}},\epsilon} \right\}.
\end{equation*}
\end{definition}}

Constructing this approximation space requires exploring the hierarchical subspaces of the generalized $\mathrm{L}^2$-based approximation space that are not contained in the original sparse-grid approximation space and selecting the basis functions whose hierarchical surpluses satisfy~\cref{eq:hier_index}.

\paragraph{Galerkin projection of the density estimator}
The adaptive criterion is based on the hierarchical surpluses of the particle density. In practice, these surpluses are obtained from the Galerkin projection of the density estimator onto the generalized $\mathrm{L}^2$-based approximation space.


Expanding the Galerkin projection in the hierarchical basis \Fabrice{of the generalized space analogously to~\cref{eq:base_2}} and testing~\cref{eq:density_projection} against each basis function $\varphi_i$ \Fabrice{for $1 \leq i \leq N_{h,\sigma}^{(1)} := \dim\big(V_{h,\sigma}^{(1)}\big)$} yields the mass system

\begin{align}
\label{eq:mass_system}
\mathbf{M}\bm{\alpha}
=
\bm{b}, \qquad \mathbf{M}_{i,j}
=
\int_{\Omega}
\varphi_i\varphi_j\,d\bm{x},
\quad
b_i
=
\sum_{s=1}^N
w_s\,\varphi_i(\bm{x}_s),
\quad
1\le i,j\le N_{h,\sigma}^{(1)}.
\end{align}

Unlike standard finite-element or finite-difference PIC schemes, where the charge density is represented in a nodal basis, the value of the projected density cannot be recovered directly from a single coefficient. Indeed, the right-hand side of~\cref{eq:mass_system} only provides the quantities $\langle\rho_N,\varphi_i\rangle$, corresponding to the action of the particle density on the basis functions, rather than the hierarchical surpluses themselves. Since the hierarchical basis is not orthogonal, the mass matrix is not diagonal, and recovering the hierarchical surpluses therefore requires solving the linear system~\cref{eq:mass_system}.

Owing to the local support of the hierarchical basis functions, the mass matrix is sparse, making~\cref{eq:mass_system} well suited to efficient sparse linear solvers. This system is assembled and solved once per time step to obtain the hierarchical surpluses of the particle density. However, solving~\cref{eq:mass_system} over the entire generalized $\mathrm{L}^2$-based approximation space is computationally expensive because this space is significantly larger than the final locally adaptive approximation space. Most of the associated basis functions are eventually discarded by the adaptive criterion. The objective is therefore to identify the relevant basis functions without solving the global system over the complete generalized approximation space.

\paragraph{Strategy: incremental exploration via Schur complement}
To this end, we explore the generalized $\mathrm{L}^2$-based approximation space incrementally, one hierarchical subspace at a time. After each exploration step, the basis functions satisfying the admissibility criterion are incorporated into the current approximation space, while the remaining ones are discarded. When exploring a new hierarchical subspace, the unknowns associated with the current approximation space are eliminated through a Schur complement reduction, so that the resulting linear system only involves the unknowns associated with the newly explored subspace. The cost of this reduction is the Cholesky factorization of the mass matrix associated with the current approximation space, whose dimension remains moderate throughout the refinement procedure.

In the initialization phase, we introduce a coarse $\mathrm{L}^2$-based approximation space, denoted by $V_{h^0}^{(1)}$, associated with a coarse mesh size $h^0=2^{-n_0} > h=2^{-n}$, chosen such that the computational cost of solving the corresponding linear system is negligible. We assemble the associated coarse mass matrix and compute its Cholesky factorization.
The initial hierarchical surpluses $\bm{\alpha}_0$ of the charge density are computed by solving the coarse mass system via forward and backward substitutions: $\mathbf{L}_0 \mathbf{L}_0^T \bm{\alpha}_0 = \mathbf{b}_0$, where $\mathbf{L}_0 $ is the Cholesky factor of the coarse mass matrix.
The active index set is then initialized by retaining only the indices whose surpluses satisfy the thresholding criterion,
\begin{equation}
\label{eq:init_index_set}
\mathcal{I}_{0,\epsilon} := \left\{(\bm{\ell},\bm{j})\;  \middle| \; \bm{\ell} \in \mathscr{H}_{h^0}^{(1)}, \, \bm{j} \in {J}_{h_{\bm{\ell}}} \; \text{ with } |(\alpha_0)_{\bm{\ell},\bm{j}}|\,
h_{\ell_1}^{-1}\cdots h_{\ell_d}^{-1} >\epsilon\right\},
\end{equation}
which defines the initial locally adaptive space
\begin{equation}
  V_{h^0,0,\epsilon}^{(1)} = \mathrm{span}\big\{ \varphi^p_{h_{\bm{\ell}},\bm{j}} \;\big|\; (\bm{\ell}, \bm{j}) \in \mathcal{I}_{0,\epsilon} \big\}.
\end{equation} 
The remaining hierarchical subspaces belonging to the full $\mathrm{L}^2$-based space are subsequently explored incrementally, diagonal by diagonal along the hierarchy of level sums. For $k = k_0, \dots, \sigma$, with $k_0 = |\log h^0| - |\log h|$, the $k$-th diagonal set of level multi-indices is defined by
\begin{equation}
\label{eq:diagonal_def}
\mathcal{D}_k := \left\{ \bm{\ell} \in \mathbb{N}^d \;\middle|\; | \; \bm{\ell}|_1 = |\log h| + k \right\}.
\end{equation}

At the beginning of step $k$, we assemble the mass matrix $\mathbf{M}_{AA}$ corresponding to the current adaptive space $V_{h,k,\epsilon}^{(1)}$ and compute its Cholesky factor once:
\begin{equation}
\label{eq:current_cholesky}
\mathbf{M}_{AA} = \mathbf{L}_{AA} \mathbf{L}_{AA}^T.
\end{equation}
Each candidate hierarchical subspace $W_{\bm{\ell}}$ for $\bm{\ell} \in \mathcal{D}_k$ is then explored independently by considering the augmented space
\begin{equation}
\label{eq:augmented_space}
\widetilde{V} = V_{h,k,\epsilon}^{(1)} \oplus W_{\bm{\ell}}.
\end{equation}

Partitioning the unknowns into those associated with the current active space $V_{h,k,\epsilon}^{(1)}$, these unkonws being denoted $A$, and those associated with the candidate subspace $W_{\bm{\ell}}$ and denoted $C^{\bm{\ell}}$, the augmented Galerkin mass system reads
\begin{equation}
\label{eq:block_system}
\begin{pmatrix}
\mathbf{M}_{AA} & \mathbf{M}_{AC^{\bm{\ell}}} \\
\mathbf{M}_{AC^{\bm{\ell}}}^T & \mathbf{M}_{C^{\bm{\ell}} C^{\bm{\ell}}}
\end{pmatrix}
\begin{pmatrix}
\bm{\alpha}_A \\
\bm{\alpha}_{C^{\bm{\ell}}}
\end{pmatrix}
=
\begin{pmatrix}
\mathbf{b}_A \\
\mathbf{b}_{C^{\bm{\ell}}}
\end{pmatrix}.
\end{equation}

Rather than solving system~\eqref{eq:block_system} from scratch, we eliminate the active unknowns $\bm{\alpha}_A$ using the Schur complement of $\mathbf{M}_{AA}$. Taking advantage of the precomputed Cholesky factor $\mathbf{L}_{AA}$ from~\eqref{eq:current_cholesky}, we evaluate the intermediate quantities
\begin{equation}
\label{eq:schur_intermediates}
\mathbf{Y}_{AC^{\bm{\ell}}} = \mathbf{L}_{AA}^{-T} \left( \mathbf{L}_{AA}^{-1} \mathbf{M}_{AC^{\bm{\ell}}} \right) \quad \text{and} \quad \widehat{\mathbf{b}}_A = \mathbf{L}_{AA}^{-T} \left( \mathbf{L}_{AA}^{-1} \mathbf{b}_A \right),
\end{equation}
via triangular solves. The reduced system for the candidate surpluses $\bm{\alpha}_C^{\bm{\ell}}$ then simplifies to the Schur complement system
\begin{equation}
\label{eq:schur_system}
\left\{\begin{aligned}
  &\mathbf{S}_{C^{\bm{\ell}} C^{\bm{\ell}}} \bm{\alpha}_{C^{\bm{\ell}}} = \widetilde{\mathbf{b}}_{C^{\bm{\ell}}},\\
  & \mathbf{S}_{C^{\bm{\ell}} C^{\bm{\ell}}} = \mathbf{M}_{C^{\bm{\ell}}C^{\bm{\ell}}} - \mathbf{M}_{A C^{\bm{\ell}}}^T \mathbf{Y}_{AC^{\bm{\ell}}}, \qquad \widetilde{\mathbf{b}}_{C^{\bm{\ell}}} = \mathbf{b}_{C^{\bm{\ell}}} - \mathbf{M}_{A C^{\bm{\ell}}}^T \widehat{\mathbf{b}}_A.
\end{aligned}\right.
\end{equation}
Solving the small, dense system~\eqref{eq:schur_system} directly yields the hierarchical surpluses $\bm{\alpha}_{C^{\bm{\ell}}}$ associated with $W_{\bm{\ell}}$ at a cost drastically lower than solving a full non-adaptive system. 

Once all candidate subspaces $W_{\bm{\ell}}$ of diagonal $\mathcal{D}_k$ have been evaluated, we select the newly admissible indices
\begin{equation}
\label{eq:admissible_update}
\Delta \mathcal{I}_k := \left\{ (\bm{\ell}, \bm{j}) \;\middle|\; \bm{\ell} \in \mathcal{D}_k, \; \bm{j} \in J_{h_{\bm{\ell}}}, \; |(\alpha_{C^{\bm{\ell}}})_{\bm{\ell},\bm{j}}|\,
h_{\ell_1}^{-1}\cdots h_{\ell_d}^{-1} >\epsilon \right\}.
\end{equation}
The active index set and the adaptive approximation space for the next diagonal step $k+1$ are updated according to
\begin{equation}
\label{eq:space_update}
\mathcal{I}_{k+1,\epsilon} = \mathcal{I}_{k,\epsilon} \cup \Delta \mathcal{I}_k, \qquad V_{h,k+1,\epsilon}^{(1)} = \mathrm{span}\left\{ \varphi^p_{h_{\bm{\ell}},\bm{j}} \;\middle|\; (\bm{\ell}, \bm{j}) \in \mathcal{I}_{k+1,\epsilon} \right\}.
\end{equation}
This incremental procedure is repeated until all diagonals up to level $\sigma$ have been explored.

\section{Numerical results}
\label{sec:num}
\subsection{General settings}
We restrict our numerical study to the two-dimensional case. We recall that $h$ denotes the mesh size, $\Delta t$ the time step and $p_q=q$, where $q=1,3$, the spline degree. Given an approximation $u_h^\kappa$ of a function $u(t)$ at time step $\kappa$, the relative $\mathrm{L}^q$-norm of the error is given by
\[
\frac{\|u_h^\kappa - u(t)\|_{\mathrm{L}^q}}{\|u(t)\|_{\mathrm{L}^q}}, 
\]
where the integrals are decomposed onto the finest mesh as
\begin{align*}
\|u(t)\|_{\mathrm{L}^q}^q &=  \sum_{i_1,i_2} \int_{i_1 h}^{(i_1+\frac{1}{2}) h} \int_{i_2 h}^{(i_2+\frac{1}{2}) h} | u(t)|^q \, dx\, dy +  \int_{(i_1+\frac{1}{2}) h}^{(i_1+1) h} \int_{(i_2+\frac{1}{2}) h}^{(i_2+1) h} | u(t)|^q \, dx\, dy,
\end{align*}
and  approximated using Gauss-Legendre quadrature. The number of quadrature points $n_q$ is chosen according to the spline degree $p$ as $n_q = p + 1$. 

For our numerical investigation, we select two representative test cases: a solution constructed using the method of manufactured solutions~\cite{tranquilli22}; and the diocotron instability~\cite{driscoll90}. We compare different PIC methods, for which we introduce shortcut notations: STD and SGCT for the STD- and SGCT-PIC methods; HSG-($\star$) for the HSG-PIC method with the full-grid, $\mathrm{L}^2$-based, $\mathrm{H}^1$-based approximation spaces; and HSG-A$(\sigma,\epsilon)$ for the locally adaptive HSG-PIC method. We may also refer only to $(\star)$ or A-$(\sigma,\epsilon)$ when there is no ambiguity.

\paragraph{Manufactured solutions}
Both the electrons and ions are modeled in the phase space, and the weights of the particles are evolved throughout the simulation. Specifically, the weights of the particles $w_s$ evolve as
\[
\frac{dw_s}{dt}=\frac{S_f(\bm{x}_s(t),\bm{v}_s(t),t)}{\rho^0(\bm{x}_s(0),\bm{v}_s(0))},\quad S_{f_s} = \frac{\partial f^s}{\partial t}+\bm{v}\cdot \bm{\nabla}_{\bm{x}}f^s + \bm{E}\cdot  \bm{\nabla}_{\bm{v}}f^s ,
\]
where $\rho^0$ is the initial charge density, $S_f$ is a forcing term and $s$ denotes the species of the particles, i.e., electrons or ions. An analytical solution can be constructed for this problem. We choose a solution constructed from a distribution of electrons and ions given by
\begin{align*}
f^e(\bm{x},\bm{v})&:= \frac{f^0_{\bm{v}}(\bm{v})}{L^2}\left(1-\sin(\pi t)\sin\left(\frac{2\pi x}{L}\right)\sin\left(\frac{2\pi y}{L}\right)\right), \\
f^i(\bm{x},\bm{v})&:= \frac{f^0_{\bm{v}}(\bm{v})}{L^2},  \quad \text{and}~f^0_{\bm{v}}(\bm{v}):=\frac{4}{\pi}v_1^2v_2^2\exp\left(-v_1^2-v_2^2\right).
\end{align*}
The distributions are defined in the spatial periodic domain $\Omega:=(0,L)^2$, with $L=60$, and the velocity domain $\Omega_{\bm{v}}:=\mathbb{R}^2$.
The charge of electrons and ions are related as $q_i=-q_e=1$.
\paragraph{Diocotron instability}
We consider a radially symmetric Gaussian ring as the initial spatial distribution of electrons, defined by
\begin{align*}
f^0_{\bm{x}}(\bm{x})=\gamma \exp\left(-\frac{(\|\bm{x}-\frac{L}{2}\|_2-\frac{L}{4})^2}{2(\beta L)^2} \right), \quad \text{with}~\gamma ~\text{such that} \int_{\Omega}f^0_{\bm{x}} dxdy =1,
\end{align*}
and $\beta = 0.03$, along with a Maxwellian velocity distribution given by
\begin{align*}
f^0_{\bm{v}}(\bm{v}) = \left(\frac{1}{\sqrt{\pi} v_T}\right)^3 \exp\Big(-\frac{\|\bm{v}\|_2^2}{v_T^2}\Big), \quad v_T = \sqrt{2 T_e q_e / m_e}.
\end{align*}
Here, $\|\cdot\|_2$ denotes the Euclidean norm. The small value of $\beta$ ensures a narrow Gaussian ring with strong spatial variations that are not aligned with the axes, a scenario known to be very challenging for sparse-grid methods.
The domain size is set to $L=60$, and an external uniform magnetic field is applied along the $z$-axis, $\bm{B}_0 = (0,0,B_z)$, with $B_z=15$. The magnetic field is sufficiently strong such that the electron dynamics is dominated by advection in the self-consistent $\bm{E}\times \bm{B}_0$ field~\cite{driscoll90}.  

The magnetic field induces an instability that deforms the initially radially symmetric electron density, forming vortices. Such dynamics are particularly demanding for sparse-grid approximations~\cite{ricketson17,muralikrishnan21,deluzet22,deluzet22-1}, as strong gradients in mixed directions (xy-plane) develop, and the solution exhibits fine-scale, non-aligned structures. The strong magnetic field also imposes a constraint on the time step: it must be smaller than the electron gyroperiod, with cyclotron frequency $\Omega_c = B_z$. We set $\Delta t = 0.02$ to satisfy $\Omega_c \Delta t \leq 1$.

\Fabrice{The Debye length resolution requires $h\leq 1/L$ to aviod grid heating \cite{langdon70}, a condition met in all computations.}

Preliminary numerical experiments, omitted here for brevity, have shown that the method does not exhibit grid heating instability as long as at least few nodes from the hierarchical levels that verify the Debye length constraint are selected.  We ensure this condition in the remaining of the paper. 

\subsection{Convergence and accuracy}
In this section, we aim to evaluate the accuracy of the different approximation spaces and the advantages of the locally adaptive mesh refinement strategy.  We focus on the convergence according to the spatial discretization and thus fix the time discretization to a step $\Delta t =0.02$, small enough so that the total error is dominated by either the grid-based or the statistical error. 

In~\cite{tranquilli22}, the authors proposed assessing the accuracy of PIC methods by examining errors in grid quantities, such as moments of the distribution function or the fields. Following this approach, we restrict our analysis to the approximation error of the charge density.

\subsubsection{Grid-based error}
\label{sec:3.1.1}
We first investigate the grid-based component of the error. Since the total error is typically dominated by the statistical contribution, observing the theoretical convergence rates would require a prohibitively large number of particles. To isolate the grid error, we remove the statistical noise by replacing the particle-based estimator with the exact continuous charge density. Thus, within the scope of~\cref{sec:3.1.1}, the discrete variational formulation~\eqref{eq:5} becomes
\[
\text{Find } \Phi_{h,N}^{(\star)} \in V_{h,0}^{(\star)}(\Omega) \quad \text{such that} \quad 
a(\Phi_{h,N}^{(\star)}, v_h) = (\Pi_{h}^{(\star)} \rho , v_h)_{\mathrm{L}^2(\Omega)}, \quad \forall v_h \in V_{h,0}^{(\star)}(\Omega),
\]
where the right-hand side is evaluated using the exact density rather than the Monte Carlo estimator.

\paragraph{Accuracy of the approximations spaces} 
We first investigate the accuracy provided by the different approximation spaces: the $\mathrm{L}^2$-, $\mathrm{H}^1$-based sparse-grid and the full-grid approximation spaces. We consider the solution constructed with the method of manufactured solution and we analyze the error at time $t=0.5$.

The $\mathrm{L}^2$-norm error of the charge density approximation is plotted on the left panel of \cref{fig:0} as a function of the mesh size for different spline degrees and approximations spaces. The number of mesh nodes, i.e., the dimension of the approximation spaces, is indicated on each curves.

\begin{center}
  \resizebox{\textwidth}{!}{
\begingroup
  \makeatletter
  \providecommand\color[2][]{%
    \GenericError{(gnuplot) \space\space\space\@spaces}{%
      Package color not loaded in conjunction with
      terminal option `colourtext'%
    }{See the gnuplot documentation for explanation.%
    }{Either use 'blacktext' in gnuplot or load the package
      color.sty in LaTeX.}%
    \renewcommand\color[2][]{}%
  }%
  \providecommand\includegraphics[2][]{%
    \GenericError{(gnuplot) \space\space\space\@spaces}{%
      Package graphicx or graphics not loaded%
    }{See the gnuplot documentation for explanation.%
    }{The gnuplot epslatex terminal needs graphicx.sty or graphics.sty.}%
    \renewcommand\includegraphics[2][]{}%
  }%
  \providecommand\rotatebox[2]{#2}%
  \@ifundefined{ifGPcolor}{%
    \newif\ifGPcolor
    \GPcolortrue
  }{}%
  \@ifundefined{ifGPblacktext}{%
    \newif\ifGPblacktext
    \GPblacktexttrue
  }{}%
  \let\gplgaddtomacro\g@addto@macro
  \gdef\gplbacktext{}%
  \gdef\gplfronttext{}%
  \makeatother
  \ifGPblacktext
    \def\colorrgb#1{}%
    \def\colorgray#1{}%
  \else
    \ifGPcolor
      \def\colorrgb#1{\color[rgb]{#1}}%
      \def\colorgray#1{\color[gray]{#1}}%
      \expandafter\def\csname LTw\endcsname{\color{white}}%
      \expandafter\def\csname LTb\endcsname{\color{black}}%
      \expandafter\def\csname LTa\endcsname{\color{black}}%
      \expandafter\def\csname LT0\endcsname{\color[rgb]{1,0,0}}%
      \expandafter\def\csname LT1\endcsname{\color[rgb]{0,1,0}}%
      \expandafter\def\csname LT2\endcsname{\color[rgb]{0,0,1}}%
      \expandafter\def\csname LT3\endcsname{\color[rgb]{1,0,1}}%
      \expandafter\def\csname LT4\endcsname{\color[rgb]{0,1,1}}%
      \expandafter\def\csname LT5\endcsname{\color[rgb]{1,1,0}}%
      \expandafter\def\csname LT6\endcsname{\color[rgb]{0,0,0}}%
      \expandafter\def\csname LT7\endcsname{\color[rgb]{1,0.3,0}}%
      \expandafter\def\csname LT8\endcsname{\color[rgb]{0.5,0.5,0.5}}%
    \else
      \def\colorrgb#1{\color{black}}%
      \def\colorgray#1{\color[gray]{#1}}%
      \expandafter\def\csname LTw\endcsname{\color{white}}%
      \expandafter\def\csname LTb\endcsname{\color{black}}%
      \expandafter\def\csname LTa\endcsname{\color{black}}%
      \expandafter\def\csname LT0\endcsname{\color{black}}%
      \expandafter\def\csname LT1\endcsname{\color{black}}%
      \expandafter\def\csname LT2\endcsname{\color{black}}%
      \expandafter\def\csname LT3\endcsname{\color{black}}%
      \expandafter\def\csname LT4\endcsname{\color{black}}%
      \expandafter\def\csname LT5\endcsname{\color{black}}%
      \expandafter\def\csname LT6\endcsname{\color{black}}%
      \expandafter\def\csname LT7\endcsname{\color{black}}%
      \expandafter\def\csname LT8\endcsname{\color{black}}%
    \fi
  \fi
    \setlength{\unitlength}{0.0500bp}%
    \ifx\gptboxheight\undefined%
      \newlength{\gptboxheight}%
      \newlength{\gptboxwidth}%
      \newsavebox{\gptboxtext}%
    \fi%
    \setlength{\fboxrule}{0.5pt}%
    \setlength{\fboxsep}{1pt}%
    \definecolor{tbcol}{rgb}{1,1,1}%
\begin{picture}(3740.00,3440.00)%
    \gplgaddtomacro\gplbacktext{%
      \csname LTb\endcsname
      \put(689,807){\makebox(0,0)[r]{\strut{}\small$10^{-10}$}}%
      \csname LTb\endcsname
      \put(689,1303){\makebox(0,0)[r]{\strut{}\small$10^{-8}$}}%
      \csname LTb\endcsname
      \put(689,1799){\makebox(0,0)[r]{\strut{}\small$10^{-6}$}}%
      \csname LTb\endcsname
      \put(689,2295){\makebox(0,0)[r]{\strut{}\small$10^{-4}$}}%
      \csname LTb\endcsname
      \put(689,2790){\makebox(0,0)[r]{\strut{}\small$10^{-2}$}}%
      \csname LTb\endcsname
      \put(689,3286){\makebox(0,0)[r]{\strut{}\small$10^{0}$}}%
      \csname LTb\endcsname
      \put(1298,426){\makebox(0,0){\strut{}$10^{-2}$}}%
      \csname LTb\endcsname
      \put(2064,426){\makebox(0,0){\strut{}$10^{-1}$}}%
      \csname LTb\endcsname
      \put(2829,426){\makebox(0,0){\strut{}$\small 10^{0}$}}%
    }%
    \gplgaddtomacro\gplfronttext{%
      \csname LTb\endcsname
      \put(2464,2033){\makebox(0,0)[r]{\strut{}\scriptsize$h^{2}$}}%
      \csname LTb\endcsname
      \put(2464,1913){\makebox(0,0)[r]{\strut{}\scriptsize$h^{4}$}}%
      \csname LTb\endcsname
      \put(2464,1794){\makebox(0,0)[r]{\strut{}\scriptsize$h^{6}$}}%
      \csname LTb\endcsname
      \put(2464,1674){\makebox(0,0)[r]{\strut{}\scriptsize$(\mathrm{E}),p_1$}}%
      \csname LTb\endcsname
      \put(2464,1554){\makebox(0,0)[r]{\strut{}\scriptsize$(1),p_1$}}%
      \csname LTb\endcsname
      \put(2464,1434){\makebox(0,0)[r]{\strut{}\scriptsize$(\infty),p_1$}}%
      \colorrgb{0.40,0.70,1.00}
      \put(1908,3299){\makebox(0,0){\strut{}\scalebox{0.9}{\tiny \color[HTML]{66b3ff}  36}}}%
      \colorrgb{0.40,0.70,1.00}
      \put(1677,3108){\makebox(0,0){\strut{}\scalebox{0.9}{\tiny \color[HTML]{66b3ff}  80}}}%
      \colorrgb{0.40,0.70,1.00}
      \put(1447,3108){\makebox(0,0){\strut{}\scalebox{0.9}{\tiny \color[HTML]{66b3ff}  160}}}%
      \colorrgb{0.40,0.70,1.00}
      \put(1216,2946){\makebox(0,0){\strut{}\scalebox{0.9}{\tiny \color[HTML]{66b3ff}  336}}}%
      \colorrgb{0.40,0.70,1.00}
      \put(986,2795){\makebox(0,0){\strut{}\scalebox{0.9}{\tiny \color[HTML]{66b3ff}  704}}}%
      \colorrgb{0.00,0.35,1.00}
      \put(2138,3299){\makebox(0,0){\strut{}\scalebox{0.9}{\tiny \color[HTML]{0059ff}  20}}}%
      \colorrgb{0.00,0.35,1.00}
      \put(1908,3108){\makebox(0,0){\strut{}\scalebox{0.9}{\tiny \color[HTML]{0059ff}  48}}}%
      \colorrgb{0.00,0.35,1.00}
      \put(1677,2946){\makebox(0,0){\strut{}\scalebox{0.9}{\tiny \color[HTML]{0059ff}  112}}}%
      \colorrgb{0.00,0.35,1.00}
      \put(1447,2795){\makebox(0,0){\strut{}\scalebox{0.9}{\tiny \color[HTML]{0059ff}  256}}}%
      \colorrgb{0.00,0.35,1.00}
      \put(1216,2647){\makebox(0,0){\strut{}\scalebox{0.9}{\tiny \color[HTML]{0059ff}  576}}}%
      \colorrgb{0.00,0.35,1.00}
      \put(986,2500){\makebox(0,0){\strut{}\scalebox{0.9}{\tiny \color[HTML]{0059ff}  1280}}}%
      \colorrgb{0.00,0.17,0.50}
      \put(2599,3299){\makebox(0,0){\strut{}\scalebox{0.9}{\tiny \color[HTML]{002b7f}  3}}}%
      \colorrgb{0.00,0.17,0.50}
      \put(2368,3106){\makebox(0,0){\strut{}\scalebox{0.9}{\tiny \color[HTML]{002b7f}  16}}}%
      \colorrgb{0.00,0.17,0.50}
      \put(2138,2935){\makebox(0,0){\strut{}\scalebox{0.9}{\tiny \color[HTML]{002b7f}  64}}}%
      \colorrgb{0.00,0.17,0.50}
      \put(1908,2780){\makebox(0,0){\strut{}\scalebox{0.9}{\tiny \color[HTML]{002b7f}  256}}}%
      \colorrgb{0.00,0.17,0.50}
      \put(1677,2629){\makebox(0,0){\strut{}\scalebox{0.9}{\tiny \color[HTML]{002b7f}  1024}}}%
      \colorrgb{0.00,0.17,0.50}
      \put(1447,2479){\makebox(0,0){\strut{}\scalebox{0.9}{\tiny \color[HTML]{002b7f}  4096}}}%
      \colorrgb{0.00,0.17,0.50}
      \put(1216,2330){\makebox(0,0){\strut{}\scalebox{0.9}{\tiny \color[HTML]{002b7f}  16384}}}%
      \csname LTb\endcsname
      \put(2464,1314){\makebox(0,0)[r]{\strut{}\scriptsize$(\mathrm{E}),p_3$}}%
      \csname LTb\endcsname
      \put(2464,1194){\makebox(0,0)[r]{\strut{}\scriptsize$(1),p_3$}}%
      \csname LTb\endcsname
      \put(2464,1074){\makebox(0,0)[r]{\strut{}\scriptsize$(\infty),p_3$}}%
      \colorrgb{0.00,0.64,0.24}
      \put(1677,2376){\makebox(0,0){\strut{}\scalebox{0.9}{\tiny \color[HTML]{00a33c}  112}}}%
      \colorrgb{0.00,0.64,0.24}
      \put(1447,2057){\makebox(0,0){\strut{}\scalebox{0.9}{\tiny \color[HTML]{00a33c}  256}}}%
      \colorrgb{0.00,0.64,0.24}
      \put(1216,1753){\makebox(0,0){\strut{}\scalebox{0.9}{\tiny \color[HTML]{00a33c}  576}}}%
      \colorrgb{0.00,0.64,0.24}
      \put(986,1453){\makebox(0,0){\strut{}\scalebox{0.9}{\tiny \color[HTML]{00a33c}  1280}}}%
      \colorrgb{0.40,1.00,0.60}
      \put(1908,3299){\makebox(0,0){\strut{}\scalebox{0.9}{\tiny \color[HTML]{66ff99}  36}}}%
      \colorrgb{0.40,1.00,0.60}
      \put(1677,2865){\makebox(0,0){\strut{}\scalebox{0.9}{\tiny \color[HTML]{66ff99}  80}}}%
      \colorrgb{0.40,1.00,0.60}
      \put(1447,2865){\makebox(0,0){\strut{}\scalebox{0.9}{\tiny \color[HTML]{66ff99}  160}}}%
      \colorrgb{0.40,1.00,0.60}
      \put(1216,2509){\makebox(0,0){\strut{}\scalebox{0.9}{\tiny \color[HTML]{66ff99}  336}}}%
      \colorrgb{0.40,1.00,0.60}
      \put(986,2191){\makebox(0,0){\strut{}\scalebox{0.9}{\tiny \color[HTML]{66ff99}  704}}}%
      \colorrgb{0.00,0.30,0.10}
      \put(2599,3299){\makebox(0,0){\strut{}\scalebox{0.9}{\tiny \color[HTML]{004d1a}  3}}}%
      \colorrgb{0.00,0.30,0.10}
      \put(2368,2865){\makebox(0,0){\strut{}\scalebox{0.9}{\tiny \color[HTML]{004d1a}  16}}}%
      \colorrgb{0.00,0.30,0.10}
      \put(2138,2506){\makebox(0,0){\strut{}\scalebox{0.9}{\tiny \color[HTML]{004d1a}  64}}}%
      \colorrgb{0.00,0.30,0.10}
      \put(1908,2188){\makebox(0,0){\strut{}\scalebox{0.9}{\tiny \color[HTML]{004d1a}  256}}}%
      \colorrgb{0.00,0.30,0.10}
      \put(1677,1884){\makebox(0,0){\strut{}\scalebox{0.9}{\tiny \color[HTML]{004d1a}  1024}}}%
      \colorrgb{0.00,0.30,0.10}
      \put(1447,1584){\makebox(0,0){\strut{}\scalebox{0.9}{\tiny \color[HTML]{004d1a}  4096}}}%
      \colorrgb{0.00,0.30,0.10}
      \put(1216,1285){\makebox(0,0){\strut{}\scalebox{0.9}{\tiny \color[HTML]{004d1a}  16384}}}%
      \csname LTb\endcsname
      \put(2464,954){\makebox(0,0)[r]{\strut{}\scriptsize$(\mathrm{E}),p_5$}}%
      \csname LTb\endcsname
      \put(2464,834){\makebox(0,0)[r]{\strut{}\scriptsize$(1),p_5$}}%
      \csname LTb\endcsname
      \put(2464,714){\makebox(0,0)[r]{\strut{}\scriptsize$(\infty),p_5$}}%
      \colorrgb{0.40,0.00,0.00}
      \put(2599,3299){\makebox(0,0){\strut{}\scalebox{0.9}{\tiny \color[HTML]{660000}  3}}}%
      \colorrgb{0.40,0.00,0.00}
      \put(2368,2627){\makebox(0,0){\strut{}\scalebox{0.9}{\tiny \color[HTML]{660000}  16}}}%
      \colorrgb{0.40,0.00,0.00}
      \put(2138,2082){\makebox(0,0){\strut{}\scalebox{0.9}{\tiny \color[HTML]{660000}  64}}}%
      \colorrgb{0.40,0.00,0.00}
      \put(1908,1598){\makebox(0,0){\strut{}\scalebox{0.9}{\tiny \color[HTML]{660000}  256}}}%
      \colorrgb{0.40,0.00,0.00}
      \put(1677,1198){\makebox(0,0){\strut{}\scalebox{0.9}{\tiny \color[HTML]{660000}  1024}}}%
      \colorrgb{1.00,0.00,0.00}
      \put(2138,3299){\makebox(0,0){\strut{}\scalebox{0.9}{\tiny \color[HTML]{ff0000}  20}}}%
      \colorrgb{1.00,0.00,0.00}
      \put(1908,2627){\makebox(0,0){\strut{}\scalebox{0.9}{\tiny \color[HTML]{ff0000}  48}}}%
      \colorrgb{1.00,0.00,0.00}
      \put(1677,2084){\makebox(0,0){\strut{}\scalebox{0.9}{\tiny \color[HTML]{ff0000}  112}}}%
      \colorrgb{1.00,0.00,0.00}
      \put(1447,1599){\makebox(0,0){\strut{}\scalebox{0.9}{\tiny \color[HTML]{ff0000}  256}}}%
      \colorrgb{1.00,0.00,0.00}
      \put(1216,1141){\makebox(0,0){\strut{}\scalebox{0.9}{\tiny \color[HTML]{ff0000}  576}}}%
      \colorrgb{1.00,0.60,0.60}
      \put(1908,3299){\makebox(0,0){\strut{}\scalebox{0.9}{\tiny \color[HTML]{ff9999}  36}}}%
      \colorrgb{1.00,0.60,0.60}
      \put(1677,2627){\makebox(0,0){\strut{}\scalebox{0.9}{\tiny \color[HTML]{ff9999}  80}}}%
      \colorrgb{1.00,0.60,0.60}
      \put(1447,2627){\makebox(0,0){\strut{}\scalebox{0.9}{\tiny \color[HTML]{ff9999}  160}}}%
      \colorrgb{1.00,0.60,0.60}
      \put(1216,2084){\makebox(0,0){\strut{}\scalebox{0.9}{\tiny \color[HTML]{ff9999}  336}}}%
      \colorrgb{1.00,0.60,0.60}
      \put(986,1599){\makebox(0,0){\strut{}\scalebox{0.9}{\tiny \color[HTML]{ff9999}  704}}}%
      \csname LTb\endcsname
      \put(48,1923){\rotatebox{-270.00}{\makebox(0,0){\strut{}$\|\Pi_{h}^{(\star)} \rho(t) - \rho(t))\|_{\mathrm{L}^2(\Omega)}$}}}%
      \csname LTb\endcsname
      \put(1796,93){\makebox(0,0){\strut{}Mesh size $h$}}%
    }%
    \gplbacktext
    \put(0,0){\includegraphics[width={187.00bp},height={172.00bp}]{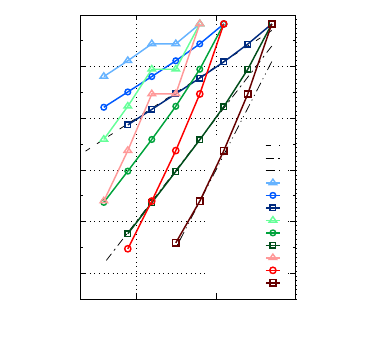}}%
    \gplfronttext
  \end{picture}%
\endgroup
 
\begingroup
  \makeatletter
  \providecommand\color[2][]{%
    \GenericError{(gnuplot) \space\space\space\@spaces}{%
      Package color not loaded in conjunction with
      terminal option `colourtext'%
    }{See the gnuplot documentation for explanation.%
    }{Either use 'blacktext' in gnuplot or load the package
      color.sty in LaTeX.}%
    \renewcommand\color[2][]{}%
  }%
  \providecommand\includegraphics[2][]{%
    \GenericError{(gnuplot) \space\space\space\@spaces}{%
      Package graphicx or graphics not loaded%
    }{See the gnuplot documentation for explanation.%
    }{The gnuplot epslatex terminal needs graphicx.sty or graphics.sty.}%
    \renewcommand\includegraphics[2][]{}%
  }%
  \providecommand\rotatebox[2]{#2}%
  \@ifundefined{ifGPcolor}{%
    \newif\ifGPcolor
    \GPcolortrue
  }{}%
  \@ifundefined{ifGPblacktext}{%
    \newif\ifGPblacktext
    \GPblacktexttrue
  }{}%
  \let\gplgaddtomacro\g@addto@macro
  \gdef\gplbacktext{}%
  \gdef\gplfronttext{}%
  \makeatother
  \ifGPblacktext
    \def\colorrgb#1{}%
    \def\colorgray#1{}%
  \else
    \ifGPcolor
      \def\colorrgb#1{\color[rgb]{#1}}%
      \def\colorgray#1{\color[gray]{#1}}%
      \expandafter\def\csname LTw\endcsname{\color{white}}%
      \expandafter\def\csname LTb\endcsname{\color{black}}%
      \expandafter\def\csname LTa\endcsname{\color{black}}%
      \expandafter\def\csname LT0\endcsname{\color[rgb]{1,0,0}}%
      \expandafter\def\csname LT1\endcsname{\color[rgb]{0,1,0}}%
      \expandafter\def\csname LT2\endcsname{\color[rgb]{0,0,1}}%
      \expandafter\def\csname LT3\endcsname{\color[rgb]{1,0,1}}%
      \expandafter\def\csname LT4\endcsname{\color[rgb]{0,1,1}}%
      \expandafter\def\csname LT5\endcsname{\color[rgb]{1,1,0}}%
      \expandafter\def\csname LT6\endcsname{\color[rgb]{0,0,0}}%
      \expandafter\def\csname LT7\endcsname{\color[rgb]{1,0.3,0}}%
      \expandafter\def\csname LT8\endcsname{\color[rgb]{0.5,0.5,0.5}}%
    \else
      \def\colorrgb#1{\color{black}}%
      \def\colorgray#1{\color[gray]{#1}}%
      \expandafter\def\csname LTw\endcsname{\color{white}}%
      \expandafter\def\csname LTb\endcsname{\color{black}}%
      \expandafter\def\csname LTa\endcsname{\color{black}}%
      \expandafter\def\csname LT0\endcsname{\color{black}}%
      \expandafter\def\csname LT1\endcsname{\color{black}}%
      \expandafter\def\csname LT2\endcsname{\color{black}}%
      \expandafter\def\csname LT3\endcsname{\color{black}}%
      \expandafter\def\csname LT4\endcsname{\color{black}}%
      \expandafter\def\csname LT5\endcsname{\color{black}}%
      \expandafter\def\csname LT6\endcsname{\color{black}}%
      \expandafter\def\csname LT7\endcsname{\color{black}}%
      \expandafter\def\csname LT8\endcsname{\color{black}}%
    \fi
  \fi
    \setlength{\unitlength}{0.0500bp}%
    \ifx\gptboxheight\undefined%
      \newlength{\gptboxheight}%
      \newlength{\gptboxwidth}%
      \newsavebox{\gptboxtext}%
    \fi%
    \setlength{\fboxrule}{0.5pt}%
    \setlength{\fboxsep}{1pt}%
    \definecolor{tbcol}{rgb}{1,1,1}%
\begin{picture}(3740.00,3440.00)%
    \gplgaddtomacro\gplbacktext{%
      \csname LTb\endcsname
      \put(689,559){\makebox(0,0)[r]{\strut{}\small$10^{-10}$}}%
      \csname LTb\endcsname
      \put(689,1105){\makebox(0,0)[r]{\strut{}\small$10^{-8}$}}%
      \csname LTb\endcsname
      \put(689,1650){\makebox(0,0)[r]{\strut{}\small$10^{-6}$}}%
      \csname LTb\endcsname
      \put(689,2195){\makebox(0,0)[r]{\strut{}\small$10^{-4}$}}%
      \csname LTb\endcsname
      \put(689,2741){\makebox(0,0)[r]{\strut{}\small$10^{-2}$}}%
      \csname LTb\endcsname
      \put(689,3286){\makebox(0,0)[r]{\strut{}\small$10^{0}$}}%
      \csname LTb\endcsname
      \put(1298,426){\makebox(0,0){\strut{}$10^{-2}$}}%
      \csname LTb\endcsname
      \put(2064,426){\makebox(0,0){\strut{}$10^{-1}$}}%
      \csname LTb\endcsname
      \put(2829,426){\makebox(0,0){\strut{}$\small 10^{0}$}}%
    }%
    \gplgaddtomacro\gplfronttext{%
      \csname LTb\endcsname
      \put(2464,2033){\makebox(0,0)[r]{\strut{}\scriptsize $h$}}%
      \csname LTb\endcsname
      \put(2464,1913){\makebox(0,0)[r]{\strut{}\scriptsize $h^3$}}%
      \csname LTb\endcsname
      \put(2464,1794){\makebox(0,0)[r]{\strut{}\scriptsize$h^{5}$}}%
      \csname LTb\endcsname
      \put(2464,1674){\makebox(0,0)[r]{\strut{}\scriptsize$(\mathrm{E}),p_1$}}%
      \csname LTb\endcsname
      \put(2464,1554){\makebox(0,0)[r]{\strut{}\scriptsize$(1),p_1$}}%
      \csname LTb\endcsname
      \put(2464,1434){\makebox(0,0)[r]{\strut{}\scriptsize$(\infty),p_1$}}%
      \colorrgb{0.40,0.70,1.00}
      \put(1908,3345){\makebox(0,0){\strut{}\scalebox{0.9}{\tiny \color[HTML]{66b3ff}  36}}}%
      \colorrgb{0.40,0.70,1.00}
      \put(1677,3250){\makebox(0,0){\strut{}\scalebox{0.9}{\tiny \color[HTML]{66b3ff}  80}}}%
      \colorrgb{0.40,0.70,1.00}
      \put(1447,3250){\makebox(0,0){\strut{}\scalebox{0.9}{\tiny \color[HTML]{66b3ff}  160}}}%
      \colorrgb{0.40,0.70,1.00}
      \put(1216,3171){\makebox(0,0){\strut{}\scalebox{0.9}{\tiny \color[HTML]{66b3ff}  336}}}%
      \colorrgb{0.40,0.70,1.00}
      \put(986,3090){\makebox(0,0){\strut{}\scalebox{0.9}{\tiny \color[HTML]{66b3ff}  704}}}%
      \colorrgb{0.00,0.35,1.00}
      \put(2138,3345){\makebox(0,0){\strut{}\scalebox{0.9}{\tiny \color[HTML]{0059ff}  20}}}%
      \colorrgb{0.00,0.35,1.00}
      \put(1908,3250){\makebox(0,0){\strut{}\scalebox{0.9}{\tiny \color[HTML]{0059ff}  48}}}%
      \colorrgb{0.00,0.35,1.00}
      \put(1677,3171){\makebox(0,0){\strut{}\scalebox{0.9}{\tiny \color[HTML]{0059ff}  112}}}%
      \colorrgb{0.00,0.35,1.00}
      \put(1447,3090){\makebox(0,0){\strut{}\scalebox{0.9}{\tiny \color[HTML]{0059ff}  256}}}%
      \colorrgb{0.00,0.35,1.00}
      \put(1216,3008){\makebox(0,0){\strut{}\scalebox{0.9}{\tiny \color[HTML]{0059ff}  576}}}%
      \colorrgb{0.00,0.35,1.00}
      \put(986,2926){\makebox(0,0){\strut{}\scalebox{0.9}{\tiny \color[HTML]{0059ff}  1280}}}%
      \colorrgb{0.00,0.17,0.50}
      \put(2599,3345){\makebox(0,0){\strut{}\scalebox{0.9}{\tiny \color[HTML]{002b7f}  3}}}%
      \colorrgb{0.00,0.17,0.50}
      \put(2368,3250){\makebox(0,0){\strut{}\scalebox{0.9}{\tiny \color[HTML]{002b7f}  16}}}%
      \colorrgb{0.00,0.17,0.50}
      \put(2138,3169){\makebox(0,0){\strut{}\scalebox{0.9}{\tiny \color[HTML]{002b7f}  64}}}%
      \colorrgb{0.00,0.17,0.50}
      \put(1908,3087){\makebox(0,0){\strut{}\scalebox{0.9}{\tiny \color[HTML]{002b7f}  256}}}%
      \colorrgb{0.00,0.17,0.50}
      \put(1677,3005){\makebox(0,0){\strut{}\scalebox{0.9}{\tiny \color[HTML]{002b7f}  1024}}}%
      \colorrgb{0.00,0.17,0.50}
      \put(1447,2923){\makebox(0,0){\strut{}\scalebox{0.9}{\tiny \color[HTML]{002b7f}  4096}}}%
      \colorrgb{0.00,0.17,0.50}
      \put(1216,2841){\makebox(0,0){\strut{}\scalebox{0.9}{\tiny \color[HTML]{002b7f}  16384}}}%
      \csname LTb\endcsname
      \put(2464,1314){\makebox(0,0)[r]{\strut{}\scriptsize$(\mathrm{E}),p_3$}}%
      \csname LTb\endcsname
      \put(2464,1194){\makebox(0,0)[r]{\strut{}\scriptsize$(1),p_3$}}%
      \csname LTb\endcsname
      \put(2464,1074){\makebox(0,0)[r]{\strut{}\scriptsize$(\infty),p_3$}}%
      \colorrgb{0.00,0.64,0.24}
      \put(1677,2535){\makebox(0,0){\strut{}\scalebox{0.9}{\tiny \color[HTML]{00a33c}  112}}}%
      \colorrgb{0.00,0.64,0.24}
      \put(1447,2275){\makebox(0,0){\strut{}\scalebox{0.9}{\tiny \color[HTML]{00a33c}  256}}}%
      \colorrgb{0.00,0.64,0.24}
      \put(1216,2025){\makebox(0,0){\strut{}\scalebox{0.9}{\tiny \color[HTML]{00a33c}  576}}}%
      \colorrgb{0.00,0.64,0.24}
      \put(986,1777){\makebox(0,0){\strut{}\scalebox{0.9}{\tiny \color[HTML]{00a33c}  1280}}}%
      \colorrgb{0.40,1.00,0.60}
      \put(2138,3345){\makebox(0,0){\strut{}\scalebox{0.9}{\tiny \color[HTML]{66ff99}  16}}}%
      \colorrgb{0.40,1.00,0.60}
      \put(1908,3345){\makebox(0,0){\strut{}\scalebox{0.9}{\tiny \color[HTML]{66ff99}  36}}}%
      \colorrgb{0.40,1.00,0.60}
      \put(1677,2964){\makebox(0,0){\strut{}\scalebox{0.9}{\tiny \color[HTML]{66ff99}  80}}}%
      \colorrgb{0.40,1.00,0.60}
      \put(1447,2964){\makebox(0,0){\strut{}\scalebox{0.9}{\tiny \color[HTML]{66ff99}  160}}}%
      \colorrgb{0.40,1.00,0.60}
      \put(1216,2669){\makebox(0,0){\strut{}\scalebox{0.9}{\tiny \color[HTML]{66ff99}  336}}}%
      \colorrgb{0.40,1.00,0.60}
      \put(986,2408){\makebox(0,0){\strut{}\scalebox{0.9}{\tiny \color[HTML]{66ff99}  704}}}%
      \colorrgb{0.00,0.30,0.10}
      \put(2599,3345){\makebox(0,0){\strut{}\scalebox{0.9}{\tiny \color[HTML]{004d1a}  3}}}%
      \colorrgb{0.00,0.30,0.10}
      \put(2368,2964){\makebox(0,0){\strut{}\scalebox{0.9}{\tiny \color[HTML]{004d1a}  16}}}%
      \colorrgb{0.00,0.30,0.10}
      \put(2138,2667){\makebox(0,0){\strut{}\scalebox{0.9}{\tiny \color[HTML]{004d1a}  64}}}%
      \colorrgb{0.00,0.30,0.10}
      \put(1908,2407){\makebox(0,0){\strut{}\scalebox{0.9}{\tiny \color[HTML]{004d1a}  256}}}%
      \colorrgb{0.00,0.30,0.10}
      \put(1677,2157){\makebox(0,0){\strut{}\scalebox{0.9}{\tiny \color[HTML]{004d1a}  1024}}}%
      \colorrgb{0.00,0.30,0.10}
      \put(1447,1910){\makebox(0,0){\strut{}\scalebox{0.9}{\tiny \color[HTML]{004d1a}  4096}}}%
      \colorrgb{0.00,0.30,0.10}
      \put(1216,1663){\makebox(0,0){\strut{}\scalebox{0.9}{\tiny \color[HTML]{004d1a}  16384}}}%
      \csname LTb\endcsname
      \put(2464,954){\makebox(0,0)[r]{\strut{}\scriptsize$(\mathrm{E}),p_5$}}%
      \csname LTb\endcsname
      \put(2464,834){\makebox(0,0)[r]{\strut{}\scriptsize$(1),p_5$}}%
      \csname LTb\endcsname
      \put(2464,714){\makebox(0,0)[r]{\strut{}\scriptsize$(\infty),p_5$}}%
      \colorrgb{0.40,0.00,0.00}
      \put(2599,3345){\makebox(0,0){\strut{}\scalebox{0.9}{\tiny \color[HTML]{660000}  3}}}%
      \colorrgb{0.40,0.00,0.00}
      \put(2368,2701){\makebox(0,0){\strut{}\scalebox{0.9}{\tiny \color[HTML]{660000}  16}}}%
      \colorrgb{0.40,0.00,0.00}
      \put(2138,2199){\makebox(0,0){\strut{}\scalebox{0.9}{\tiny \color[HTML]{660000}  64}}}%
      \colorrgb{0.40,0.00,0.00}
      \put(1908,1757){\makebox(0,0){\strut{}\scalebox{0.9}{\tiny \color[HTML]{660000}  256}}}%
      \colorrgb{0.40,0.00,0.00}
      \put(1677,1337){\makebox(0,0){\strut{}\scalebox{0.9}{\tiny \color[HTML]{660000}  1024}}}%
      \colorrgb{0.40,0.00,0.00}
      \put(1447,937){\makebox(0,0){\strut{}\scalebox{0.9}{\tiny \color[HTML]{660000}  4096}}}%
      \colorrgb{1.00,0.00,0.00}
      \put(2138,3345){\makebox(0,0){\strut{}\scalebox{0.9}{\tiny \color[HTML]{ff0000}  20}}}%
      \colorrgb{1.00,0.00,0.00}
      \put(1908,2701){\makebox(0,0){\strut{}\scalebox{0.9}{\tiny \color[HTML]{ff0000}  48}}}%
      \colorrgb{1.00,0.00,0.00}
      \put(1677,2199){\makebox(0,0){\strut{}\scalebox{0.9}{\tiny \color[HTML]{ff0000}  112}}}%
      \colorrgb{1.00,0.00,0.00}
      \put(1447,1757){\makebox(0,0){\strut{}\scalebox{0.9}{\tiny \color[HTML]{ff0000}  256}}}%
      \colorrgb{1.00,0.00,0.00}
      \put(1216,1337){\makebox(0,0){\strut{}\scalebox{0.9}{\tiny \color[HTML]{ff0000}  576}}}%
      \colorrgb{1.00,0.60,0.60}
      \put(1908,3345){\makebox(0,0){\strut{}\scalebox{0.9}{\tiny \color[HTML]{ff9999}  36}}}%
      \colorrgb{1.00,0.60,0.60}
      \put(1677,2701){\makebox(0,0){\strut{}\scalebox{0.9}{\tiny \color[HTML]{ff9999}  80}}}%
      \colorrgb{1.00,0.60,0.60}
      \put(1447,2701){\makebox(0,0){\strut{}\scalebox{0.9}{\tiny \color[HTML]{ff9999}  160}}}%
      \colorrgb{1.00,0.60,0.60}
      \put(1216,2199){\makebox(0,0){\strut{}\scalebox{0.9}{\tiny \color[HTML]{ff9999}  336}}}%
      \colorrgb{1.00,0.60,0.60}
      \put(986,1757){\makebox(0,0){\strut{}\scalebox{0.9}{\tiny \color[HTML]{ff9999}  704}}}%
      \csname LTb\endcsname
      \put(48,1923){\rotatebox{-270.00}{\makebox(0,0){\strut{}$|\Phi_{h}^{(\star)}(t)- \Phi(t))|_{\mathrm{H}^1(\Omega)}$}}}%
      \csname LTb\endcsname
      \put(1796,93){\makebox(0,0){\strut{}Mesh size $h$}}%
    }%
    \gplbacktext
    \put(0,0){\includegraphics[width={187.00bp},height={172.00bp}]{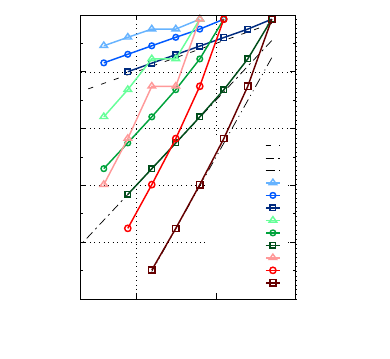}}%
    \gplfronttext
  \end{picture}%
\endgroup
}
 \captionof{figure}{$\mathrm{L}^2$-norm and $\mathrm{H}^1$-semi-norm approximation error of the charge density (left) and electric field (right) at time $t=0.5$ for the manufactured solutions as functions of the mesh size, for different spline degrees.}
\label{fig:0}
\end{center}

For the three approximation spaces, the charge density projection error scales as $\mathcal{O}(h^{p+1})$. These results are consistent with the theoretical bounds of~\cite[Theorem 3.2]{deluzet26} for the full-grid approximation space and better than the estimates predicted for the sparse-grid space as the convergence order does not include the additional logarithmic term. This super-convergence behavior results from the high regularity of the solution in the mixed derivatives. However, we observe a significant difference in the constants between the approximation spaces, which results in a consequent loss of accuracy for the $\mathrm{L}^2$- and $\mathrm{H}^1$-based sparse-grid spaces compared to the full-grid space; a deterioration that can be mitigated by increasing the degree of the B-splines functions.

The $\mathrm{H}^1$-semi-norm error of the electric potential is plotted on the right panel of~\cref{fig:1} as a function of the mesh size $h$ for different spline degrees $p$. 
For both approximation spaces, the error converges as $\mathcal{O}(h^{p})$, which is consistent with the error bounds obtained in~\cite[Theorem 3.2]{deluzet26}. Similarly to the charge density, increasing the spline degree yields a higher convergence order for all approximation spaces. \cle{However, the $\mathrm{H}^1$-based space requires more degrees of freedom than the $\mathrm{L}^2$-based space to achieve a given $\mathrm{H}^1$-error. At first glance, this may appear to contradict the $\mathrm{H}^1$-optimality of the $\mathrm{H}^1$-based space. The key point is that this optimality is an asymptotic property. Indeed, although all three spaces exhibit the same asymptotic convergence rate, $\mathcal{O}(h^2)$, the number of mesh nodes grows significantly more slowly for the $\mathrm{L}^2$-based space, and even more so for the energy-based space, than for the full-grid space.}

\paragraph{Adaptive mesh refinement strategy}
We now assess the benefits of the locally adaptive mesh refinement strategy on the grid-based error. We first consider the manufactured solution benchmark, whose smooth solution (a sine perturbation) is sufficiently regular for sparse-grid PIC methods to outperform the standard PIC method even without mesh refinement.

The $\mathrm{L}^2$-norm error of the projected charge density at time $t=0.5$ is shown as a function of the mesh size on the left panel of~\cref{fig:1}. We consider linear and cubic B-splines, admissibility tolerances $\epsilon_0=0$ and $\epsilon_1=Ch^{p+1}$, where $C=\mathcal{O}(1)$, and fix the generalized-space parameter to $\sigma_0=0$. The adaptive mesh refinement strategy is applied only for the $\epsilon_1$ configurations. Since no additional subspaces are introduced in this case, the resulting approximation space is strictly smaller than the $\mathrm{L}^2$-based sparse-grid space. For each method and mesh size, the corresponding number of mesh nodes is also reported.

All methods exhibit the expected convergence rate of $\mathcal{O}(h^{p+1})$, although the sparse-grid approximations have a larger error constant. The adaptive mesh refinement strategy substantially improves the efficiency of HSG-$(1)$ by significantly reducing the number of mesh nodes while preserving the approximation accuracy. In particular, the same error is achieved with less than half as many mesh nodes. Moreover, for fine meshes, HSG-A$(\sigma_0,\epsilon_1)$ attains essentially the same error as HSG-$(\infty)$ while requiring approximately seven times fewer mesh nodes.

\begin{center}
\begin{center}
  \resizebox{0.9\textwidth}{!}{
\begingroup
  \makeatletter
  \providecommand\color[2][]{%
    \GenericError{(gnuplot) \space\space\space\@spaces}{%
      Package color not loaded in conjunction with
      terminal option `colourtext'%
    }{See the gnuplot documentation for explanation.%
    }{Either use 'blacktext' in gnuplot or load the package
      color.sty in LaTeX.}%
    \renewcommand\color[2][]{}%
  }%
  \providecommand\includegraphics[2][]{%
    \GenericError{(gnuplot) \space\space\space\@spaces}{%
      Package graphicx or graphics not loaded%
    }{See the gnuplot documentation for explanation.%
    }{The gnuplot epslatex terminal needs graphicx.sty or graphics.sty.}%
    \renewcommand\includegraphics[2][]{}%
  }%
  \providecommand\rotatebox[2]{#2}%
  \@ifundefined{ifGPcolor}{%
    \newif\ifGPcolor
    \GPcolortrue
  }{}%
  \@ifundefined{ifGPblacktext}{%
    \newif\ifGPblacktext
    \GPblacktexttrue
  }{}%
  \let\gplgaddtomacro\g@addto@macro
  \gdef\gplbacktext{}%
  \gdef\gplfronttext{}%
  \makeatother
  \ifGPblacktext
    \def\colorrgb#1{}%
    \def\colorgray#1{}%
  \else
    \ifGPcolor
      \def\colorrgb#1{\color[rgb]{#1}}%
      \def\colorgray#1{\color[gray]{#1}}%
      \expandafter\def\csname LTw\endcsname{\color{white}}%
      \expandafter\def\csname LTb\endcsname{\color{black}}%
      \expandafter\def\csname LTa\endcsname{\color{black}}%
      \expandafter\def\csname LT0\endcsname{\color[rgb]{1,0,0}}%
      \expandafter\def\csname LT1\endcsname{\color[rgb]{0,1,0}}%
      \expandafter\def\csname LT2\endcsname{\color[rgb]{0,0,1}}%
      \expandafter\def\csname LT3\endcsname{\color[rgb]{1,0,1}}%
      \expandafter\def\csname LT4\endcsname{\color[rgb]{0,1,1}}%
      \expandafter\def\csname LT5\endcsname{\color[rgb]{1,1,0}}%
      \expandafter\def\csname LT6\endcsname{\color[rgb]{0,0,0}}%
      \expandafter\def\csname LT7\endcsname{\color[rgb]{1,0.3,0}}%
      \expandafter\def\csname LT8\endcsname{\color[rgb]{0.5,0.5,0.5}}%
    \else
      \def\colorrgb#1{\color{black}}%
      \def\colorgray#1{\color[gray]{#1}}%
      \expandafter\def\csname LTw\endcsname{\color{white}}%
      \expandafter\def\csname LTb\endcsname{\color{black}}%
      \expandafter\def\csname LTa\endcsname{\color{black}}%
      \expandafter\def\csname LT0\endcsname{\color{black}}%
      \expandafter\def\csname LT1\endcsname{\color{black}}%
      \expandafter\def\csname LT2\endcsname{\color{black}}%
      \expandafter\def\csname LT3\endcsname{\color{black}}%
      \expandafter\def\csname LT4\endcsname{\color{black}}%
      \expandafter\def\csname LT5\endcsname{\color{black}}%
      \expandafter\def\csname LT6\endcsname{\color{black}}%
      \expandafter\def\csname LT7\endcsname{\color{black}}%
      \expandafter\def\csname LT8\endcsname{\color{black}}%
    \fi
  \fi
    \setlength{\unitlength}{0.0500bp}%
    \ifx\gptboxheight\undefined%
      \newlength{\gptboxheight}%
      \newlength{\gptboxwidth}%
      \newsavebox{\gptboxtext}%
    \fi%
    \setlength{\fboxrule}{0.5pt}%
    \setlength{\fboxsep}{1pt}%
    \definecolor{tbcol}{rgb}{1,1,1}%
\begin{picture}(3740.00,3440.00)%
    \gplgaddtomacro\gplbacktext{%
      \csname LTb\endcsname
      \put(615,559){\makebox(0,0)[r]{\strut{}\small$10^{-9}$}}%
      \csname LTb\endcsname
      \put(615,862){\makebox(0,0)[r]{\strut{}\small$10^{-8}$}}%
      \csname LTb\endcsname
      \put(615,1165){\makebox(0,0)[r]{\strut{}\small$10^{-7}$}}%
      \csname LTb\endcsname
      \put(615,1468){\makebox(0,0)[r]{\strut{}\small$10^{-6}$}}%
      \csname LTb\endcsname
      \put(615,1771){\makebox(0,0)[r]{\strut{}\small$10^{-5}$}}%
      \csname LTb\endcsname
      \put(615,2074){\makebox(0,0)[r]{\strut{}\small$10^{-4}$}}%
      \csname LTb\endcsname
      \put(615,2377){\makebox(0,0)[r]{\strut{}\small$10^{-3}$}}%
      \csname LTb\endcsname
      \put(615,2680){\makebox(0,0)[r]{\strut{}\small$10^{-2}$}}%
      \csname LTb\endcsname
      \put(615,2983){\makebox(0,0)[r]{\strut{}\small$10^{-1}$}}%
      \csname LTb\endcsname
      \put(615,3286){\makebox(0,0)[r]{\strut{}\small$10^{0}$}}%
      \csname LTb\endcsname
      \put(1205,426){\makebox(0,0){\strut{}$10^{-2}$}}%
      \csname LTb\endcsname
      \put(1943,426){\makebox(0,0){\strut{}$10^{-1}$}}%
      \csname LTb\endcsname
      \put(2681,426){\makebox(0,0){\strut{}$\small 10^{0}$}}%
    }%
    \gplgaddtomacro\gplfronttext{%
      \csname LTb\endcsname
      \put(2316,1554){\makebox(0,0)[r]{\strut{}\scriptsize$h^{2}$}}%
      \csname LTb\endcsname
      \put(2316,1434){\makebox(0,0)[r]{\strut{}\scriptsize$h^{4}$}}%
      \csname LTb\endcsname
      \put(2316,1314){\makebox(0,0)[r]{\strut{}\scriptsize$(\infty)$}}%
      \csname LTb\endcsname
      \put(2316,1194){\makebox(0,0)[r]{\strut{}\scriptsize$\mathrm{A}(\sigma_0, \epsilon_0)$}}%
      \csname LTb\endcsname
      \put(2316,1074){\makebox(0,0)[r]{\strut{}\scriptsize$\mathrm{A}(\sigma_0, \epsilon_1)$}}%
      \csname LTb\endcsname
      \put(2316,954){\makebox(0,0)[r]{\strut{}\scriptsize$(\infty)$}}%
      \csname LTb\endcsname
      \put(2316,834){\makebox(0,0)[r]{\strut{}\scriptsize$\mathrm{A}(\sigma_0, \epsilon_0)$}}%
      \csname LTb\endcsname
      \put(2316,714){\makebox(0,0)[r]{\strut{}\scriptsize$\mathrm{A}(\sigma_0, \epsilon_1)$}}%
      \colorrgb{0.40,0.70,1.00}
      \put(2014,3280){\makebox(0,0){\strut{}\scalebox{0.9}{\tiny \color[HTML]{66b3ff}  20}}}%
      \colorrgb{0.40,0.70,1.00}
      \put(1792,3047){\makebox(0,0){\strut{}\scalebox{0.9}{\tiny \color[HTML]{66b3ff}  48}}}%
      \colorrgb{0.40,0.70,1.00}
      \put(1570,2848){\makebox(0,0){\strut{}\scalebox{0.9}{\tiny \color[HTML]{66b3ff}  112}}}%
      \colorrgb{0.40,0.70,1.00}
      \put(1348,2663){\makebox(0,0){\strut{}\scalebox{0.9}{\tiny \color[HTML]{66b3ff}  256}}}%
      \colorrgb{0.40,0.70,1.00}
      \put(1126,2483){\makebox(0,0){\strut{}\scalebox{0.9}{\tiny \color[HTML]{66b3ff}  576}}}%
      \colorrgb{0.40,0.70,1.00}
      \put(904,2303){\makebox(0,0){\strut{}\scalebox{0.9}{\tiny \color[HTML]{66b3ff}  1280}}}%
      \colorrgb{0.00,0.17,0.50}
      \put(2459,3280){\makebox(0,0){\strut{}\scalebox{0.9}{\tiny \color[HTML]{002b7f}  4}}}%
      \colorrgb{0.00,0.17,0.50}
      \put(2236,3044){\makebox(0,0){\strut{}\scalebox{0.9}{\tiny \color[HTML]{002b7f}  16}}}%
      \colorrgb{0.00,0.17,0.50}
      \put(2014,2835){\makebox(0,0){\strut{}\scalebox{0.9}{\tiny \color[HTML]{002b7f}  64}}}%
      \colorrgb{0.00,0.17,0.50}
      \put(1792,2645){\makebox(0,0){\strut{}\scalebox{0.9}{\tiny \color[HTML]{002b7f}  256}}}%
      \colorrgb{0.00,0.17,0.50}
      \put(1570,2461){\makebox(0,0){\strut{}\scalebox{0.9}{\tiny \color[HTML]{002b7f}  1024}}}%
      \colorrgb{0.00,0.17,0.50}
      \put(1348,2278){\makebox(0,0){\strut{}\scalebox{0.9}{\tiny \color[HTML]{002b7f}  4096}}}%
      \colorrgb{0.00,0.17,0.50}
      \put(1126,2095){\makebox(0,0){\strut{}\scalebox{0.9}{\tiny \color[HTML]{002b7f}  16384}}}%
      \colorrgb{0.00,0.35,1.00}
      \put(2014,3147){\makebox(0,0){\strut{}\scalebox{0.9}{\tiny \color[HTML]{0059ff}  3}}}%
      \colorrgb{0.00,0.35,1.00}
      \put(1792,2913){\makebox(0,0){\strut{}\scalebox{0.9}{\tiny \color[HTML]{0059ff}  8}}}%
      \colorrgb{0.00,0.35,1.00}
      \put(1570,2715){\makebox(0,0){\strut{}\scalebox{0.9}{\tiny \color[HTML]{0059ff}  28}}}%
      \colorrgb{0.00,0.35,1.00}
      \put(1348,2530){\makebox(0,0){\strut{}\scalebox{0.9}{\tiny \color[HTML]{0059ff}  88}}}%
      \colorrgb{0.00,0.35,1.00}
      \put(1126,2349){\makebox(0,0){\strut{}\scalebox{0.9}{\tiny \color[HTML]{0059ff}  240}}}%
      \colorrgb{0.00,0.35,1.00}
      \put(904,2174){\makebox(0,0){\strut{}\scalebox{0.9}{\tiny \color[HTML]{0059ff}  592}}}%
      \colorrgb{0.00,0.64,0.24}
      \put(2014,3280){\makebox(0,0){\strut{}\scalebox{0.9}{\tiny \color[HTML]{00a33c}  20}}}%
      \colorrgb{0.00,0.64,0.24}
      \put(1792,2749){\makebox(0,0){\strut{}\scalebox{0.9}{\tiny \color[HTML]{00a33c}  48}}}%
      \colorrgb{0.00,0.64,0.24}
      \put(1570,2314){\makebox(0,0){\strut{}\scalebox{0.9}{\tiny \color[HTML]{00a33c}  112}}}%
      \colorrgb{0.00,0.64,0.24}
      \put(1348,1925){\makebox(0,0){\strut{}\scalebox{0.9}{\tiny \color[HTML]{00a33c}  256}}}%
      \colorrgb{0.00,0.64,0.24}
      \put(1126,1553){\makebox(0,0){\strut{}\scalebox{0.9}{\tiny \color[HTML]{00a33c}  576}}}%
      \colorrgb{0.40,1.00,0.60}
      \put(2014,3147){\makebox(0,0){\strut{}\scalebox{0.9}{\tiny \color[HTML]{66ff99}  3}}}%
      \colorrgb{0.40,1.00,0.60}
      \put(1792,2616){\makebox(0,0){\strut{}\scalebox{0.9}{\tiny \color[HTML]{66ff99}  8}}}%
      \colorrgb{0.40,1.00,0.60}
      \put(1570,2181){\makebox(0,0){\strut{}\scalebox{0.9}{\tiny \color[HTML]{66ff99}  28}}}%
      \colorrgb{0.40,1.00,0.60}
      \put(1348,1792){\makebox(0,0){\strut{}\scalebox{0.9}{\tiny \color[HTML]{66ff99}  88}}}%
      \colorrgb{0.40,1.00,0.60}
      \put(1126,1420){\makebox(0,0){\strut{}\scalebox{0.9}{\tiny \color[HTML]{66ff99}  240}}}%
      \colorrgb{0.00,0.30,0.10}
      \put(2459,3147){\makebox(0,0){\strut{}\scalebox{0.9}{\tiny \color[HTML]{004d1a}  4}}}%
      \colorrgb{0.00,0.30,0.10}
      \put(2236,2616){\makebox(0,0){\strut{}\scalebox{0.9}{\tiny \color[HTML]{004d1a}  16}}}%
      \colorrgb{0.00,0.30,0.10}
      \put(2014,2177){\makebox(0,0){\strut{}\scalebox{0.9}{\tiny \color[HTML]{004d1a}  64}}}%
      \colorrgb{0.00,0.30,0.10}
      \put(1792,1788){\makebox(0,0){\strut{}\scalebox{0.9}{\tiny \color[HTML]{004d1a}  256}}}%
      \colorrgb{0.00,0.30,0.10}
      \put(1570,1417){\makebox(0,0){\strut{}\scalebox{0.9}{\tiny \color[HTML]{004d1a}  1024}}}%
      \colorrgb{0.00,0.30,0.10}
      \put(1348,1050){\makebox(0,0){\strut{}\scalebox{0.9}{\tiny \color[HTML]{004d1a}  4096}}}%
      \csname LTb\endcsname
      \put(48,1923){\rotatebox{-270.00}{\makebox(0,0){\strut{}$\|\Pi_{h}^{(\star)} \rho(t) - \rho(t))\|_{\mathrm{L}^2(\Omega)}$}}}%
      \csname LTb\endcsname
      \put(1685,93){\makebox(0,0){\strut{}Mesh size $h$}}%
    }%
    \gplbacktext
    \put(0,0){\includegraphics[width={187.00bp},height={172.00bp}]{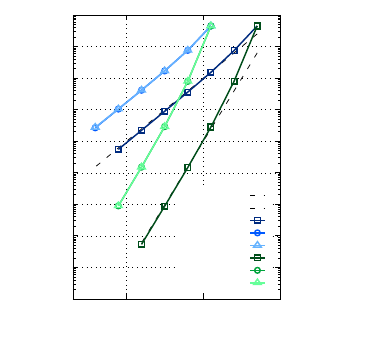}}%
    \gplfronttext
  \end{picture}%
\endgroup
\begingroup
  \makeatletter
  \providecommand\color[2][]{%
    \GenericError{(gnuplot) \space\space\space\@spaces}{%
      Package color not loaded in conjunction with
      terminal option `colourtext'%
    }{See the gnuplot documentation for explanation.%
    }{Either use 'blacktext' in gnuplot or load the package
      color.sty in LaTeX.}%
    \renewcommand\color[2][]{}%
  }%
  \providecommand\includegraphics[2][]{%
    \GenericError{(gnuplot) \space\space\space\@spaces}{%
      Package graphicx or graphics not loaded%
    }{See the gnuplot documentation for explanation.%
    }{The gnuplot epslatex terminal needs graphicx.sty or graphics.sty.}%
    \renewcommand\includegraphics[2][]{}%
  }%
  \providecommand\rotatebox[2]{#2}%
  \@ifundefined{ifGPcolor}{%
    \newif\ifGPcolor
    \GPcolortrue
  }{}%
  \@ifundefined{ifGPblacktext}{%
    \newif\ifGPblacktext
    \GPblacktexttrue
  }{}%
  \let\gplgaddtomacro\g@addto@macro
  \gdef\gplbacktext{}%
  \gdef\gplfronttext{}%
  \makeatother
  \ifGPblacktext
    \def\colorrgb#1{}%
    \def\colorgray#1{}%
  \else
    \ifGPcolor
      \def\colorrgb#1{\color[rgb]{#1}}%
      \def\colorgray#1{\color[gray]{#1}}%
      \expandafter\def\csname LTw\endcsname{\color{white}}%
      \expandafter\def\csname LTb\endcsname{\color{black}}%
      \expandafter\def\csname LTa\endcsname{\color{black}}%
      \expandafter\def\csname LT0\endcsname{\color[rgb]{1,0,0}}%
      \expandafter\def\csname LT1\endcsname{\color[rgb]{0,1,0}}%
      \expandafter\def\csname LT2\endcsname{\color[rgb]{0,0,1}}%
      \expandafter\def\csname LT3\endcsname{\color[rgb]{1,0,1}}%
      \expandafter\def\csname LT4\endcsname{\color[rgb]{0,1,1}}%
      \expandafter\def\csname LT5\endcsname{\color[rgb]{1,1,0}}%
      \expandafter\def\csname LT6\endcsname{\color[rgb]{0,0,0}}%
      \expandafter\def\csname LT7\endcsname{\color[rgb]{1,0.3,0}}%
      \expandafter\def\csname LT8\endcsname{\color[rgb]{0.5,0.5,0.5}}%
    \else
      \def\colorrgb#1{\color{black}}%
      \def\colorgray#1{\color[gray]{#1}}%
      \expandafter\def\csname LTw\endcsname{\color{white}}%
      \expandafter\def\csname LTb\endcsname{\color{black}}%
      \expandafter\def\csname LTa\endcsname{\color{black}}%
      \expandafter\def\csname LT0\endcsname{\color{black}}%
      \expandafter\def\csname LT1\endcsname{\color{black}}%
      \expandafter\def\csname LT2\endcsname{\color{black}}%
      \expandafter\def\csname LT3\endcsname{\color{black}}%
      \expandafter\def\csname LT4\endcsname{\color{black}}%
      \expandafter\def\csname LT5\endcsname{\color{black}}%
      \expandafter\def\csname LT6\endcsname{\color{black}}%
      \expandafter\def\csname LT7\endcsname{\color{black}}%
      \expandafter\def\csname LT8\endcsname{\color{black}}%
    \fi
  \fi
    \setlength{\unitlength}{0.0500bp}%
    \ifx\gptboxheight\undefined%
      \newlength{\gptboxheight}%
      \newlength{\gptboxwidth}%
      \newsavebox{\gptboxtext}%
    \fi%
    \setlength{\fboxrule}{0.5pt}%
    \setlength{\fboxsep}{1pt}%
    \definecolor{tbcol}{rgb}{1,1,1}%
\begin{picture}(3740.00,3440.00)%
    \gplgaddtomacro\gplbacktext{%
      \csname LTb\endcsname
      \put(615,559){\makebox(0,0)[r]{\strut{}\small$10^{-4}$}}%
      \csname LTb\endcsname
      \put(615,1241){\makebox(0,0)[r]{\strut{}\small$10^{-3}$}}%
      \csname LTb\endcsname
      \put(615,1923){\makebox(0,0)[r]{\strut{}\small$10^{-2}$}}%
      \csname LTb\endcsname
      \put(615,2604){\makebox(0,0)[r]{\strut{}\small$10^{-1}$}}%
      \csname LTb\endcsname
      \put(615,3286){\makebox(0,0)[r]{\strut{}\small$10^{0}$}}%
      \csname LTb\endcsname
      \put(1205,426){\makebox(0,0){\strut{}$10^{-2}$}}%
      \csname LTb\endcsname
      \put(1943,426){\makebox(0,0){\strut{}$10^{-1}$}}%
      \csname LTb\endcsname
      \put(2681,426){\makebox(0,0){\strut{}$\small 10^{0}$}}%
    }%
    \gplgaddtomacro\gplfronttext{%
      \csname LTb\endcsname
      \put(2316,1074){\makebox(0,0)[r]{\strut{}\scriptsize$h^{2}$}}%
      \csname LTb\endcsname
      \put(2316,954){\makebox(0,0)[r]{\strut{}\scriptsize$(\infty)$}}%
      \csname LTb\endcsname
      \put(2316,834){\makebox(0,0)[r]{\strut{}\scriptsize $\mathrm{A}(\sigma_0, \epsilon_0)$}}%
      \csname LTb\endcsname
      \put(2316,714){\makebox(0,0)[r]{\strut{}\scriptsize $\mathrm{A}(\sigma_1, \epsilon_1)$}}%
      \colorrgb{0.00,0.35,1.00}
      \put(2014,2893){\makebox(0,0){\strut{}\scalebox{0.9}{\tiny \color[HTML]{0059ff}  20}}}%
      \colorrgb{0.00,0.35,1.00}
      \put(1792,2832){\makebox(0,0){\strut{}\scalebox{0.9}{\tiny \color[HTML]{0059ff}  48}}}%
      \colorrgb{0.00,0.35,1.00}
      \put(1570,2769){\makebox(0,0){\strut{}\scalebox{0.9}{\tiny \color[HTML]{0059ff}  112}}}%
      \colorrgb{0.00,0.35,1.00}
      \put(1348,2732){\makebox(0,0){\strut{}\scalebox{0.9}{\tiny \color[HTML]{0059ff}  256}}}%
      \colorrgb{0.00,0.35,1.00}
      \put(1126,2580){\makebox(0,0){\strut{}\scalebox{0.9}{\tiny \color[HTML]{0059ff}  576}}}%
      \colorrgb{0.00,0.35,1.00}
      \put(904,2428){\makebox(0,0){\strut{}\scalebox{0.9}{\tiny \color[HTML]{0059ff}  1280}}}%
      \colorrgb{0.00,0.17,0.50}
      \put(2459,2760){\makebox(0,0){\strut{}\scalebox{0.9}{\tiny \color[HTML]{002b7f}  4}}}%
      \colorrgb{0.00,0.17,0.50}
      \put(2236,2870){\makebox(0,0){\strut{}\scalebox{0.9}{\tiny \color[HTML]{002b7f}  16}}}%
      \colorrgb{0.00,0.17,0.50}
      \put(2014,2811){\makebox(0,0){\strut{}\scalebox{0.9}{\tiny \color[HTML]{002b7f}  64}}}%
      \colorrgb{0.00,0.17,0.50}
      \put(1792,2564){\makebox(0,0){\strut{}\scalebox{0.9}{\tiny \color[HTML]{002b7f}  256}}}%
      \colorrgb{0.00,0.17,0.50}
      \put(1570,2154){\makebox(0,0){\strut{}\scalebox{0.9}{\tiny \color[HTML]{002b7f}  1024}}}%
      \colorrgb{0.00,0.17,0.50}
      \put(1348,1676){\makebox(0,0){\strut{}\scalebox{0.9}{\tiny \color[HTML]{002b7f}  4096}}}%
      \colorrgb{0.00,0.17,0.50}
      \put(1126,1243){\makebox(0,0){\strut{}\scalebox{0.9}{\tiny \color[HTML]{002b7f}  16384}}}%
      \colorrgb{0.00,0.17,0.50}
      \put(904,827){\makebox(0,0){\strut{}\scalebox{0.9}{\tiny \color[HTML]{002b7f}  65536}}}%
      \colorrgb{0.40,0.70,1.00}
      \put(1570,2614){\makebox(0,0){\strut{}\scalebox{0.9}{\tiny \color[HTML]{66b3ff}  152}}}%
      \colorrgb{0.40,0.70,1.00}
      \put(1348,2230){\makebox(0,0){\strut{}\scalebox{0.9}{\tiny \color[HTML]{66b3ff}  532}}}%
      \colorrgb{0.40,0.70,1.00}
      \put(1126,1716){\makebox(0,0){\strut{}\scalebox{0.9}{\tiny \color[HTML]{66b3ff}  1604}}}%
      \colorrgb{0.40,0.70,1.00}
      \put(904,1261){\makebox(0,0){\strut{}\scalebox{0.9}{\tiny \color[HTML]{66b3ff}  3928}}}%
      \csname LTb\endcsname
      \put(48,1923){\rotatebox{-270.00}{\makebox(0,0){\strut{}$\|\Pi_{h}^{(\star)} \rho(t) - \rho(t))\|_{\mathrm{L}^2(\Omega)}$}}}%
      \csname LTb\endcsname
      \put(1685,93){\makebox(0,0){\strut{}Mesh size $h$}}%
    }%
    \gplbacktext
    \put(0,0){\includegraphics[width={187.00bp},height={172.00bp}]{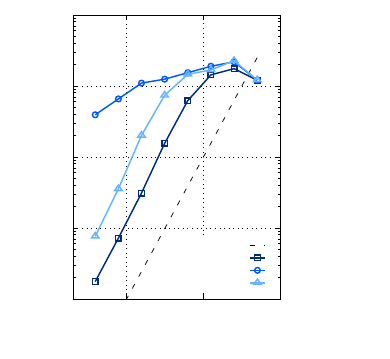}}%
    \gplfronttext
  \end{picture}%
\endgroup
 
	}
 \end{center}
 \captionof{figure}{\Fabrice{$\mathrm{L}^2$-norm approximation error of the charge density for the manufactured solution at $t=0.5$ (left) and diocotron instability at $t=0$ (right) as functions of the mesh size, for $p_1$ (blue curves) and $p_3$ (green curves) splines. Admissibility tolerances and generalized-spaces are parameterized by $(\epsilon_0,\epsilon_1)=(0,10^{-3})$ and $(\sigma_0,\sigma_1)=(0,4)$.}}
\label{fig:1}
\end{center}
We now consider the diocotron instability test case and restrict our attention to the initial state, for which the exact charge density is known. In this configuration, sparse-grid methods struggle to accurately represent the solution because of its strong anisotropy. Consequently, even at the initial time, the convergence rate deteriorates and departs from the theoretical estimate of $\mathcal{O}(h^{p+1}|\log h|^{d-1})$. The $\mathrm{L}^2$-norm error of the projected charge density at $t=0$ is shown as a function of the mesh size in the right panel of~\cref{fig:1}. Only linear B-splines are considered, as it was observed in~\cite{deluzet26} that increasing the spline degree provides only marginal accuracy improvements.

The admissibility tolerance is chosen as either $\epsilon_0=0$ or $\epsilon_1=Ch^2$, where $C=\mathcal{O}(1)$ is a scaling constant, while the generalized-space parameter is set to either $\sigma_0=0$ or $\sigma_1=4$. Increasing the generalized-space parameter enriches the approximation space by incorporating additional subspaces, thereby increasing the number of mesh nodes. Consequently, the adaptive truncation strategy is essential to preserve the low computational complexity of sparse-grid approximations while outperforming the standard approach.

The adaptive mesh refinement strategy substantially improves the accuracy of the HSG-$(1)$ method. In particular, it recovers the same second-order convergence rate as HSG-$(\infty)$, namely $\mathcal{O}(h^2)$, although with a larger error constant. At the same time, it requires significantly fewer mesh nodes than the full-grid method, reducing their number by approximately a factor of four on the finest meshes considered.

\subsection{Total error}
We now investigate the impact of the mesh refinement strategy on the total error, which consists of both grid-based and statistical components and depends on three discretization parameters: the mesh size, the spline degree, and the total number of particles. Throughout this section, we restrict our attention to linear B-splines, as this configuration provides the best performance for the diocotron instability test case.

For each experiment presented in this section, we perform $15$ independent simulations using different random seeds to generate the initial particle distribution. The reported results correspond to the mean and standard deviation over these $15$ realizations. We fix the number of particles per species to $N=1,000,000$, corresponding to $N$ electrons and $N$ ions when both species are modeled, which is a representative value for the test cases considered. With this choice, the total error is dominated by the grid-based error on coarse meshes and by the statistical error on fine meshes, allowing both error regimes to be examined. The errors are evaluated at time $t=0.5$ for the manufactured solution and at time $t=0$ for the diocotron instability. The admissibility tolerance and generalized-space parameters are set to $(\sigma_0,\epsilon_0)=(0,0)$ for the manufactured solution and to $(\sigma_1,\epsilon_1)=(4,10^{-3})$ for the diocotron instability.

The $\mathrm{L}^2$-norm error of the charge density approximation is shown as a function of the mesh size in~\cref{fig:2}, with the manufactured solution displayed in the left panel and the diocotron instability in the right panel.
\begin{center}
\noindent
\begin{minipage}[c]{0.65\textwidth}
  \centering
  \resizebox{\linewidth}{!}{%
\begingroup
  \makeatletter
  \providecommand\color[2][]{%
    \GenericError{(gnuplot) \space\space\space\@spaces}{%
      Package color not loaded in conjunction with
      terminal option `colourtext'%
    }{See the gnuplot documentation for explanation.%
    }{Either use 'blacktext' in gnuplot or load the package
      color.sty in LaTeX.}%
    \renewcommand\color[2][]{}%
  }%
  \providecommand\includegraphics[2][]{%
    \GenericError{(gnuplot) \space\space\space\@spaces}{%
      Package graphicx or graphics not loaded%
    }{See the gnuplot documentation for explanation.%
    }{The gnuplot epslatex terminal needs graphicx.sty or graphics.sty.}%
    \renewcommand\includegraphics[2][]{}%
  }%
  \providecommand\rotatebox[2]{#2}%
  \@ifundefined{ifGPcolor}{%
    \newif\ifGPcolor
    \GPcolortrue
  }{}%
  \@ifundefined{ifGPblacktext}{%
    \newif\ifGPblacktext
    \GPblacktexttrue
  }{}%
  \let\gplgaddtomacro\g@addto@macro
  \gdef\gplbacktext{}%
  \gdef\gplfronttext{}%
  \makeatother
  \ifGPblacktext
    \def\colorrgb#1{}%
    \def\colorgray#1{}%
  \else
    \ifGPcolor
      \def\colorrgb#1{\color[rgb]{#1}}%
      \def\colorgray#1{\color[gray]{#1}}%
      \expandafter\def\csname LTw\endcsname{\color{white}}%
      \expandafter\def\csname LTb\endcsname{\color{black}}%
      \expandafter\def\csname LTa\endcsname{\color{black}}%
      \expandafter\def\csname LT0\endcsname{\color[rgb]{1,0,0}}%
      \expandafter\def\csname LT1\endcsname{\color[rgb]{0,1,0}}%
      \expandafter\def\csname LT2\endcsname{\color[rgb]{0,0,1}}%
      \expandafter\def\csname LT3\endcsname{\color[rgb]{1,0,1}}%
      \expandafter\def\csname LT4\endcsname{\color[rgb]{0,1,1}}%
      \expandafter\def\csname LT5\endcsname{\color[rgb]{1,1,0}}%
      \expandafter\def\csname LT6\endcsname{\color[rgb]{0,0,0}}%
      \expandafter\def\csname LT7\endcsname{\color[rgb]{1,0.3,0}}%
      \expandafter\def\csname LT8\endcsname{\color[rgb]{0.5,0.5,0.5}}%
    \else
      \def\colorrgb#1{\color{black}}%
      \def\colorgray#1{\color[gray]{#1}}%
      \expandafter\def\csname LTw\endcsname{\color{white}}%
      \expandafter\def\csname LTb\endcsname{\color{black}}%
      \expandafter\def\csname LTa\endcsname{\color{black}}%
      \expandafter\def\csname LT0\endcsname{\color{black}}%
      \expandafter\def\csname LT1\endcsname{\color{black}}%
      \expandafter\def\csname LT2\endcsname{\color{black}}%
      \expandafter\def\csname LT3\endcsname{\color{black}}%
      \expandafter\def\csname LT4\endcsname{\color{black}}%
      \expandafter\def\csname LT5\endcsname{\color{black}}%
      \expandafter\def\csname LT6\endcsname{\color{black}}%
      \expandafter\def\csname LT7\endcsname{\color{black}}%
      \expandafter\def\csname LT8\endcsname{\color{black}}%
    \fi
  \fi
    \setlength{\unitlength}{0.0500bp}%
    \ifx\gptboxheight\undefined%
      \newlength{\gptboxheight}%
      \newlength{\gptboxwidth}%
      \newsavebox{\gptboxtext}%
    \fi%
    \setlength{\fboxrule}{0.5pt}%
    \setlength{\fboxsep}{1pt}%
    \definecolor{tbcol}{rgb}{1,1,1}%
\begin{picture}(3740.00,3440.00)%
    \gplgaddtomacro\gplbacktext{%
      \csname LTb\endcsname
      \put(615,916){\makebox(0,0)[r]{\strut{}\small$10^{-2}$}}%
      \csname LTb\endcsname
      \put(615,2101){\makebox(0,0)[r]{\strut{}\small$10^{-1}$}}%
      \csname LTb\endcsname
      \put(615,3286){\makebox(0,0)[r]{\strut{}\small$10^{0}$}}%
      \csname LTb\endcsname
      \put(1205,426){\makebox(0,0){\strut{}$10^{-2}$}}%
      \csname LTb\endcsname
      \put(1943,426){\makebox(0,0){\strut{}$10^{-1}$}}%
      \csname LTb\endcsname
      \put(2681,426){\makebox(0,0){\strut{}$\small 10^{0}$}}%
    }%
    \gplgaddtomacro\gplfronttext{%
      \csname LTb\endcsname
      \put(2316,1074){\makebox(0,0)[r]{\strut{}\scriptsize$h^{2}$}}%
      \csname LTb\endcsname
      \put(2316,954){\makebox(0,0)[r]{\strut{}\scriptsize$(\infty)$}}%
      \colorrgb{0.00,0.17,0.50}
      \put(2459,2972){\makebox(0,0){\strut{}\scalebox{0.9}{\tiny \color[HTML]{002b7f}  4}}}%
      \colorrgb{0.00,0.17,0.50}
      \put(2236,2049){\makebox(0,0){\strut{}\scalebox{0.9}{\tiny \color[HTML]{002b7f}  16}}}%
      \colorrgb{0.00,0.17,0.50}
      \put(2014,1288){\makebox(0,0){\strut{}\scalebox{0.9}{\tiny \color[HTML]{002b7f}  64}}}%
      \colorrgb{0.00,0.17,0.50}
      \put(1792,1265){\makebox(0,0){\strut{}\scalebox{0.9}{\tiny \color[HTML]{002b7f}  256}}}%
      \colorrgb{0.00,0.17,0.50}
      \put(1570,1610){\makebox(0,0){\strut{}\scalebox{0.9}{\tiny \color[HTML]{002b7f}  1024}}}%
      \colorrgb{0.00,0.17,0.50}
      \put(1348,1968){\makebox(0,0){\strut{}\scalebox{0.9}{\tiny \color[HTML]{002b7f}  4096}}}%
      \colorrgb{0.00,0.17,0.50}
      \put(1126,2326){\makebox(0,0){\strut{}\scalebox{0.9}{\tiny \color[HTML]{002b7f}  16384}}}%
      \csname LTb\endcsname
      \put(2316,834){\makebox(0,0)[r]{\strut{}\scriptsize$\mathrm{A}(\sigma_0, \epsilon_0)$}}%
      \colorrgb{0.40,0.70,1.00}
      \put(2014,2972){\makebox(0,0){\strut{}\scalebox{0.9}{\tiny \color[HTML]{66b3ff}  20}}}%
      \colorrgb{0.40,0.70,1.00}
      \put(1792,2060){\makebox(0,0){\strut{}\scalebox{0.9}{\tiny \color[HTML]{66b3ff}  48}}}%
      \colorrgb{0.40,0.70,1.00}
      \put(1570,1368){\makebox(0,0){\strut{}\scalebox{0.9}{\tiny \color[HTML]{66b3ff}  112}}}%
      \colorrgb{0.40,0.70,1.00}
      \put(1348,1257){\makebox(0,0){\strut{}\scalebox{0.9}{\tiny \color[HTML]{66b3ff}  256}}}%
      \colorrgb{0.40,0.70,1.00}
      \put(1126,1476){\makebox(0,0){\strut{}\scalebox{0.9}{\tiny \color[HTML]{66b3ff}  576}}}%
      \csname LTb\endcsname
      \put(2316,714){\makebox(0,0)[r]{\strut{}\scriptsize$\mathrm{A}(\sigma_0, \epsilon_1)$}}%
      \colorrgb{0.00,0.35,1.00}
      \put(2014,2839){\makebox(0,0){\strut{}\scalebox{0.9}{\tiny \color[HTML]{0059ff}  12}}}%
      \colorrgb{0.00,0.35,1.00}
      \put(1792,1926){\makebox(0,0){\strut{}\scalebox{0.9}{\tiny \color[HTML]{0059ff}  10}}}%
      \colorrgb{0.00,0.35,1.00}
      \put(1570,1192){\makebox(0,0){\strut{}\scalebox{0.9}{\tiny \color[HTML]{0059ff}  36}}}%
      \colorrgb{0.00,0.35,1.00}
      \put(1348,1064){\makebox(0,0){\strut{}\scalebox{0.9}{\tiny \color[HTML]{0059ff}  59}}}%
      \colorrgb{0.00,0.35,1.00}
      \put(1126,1088){\makebox(0,0){\strut{}\scalebox{0.9}{\tiny \color[HTML]{0059ff}  64}}}%
      \csname LTb\endcsname
      \put(48,1923){\rotatebox{-270.00}{\makebox(0,0){\strut{}$\|\rho_h^{(\star),\tau} - \rho(t))\|_{\mathrm{L}^2(\Omega)}$}}}%
      \csname LTb\endcsname
      \put(1685,93){\makebox(0,0){\strut{}Mesh size $h$}}%
    }%
    \gplbacktext
    \put(0,0){\includegraphics[width={187.00bp},height={172.00bp}]{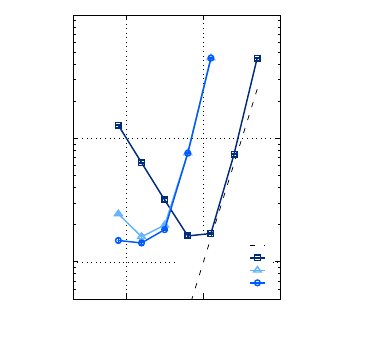}}%
    \gplfronttext
  \end{picture}%
\endgroup
\hspace{-0.33\textwidth}
\begingroup
  \makeatletter
  \providecommand\color[2][]{%
    \GenericError{(gnuplot) \space\space\space\@spaces}{%
      Package color not loaded in conjunction with
      terminal option `colourtext'%
    }{See the gnuplot documentation for explanation.%
    }{Either use 'blacktext' in gnuplot or load the package
      color.sty in LaTeX.}%
    \renewcommand\color[2][]{}%
  }%
  \providecommand\includegraphics[2][]{%
    \GenericError{(gnuplot) \space\space\space\@spaces}{%
      Package graphicx or graphics not loaded%
    }{See the gnuplot documentation for explanation.%
    }{The gnuplot epslatex terminal needs graphicx.sty or graphics.sty.}%
    \renewcommand\includegraphics[2][]{}%
  }%
  \providecommand\rotatebox[2]{#2}%
  \@ifundefined{ifGPcolor}{%
    \newif\ifGPcolor
    \GPcolortrue
  }{}%
  \@ifundefined{ifGPblacktext}{%
    \newif\ifGPblacktext
    \GPblacktexttrue
  }{}%
  \let\gplgaddtomacro\g@addto@macro
  \gdef\gplbacktext{}%
  \gdef\gplfronttext{}%
  \makeatother
  \ifGPblacktext
    \def\colorrgb#1{}%
    \def\colorgray#1{}%
  \else
    \ifGPcolor
      \def\colorrgb#1{\color[rgb]{#1}}%
      \def\colorgray#1{\color[gray]{#1}}%
      \expandafter\def\csname LTw\endcsname{\color{white}}%
      \expandafter\def\csname LTb\endcsname{\color{black}}%
      \expandafter\def\csname LTa\endcsname{\color{black}}%
      \expandafter\def\csname LT0\endcsname{\color[rgb]{1,0,0}}%
      \expandafter\def\csname LT1\endcsname{\color[rgb]{0,1,0}}%
      \expandafter\def\csname LT2\endcsname{\color[rgb]{0,0,1}}%
      \expandafter\def\csname LT3\endcsname{\color[rgb]{1,0,1}}%
      \expandafter\def\csname LT4\endcsname{\color[rgb]{0,1,1}}%
      \expandafter\def\csname LT5\endcsname{\color[rgb]{1,1,0}}%
      \expandafter\def\csname LT6\endcsname{\color[rgb]{0,0,0}}%
      \expandafter\def\csname LT7\endcsname{\color[rgb]{1,0.3,0}}%
      \expandafter\def\csname LT8\endcsname{\color[rgb]{0.5,0.5,0.5}}%
    \else
      \def\colorrgb#1{\color{black}}%
      \def\colorgray#1{\color[gray]{#1}}%
      \expandafter\def\csname LTw\endcsname{\color{white}}%
      \expandafter\def\csname LTb\endcsname{\color{black}}%
      \expandafter\def\csname LTa\endcsname{\color{black}}%
      \expandafter\def\csname LT0\endcsname{\color{black}}%
      \expandafter\def\csname LT1\endcsname{\color{black}}%
      \expandafter\def\csname LT2\endcsname{\color{black}}%
      \expandafter\def\csname LT3\endcsname{\color{black}}%
      \expandafter\def\csname LT4\endcsname{\color{black}}%
      \expandafter\def\csname LT5\endcsname{\color{black}}%
      \expandafter\def\csname LT6\endcsname{\color{black}}%
      \expandafter\def\csname LT7\endcsname{\color{black}}%
      \expandafter\def\csname LT8\endcsname{\color{black}}%
    \fi
  \fi
    \setlength{\unitlength}{0.0500bp}%
    \ifx\gptboxheight\undefined%
      \newlength{\gptboxheight}%
      \newlength{\gptboxwidth}%
      \newsavebox{\gptboxtext}%
    \fi%
    \setlength{\fboxrule}{0.5pt}%
    \setlength{\fboxsep}{1pt}%
    \definecolor{tbcol}{rgb}{1,1,1}%
\begin{picture}(3740.00,3440.00)%
    \gplgaddtomacro\gplbacktext{%
      \csname LTb\endcsname
      \put(1205,426){\makebox(0,0){\strut{}$10^{-2}$}}%
      \csname LTb\endcsname
      \put(1943,426){\makebox(0,0){\strut{}$10^{-1}$}}%
      \csname LTb\endcsname
      \put(2681,426){\makebox(0,0){\strut{}$\small 10^{0}$}}%
    }%
    \gplgaddtomacro\gplfronttext{%
      \csname LTb\endcsname
      \put(2316,1074){\makebox(0,0)[r]{\strut{}\scriptsize$h^{2}$}}%
      \csname LTb\endcsname
      \put(2316,954){\makebox(0,0)[r]{\strut{}\scriptsize$(\infty)$}}%
      \colorrgb{0.00,0.17,0.50}
      \put(2459,2338){\makebox(0,0){\strut{}\scalebox{0.9}{\tiny \color[HTML]{002b7f}  4}}}%
      \colorrgb{0.00,0.17,0.50}
      \put(2236,2491){\makebox(0,0){\strut{}\scalebox{0.9}{\tiny \color[HTML]{002b7f}  16}}}%
      \colorrgb{0.00,0.17,0.50}
      \put(2014,2387){\makebox(0,0){\strut{}\scalebox{0.9}{\tiny \color[HTML]{002b7f}  64}}}%
      \colorrgb{0.00,0.17,0.50}
      \put(1792,1972){\makebox(0,0){\strut{}\scalebox{0.9}{\tiny \color[HTML]{002b7f}  256}}}%
      \colorrgb{0.00,0.17,0.50}
      \put(1570,1477){\makebox(0,0){\strut{}\scalebox{0.9}{\tiny \color[HTML]{002b7f}  1024}}}%
      \colorrgb{0.00,0.17,0.50}
      \put(1348,1373){\makebox(0,0){\strut{}\scalebox{0.9}{\tiny \color[HTML]{002b7f}  4096}}}%
      \colorrgb{0.00,0.17,0.50}
      \put(1126,1433){\makebox(0,0){\strut{}\scalebox{0.9}{\tiny \color[HTML]{002b7f}  16384}}}%
      \colorrgb{0.00,0.17,0.50}
      \put(904,1604){\makebox(0,0){\strut{}\scalebox{0.9}{\tiny \color[HTML]{002b7f}  65536}}}%
      \csname LTb\endcsname
      \put(2316,834){\makebox(0,0)[r]{\strut{}\scriptsize$\mathrm{A}(\sigma_0, \epsilon_0)$}}%
      \colorrgb{0.40,0.70,1.00}
      \put(2236,2604){\makebox(0,0){\strut{}\scalebox{0.9}{\tiny \color[HTML]{66b3ff}  8}}}%
      \colorrgb{0.40,0.70,1.00}
      \put(2014,2530){\makebox(0,0){\strut{}\scalebox{0.9}{\tiny \color[HTML]{66b3ff}  20}}}%
      \colorrgb{0.40,0.70,1.00}
      \put(1792,2423){\makebox(0,0){\strut{}\scalebox{0.9}{\tiny \color[HTML]{66b3ff}  48}}}%
      \colorrgb{0.40,0.70,1.00}
      \put(1570,2316){\makebox(0,0){\strut{}\scalebox{0.9}{\tiny \color[HTML]{66b3ff}  112}}}%
      \colorrgb{0.40,0.70,1.00}
      \put(1348,2253){\makebox(0,0){\strut{}\scalebox{0.9}{\tiny \color[HTML]{66b3ff}  256}}}%
      \colorrgb{0.40,0.70,1.00}
      \put(1126,1998){\makebox(0,0){\strut{}\scalebox{0.9}{\tiny \color[HTML]{66b3ff}  576}}}%
      \colorrgb{0.40,0.70,1.00}
      \put(904,1769){\makebox(0,0){\strut{}\scalebox{0.9}{\tiny \color[HTML]{66b3ff}  1280}}}%
      \csname LTb\endcsname
      \put(2316,714){\makebox(0,0)[r]{\strut{}\scriptsize$\mathrm{A}(\sigma_1, \epsilon_1)$}}%
      \colorrgb{0.00,0.35,1.00}
      \put(2014,2254){\makebox(0,0){\strut{}\scalebox{0.9}{\tiny \color[HTML]{0059ff}  60}}}%
      \colorrgb{0.00,0.35,1.00}
      \put(1792,1844){\makebox(0,0){\strut{}\scalebox{0.9}{\tiny \color[HTML]{0059ff}  237}}}%
      \colorrgb{0.00,0.35,1.00}
      \put(1570,1483){\makebox(0,0){\strut{}\scalebox{0.9}{\tiny \color[HTML]{0059ff}  647}}}%
      \colorrgb{0.00,0.35,1.00}
      \put(1348,1258){\makebox(0,0){\strut{}\scalebox{0.9}{\tiny \color[HTML]{0059ff}  1240}}}%
      \colorrgb{0.00,0.35,1.00}
      \put(1126,1238){\makebox(0,0){\strut{}\scalebox{0.9}{\tiny \color[HTML]{0059ff}  1640}}}%
      \colorrgb{0.00,0.35,1.00}
      \put(904,1259){\makebox(0,0){\strut{}\scalebox{0.9}{\tiny \color[HTML]{0059ff}  1230}}}%
      \csname LTb\endcsname
      \put(1685,93){\makebox(0,0){\strut{}Mesh size $h$}}%
    }%
    \gplbacktext
    \put(0,0){\includegraphics[width={187.00bp},height={172.00bp}]{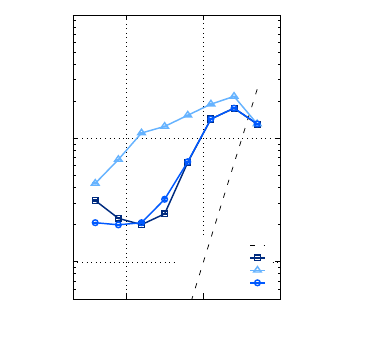}}%
    \gplfronttext
  \end{picture}%
\endgroup
\hspace{-0.15\textwidth}%
  }%
  \end{minipage}
  \captionof{figure}{\Fabrice{$\mathrm{L}^2$-norm error of the charge density for the manufactured solution at $t=0.5$ (left) and the diocotron instability at $t=0$ (right), as a function of the mesh size. Linear B-splines and $N=1{,}000{,}000$ particles are used. The admissibility tolerance and generalized-space parameters are set to either $(\sigma_0,\epsilon_0)=(0,0)$ or $(\sigma_1,\epsilon_1)=(4,10^{-3})$.}}
\label{fig:2}
 \end{center}
 
The adaptive mesh refinement strategy not only reduces the complexity of the approximation space while preserving the approximation accuracy, but also decreases the statistical error compared with both the sparse-grid and full-grid methods. These two improvements are closely related. By retaining only the mesh nodes that contribute most to the approximation, the average number of particles per active cell increases for a fixed total number of particles, thereby mitigating statistical noise over most of the computational domain.
Furthermore, enriching the approximation space with additional subspaces enables the sparse-grid approximation to achieve an accuracy comparable to that of the full-grid method for the diocotron instability while using significantly fewer mesh nodes. On fine meshes, it even outperforms the full-grid approximation owing to its lower statistical error. In particular, it achieves a more accurate approximation of the charge density while requiring $53$ times fewer mesh nodes than the full-grid method.
This substantial reduction in complexity results from two complementary mechanisms: (i) the enrichment of the sparse-grid approximation with hierarchical subspaces that are absent from the standard sparse-grid construction, allowing it to approach the accuracy of the full-grid method; and (ii) the adaptive selection of the mesh nodes that contribute most to the approximation, which reduces the dimension of the approximation space and, consequently, the statistical error.

Finally, we consider the diocotron instability beyond the initial state, after the instability has developed. The electron density at time $t=35$ is shown in~\cref{fig:5} for three methods with different configurations: STD-PIC with $N=2{,}621{,}440$ particles and $N_h=65{,}636$ mesh nodes; HSG-$(1)$-PIC using $N=588{,}800$ particles and $N_h=1{,}280$ mesh nodes; HSG-A$(\sigma,\epsilon)$ with $(\sigma,\epsilon)=(4,10^{-3})$, using the same number of particles but $N_h=1{,}695$ mesh nodes. The adaptive mesh is updated every $50$ iterations, i.e., $35$ times throughout the simulation. It is represented at time $t=35$ on the bottom right panel; black dots represent the sparse-grid adaptive mesh nodes and the full-grid mesh is represented in light grey. The mesh refinement strategy yields a substantial improvement in the quality of the charge density approximation for the same total number of particles while increasing the number of mesh nodes only marginally.

\begin{figure}
\begin{minipage}{\textwidth}
  \centering
  \begin{minipage}[]{0.37\textwidth}
\centering {\footnotesize STD-PIC} \\  
 \includegraphics[width=1\textwidth]{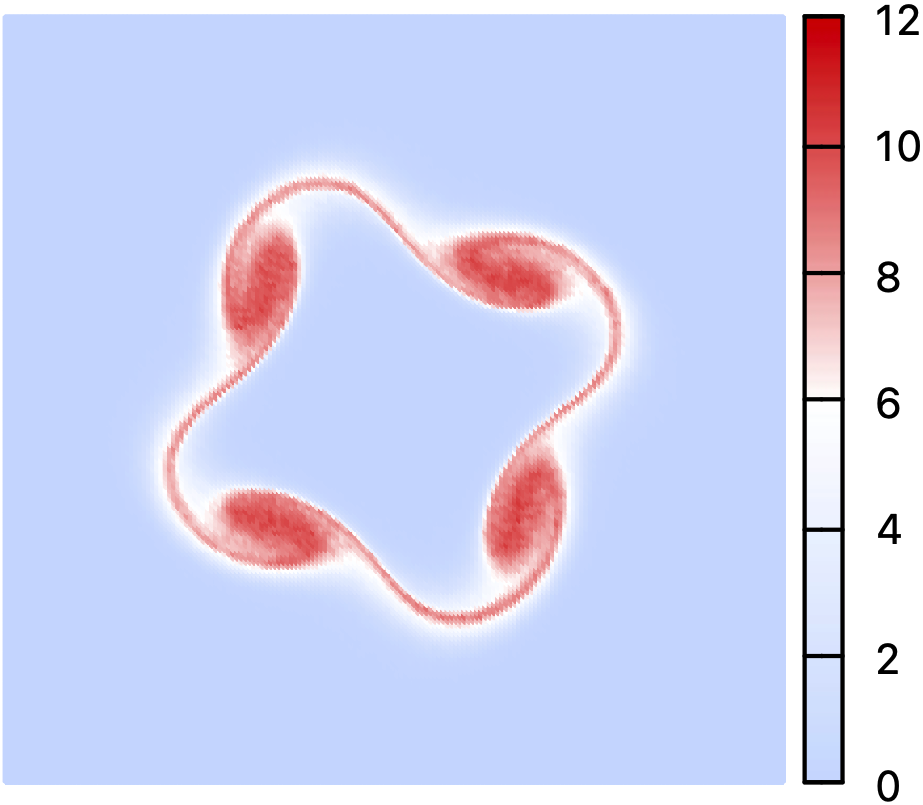}   
\end{minipage} \hspace{0.1\textwidth}%
  \begin{minipage}[]{0.37\textwidth}
\centering {\footnotesize HSG-$(1)$-PIC} \\  
 \includegraphics[width=1\textwidth]{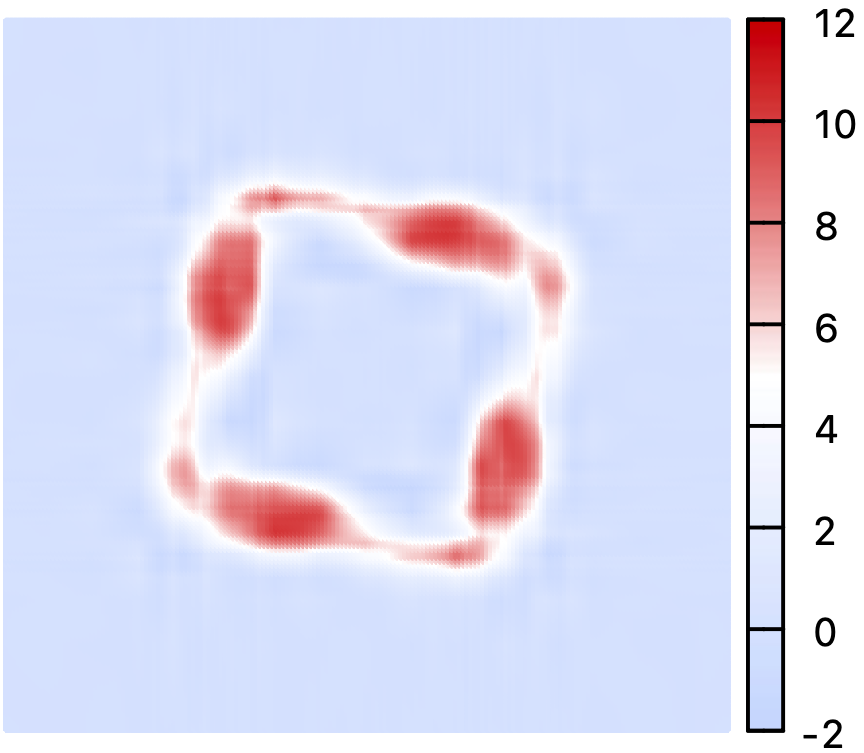}   
\end{minipage}  \\[0.5em]
  \begin{minipage}[]{0.37\textwidth}
\centering {\footnotesize HSG-A$(\sigma,\epsilon)$-PIC} \\  
 \includegraphics[width=1\textwidth]{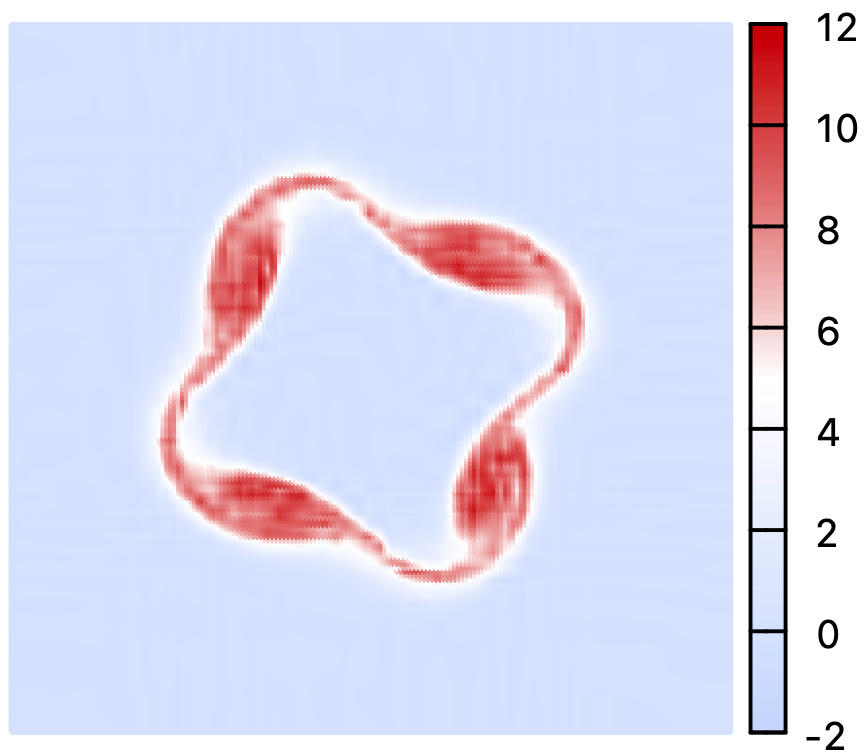}   
\end{minipage}  \hspace{0.1\textwidth}%
 \begin{minipage}[]{0.33\textwidth}
\centering {\footnotesize HSG-A$(\sigma,\epsilon)$-PIC mesh} \\  
 \includegraphics[width=1\textwidth]{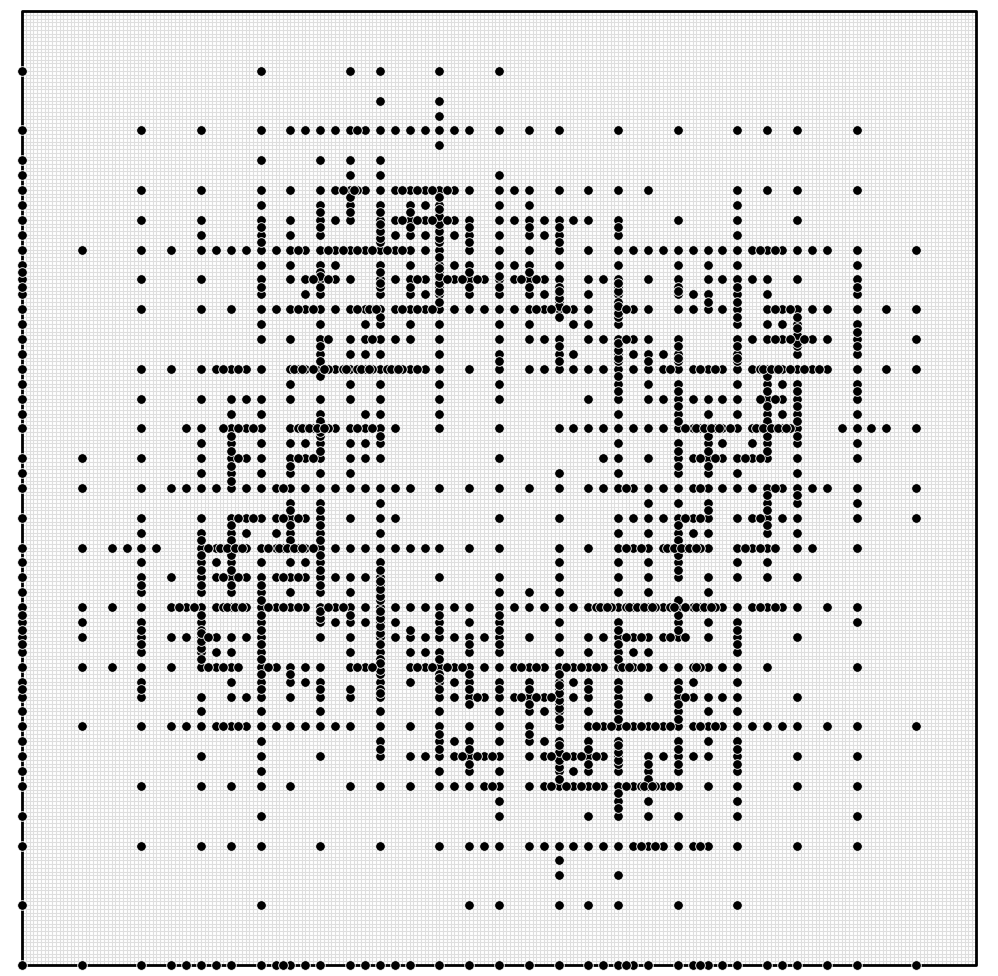}   
\end{minipage} 
\caption{Electron density for the diocotron instability at $t=35$, with mesh size $h=2^{-8}$ and linear B-splines. Three configurations are considered: STD-PIC with $N=2{,}621{,}440$ and $N_h=65{,}636$; HSG-$(1)$-PIC,
$N=588{,}800$, $N_h=1{,}280$;  HSG-A$(\sigma,\epsilon)$-PIC, with $\epsilon=10^{-3}$, $\sigma=4$, $N=588{,}800$, $N_h=1{,}695$. } \label{fig:5}
\end{minipage}
\end{figure}

\subsection{Performance}
We now compare the computational performance of the methods. The results reported in this section are intended to provide insight into their relative computational efficiency rather than to quantify the performance gains expected in large-scale plasma simulations. As shown in~\cite{deluzet23}, substantially larger speedups can be achieved in three-dimensional configurations with a dedicated parallel implementation.

All experiments are performed sequentially on a single CPU core, an Apple M4 processor with $32\mathrm{GB}$ of RAM. The methods are implemented in a code developed in Python and C++\footnote{The code is available at \url{https://gitlab.inria.fr/cguillet/sg-pic}}. The linear systems arising in the SGCT- and HSG-PIC methods are solved by a sparse Cholesky factorization using the \texttt{scipy.sparse} library~\cite{scipy20}, whereas those arising in the STD-PIC method are solved with a flexible GMRES solver preconditioned by algebraic multigrid from the \texttt{PETSc} library~\cite{petsc-user-ref}. We consider the manufactured solution test case, for which the exact solution is available at all times, enabling an accurate assessment of the approximation error. The Debye length is resolved whenever the mesh size satisfies $h<2^{-6}$.  The computational cost is assessed in terms of the total number of particles, mesh nodes, and the wall-clock execution time.

We report in~\cref{tab:1} the computational costs required to reduce the $\mathrm{L}^2$-norm approximation error of the charge density below the threshold $\varepsilon=10^{-2}$, at time $t=0.3$. For each configuration, we report the time execution per time step, averaged over the $15$ realizations, for the most computationally intensive steps of the algorithm: the charge density approximation (proj.), the field interpolation at the particle positions (inter.), the particle advance (push) and the mesh refinement (ref.).  The cost of the resolution of Poisson equation is negligible and therefore omitted here. The mesh refinement is performed every $7$ iterations, so that the time reported corresponds to three applications of the procedure, divided by $15$ iterations. 

\begin{table}[ht!]
\centering
\caption{Computational costs required to achieve a $\mathrm{L}^2$-norm error of
$\varepsilon=10^{-2}$ on the charge density at time $t=0.3$ for the manufactured solution.
Complexity and mean execution time per iteration, averaged over 15 iterations, are reported. HSG-A$(\sigma,\epsilon)$-PIC is run with $(\sigma,\epsilon)=(0,0.0005)$.}
\label{tab:1}
\setlength{\tabcolsep}{4pt}
\begin{tabular}{c|c|cc|cc|c|cccc}
\hline
method &
error &
$h$ & $p$ &
$N$ & $|N_h|$ &
\multicolumn{5}{c}{time per iteration [s]}
\\
\cline{7-11}
&
&
&
&
&
&
total & proj. & inter. & push & ref.
\\
\hline

STD
&0.0094
&$2^{-6}$&1
&28{,}672{,}000&4{,}096
&4.08&1.67&0.08&1.90&0
\\

HSG-$(\infty)$
&0.0097
&$2^{-6}$&1
&40{,}960{,}000&4{,}096
&29.4&20.1&6.4&3.09&0
\\

SGCT
&0.0100
&$2^{-7}$&1
&17{,}920{,}000&2{,}560
&15.1&13.3&0.63&1.17&0
\\

SGCT
&0.0099
&$2^{-6}$&3
&3{,}264{,}000&1{,}088
&2.89&2.48&0.28&0.20&0
\\

HSG-$(1)$
&0.0095
&$2^{-6}$&1
&2{,}816{,}000&256
&0.94&0.49&0.24&0.18&0
\\

HSG-$(\mathrm{E})$
& 0.0109
&$2^{-6}$&3
&1{,}920{,}000&160
&1.99&1.21&0.69&0.12&0
\\

HSG-A$(\sigma,\epsilon)$
& 0.0096
&$2^{-6}$&1
&1{,}620{,}000&91
&0.86&0.35&0.18&0.10&0.21
\\

\hline
\end{tabular}
\end{table}

All methods achieve the target accuracy, but with substantially different computational costs. The sparse-grid approaches, namely the SGCT-, HSG-$(1)$-, and HSG-$(\mathrm{E})$-PIC methods, require significantly fewer particles and mesh nodes than the STD-PIC method, even in the two-dimensional setting considered here, which is the least favorable regime for sparse-grid approaches, and despite the use of a non-optimized high-performance implementation.

Among the tested methods, HSG-A$(\sigma,\epsilon)$ provides the best computational efficiency. It requires $17$ times fewer particles and $45$ times fewer mesh nodes than STD-PIC, while reducing the wall-clock execution time by a factor of $5$.

{\cleparagraph
\section{Conclusion}
\label{sec:conclusion}
In this paper, we introduced new approximation spaces and a locally adaptive refinement strategy for the HSG-PIC method, with the objective of reducing the bias of sparse-grid PIC methods and decreasing the dimension of the approximation space while preserving their intrinsic noise-reduction properties. The proposed approach addresses two fundamental limitations of SGCT-PIC methods: the loss of accuracy for solutions with reduced mixed regularity and the difficulty of efficiently enriching the approximation space without reverting to a full-grid discretization.

First, we introduced an energy-based sparse-grid space that improves approximation efficiency by optimizing the relation between the $\mathrm{H}^1$-error and the number of degrees of freedom, as well as a family of generalized sparse-grid approximation spaces that provides a continuous transition between the classical sparse-grid and full-grid spaces. Numerical experiments confirmed that these spaces preserve the theoretical convergence rates of sparse-grid approximations while requiring substantially fewer mesh nodes than the full-grid discretization. Although the asymptotic optimality of the energy-based space is not necessarily reflected in pre-asymptotic regimes, it leads to a significantly more favorable, \Fabrice{dimensionality-independent}, complexity growth.

Second, we developed a locally adaptive refinement strategy based on the hierarchical structure of the HSG-PIC method. By exploiting hierarchical surpluses as local error indicators, the proposed algorithm selectively enriches the approximation space in regions where the sparse-grid representation is insufficient. The incremental construction of the adaptive space, combined with Schur complement eliminations, avoids solving the Galerkin problem on the complete enriched space and makes the refinement procedure computationally efficient.

The numerical results demonstrate the benefits of the proposed adaptive refinement algorithm. For solutions containing localized structures, the adaptive approximation substantially improves the charge density accuracy while preserving the statistical noise reduction characteristic of sparse-grid methods. This improvement results from two complementary effects: the enrichment of the approximation space through additional hierarchical subspaces, which reduces the bias, and the selective activation of relevant degrees of freedom, which increases the average number of particles per active cell and consequently reduces the statistical noise. 

Overall, the proposed locally adaptive HSG-PIC method provides a flexible compromise between the noise reduction of sparse-grid methods and the approximation capability of full-grid discretizations. It opens new perspectives for efficient simulations of kinetic plasma problems involving localized or anisotropic structures, for which nonadaptive sparse-grid approximations may fail to provide sufficient accuracy. Moreover, the B-spline formulation underlying the proposed approach naturally allows for an extension to non-rectangular geometries through geometric mappings, as proposed in~\cite{guillet25-1}. Such a development would broaden the applicability of adaptive HSG-PIC methods to more realistic plasma configurations. Future work will focus on the extension of the present method to fully three-dimensional simulations.
}
\section*{Acknowledgements}
This work has been carried out within the framework of the EUROfusion Consortium, funded by the European Union via the
Euratom Research and Training Programme (Grant Agreement No 101052200 — EUROfusion). 
Views and opinions expressed are however those of the author only and do not necessarily re­flect those of the European Union or the European Commission. Neither the European Union nor the European Commission can be held responsible for them.\\
This work has been supported by a grant from the French National Research Agency (ANR) project MATURATION (reference
ANR-22-CE46-0012) \\
Support from the FrFCM (Fédération de recherche pour la Fusion par Co­finement Magnétique) in the frame of the SPARCLE
project (SParse grid Acceleration for the paRticle-in-CelL mEthod) is also acknowledged.

\appendix

\bibliography{bib/bib}

@book{birdsall18,
	author = {Birdsall, C.K. and Langdon, A.B},
	publisher = {CRC Press},
	title = {Plasma physics via computer simulation},
	year = {2018}}

@article{bungartz99,
        title = {A Note on the Complexity of Solving {Poisson}'s Equation for Spaces of Bounded Mixed Derivatives},
        volume = {15},
        issn = {0885064X},
        number = {2},
        journal = {Journal of Complexity},
        author = {Bungartz, H.-J. and Griebel, M.},
        year = {1999},
        pages = {167--199},
}

@unpublished{deluzet26,
	author = {Deluzet, F. and Guillet, C. and Narski, J. },
	note = {Submitted to SIAM Journal on Numerical Analysis},
	title = {A hierarchical sparse-grid particle method for the {Vlasov}--{Poisson} system}}

@book{hockney88,
  title={Computer Simulation Using Particles},
  author={Hockney, R.W. and Eastwood, J.W.},
  isbn={9780070291089},
  series={Advanced book program},
  year={1981},
  publisher={McGraw-Hill International Book Company}
}

@article{bungartz04,
	author = {Bungartz, H.-J. and Griebel, M.},
	journal = {Acta Numerica},
	pages = {147--269},
	title = {Sparse grids},
	volume = {13},
	year = {2004}}

@article{griebel90,
	author = {Griebel, M. and Schneider, M. and Zenger, C.},
	journal = {Forschungsberichte, TU Munich},
	pages = {1-24},
	title = {A combination technique for the solution of sparse grid problems},
	volume = {TUM I 9038},
	year = {1990}}

@article{cottet84,
author = {Cottet, G.-H. and Raviart, P.-A.},
title = {Particle Methods for the One-Dimensional {Vlasov}--{Poisson} Equations},
journal = {SIAM Journal on Numerical Analysis},
volume = {21},
number = {1},
pages = {52-76},
year = {1984}}

@article{guillet24,
	author = {Guillet, C.},
	journal = {Multiscale Modeling and Simulation},
	pages = {891--924},
	publisher = {Society for Industrial & Applied Mathematics (SIAM)},
	title = {Semi-Implicit Particle-in-Cell Methods Embedding Sparse Grid Reconstructions},
	volume = {22},
	year = {2024},
	}

@article{deluzet22,
	author = {Deluzet, F. and Fubiani, G. and Garrigues, L. and Guillet, C. and Narski, J.},
	journal = {ESAIM: M2AN},
	pages = {1809--1841},
	title = {Sparse grid reconstructions for Particle-In-Cell methods},
	volume = {56},
	year = {2022}
	}

@article{guillet25,
title = {Energy-conserving Particle-In-Cell scheme based on {Galerkin} methods with sparse grids},
journal = {Journal of Computational Physics},
volume = {524},
pages = {113739},
year = {2025},
author = {C. Guillet},
}

@article{deluzet22-1,
title = {Efficient parallelization for 3d-3v sparse grid Particle-In-Cell: Shared memory architectures},
journal = {Journal of Computational Physics},
volume = {480},
pages = {112022},
year = {2023},
author = {Deluzet, F. and Fubiani, G. and Garrigues, L. and Guillet, C. and Narski, J.}}

@article{deluzet23,
title = {Efficient parallelization for 3D-3V sparse grid Particle-In-Cell: Single {GPU} architectures},
journal = {Computer Physics Communications},
volume = {289},
pages = {108755},
year = {2023},
author = {Deluzet, F. and Fubiani, G. and Garrigues, L. and Guillet, C. and Narski, J.}
}

@article{guillet25-1,
      title={Error Estimates for Sparse Tensor Products of {B}-spline Approximation Spaces}, 
      author={Guillet, C.},
      journal = {ESAIM: M2AN},
      year={2026},

}

@article{garrigues24,
    author = {Garrigues, L. and Chung-To-Sang, M. and Fubiani, G. and Guillet, C. and Deluzet, F. and Narski, J.},
    title = {Acceleration of particle-in-cell simulations using sparse grid algorithms. {I.} {Application} to dual frequency capacitive discharges},
    journal = {Physics of Plasmas},
    volume = {31},
    number = {7},
    pages = {073907},
    year = {2024}
}

@article{garrigues24-1,
    author = {Garrigues, L. and Chung-To-Sang, M. and Fubiani, G. and Guillet, C. and Deluzet, F. and Narski, J.},
    title = {Acceleration of particle-in-cell simulations using sparse grid algorithms. {II.} {Application} to partially magnetized low temperature plasmas},
    journal = {Physics of Plasmas},
    volume = {31},
    number = {7},
    pages = {073908},
    year = {2024}}

@article{garrigues21,
    author = {Garrigues, L. and Tezenas du Montcel, B. and Fubiani, G. and Bertomeu, F. and Deluzet, F. and Narski, J.},
    title = {Application of sparse grid combination techniques to low temperature plasmas particle-in-cell simulations. {I.} {Capacitively} coupled radio frequency discharges},
    journal = {Journal of Applied Physics},
    volume = {129},
    number = {15},
    pages = {153303},
    year = {2021}}

@article{deluzet25,
author = {Deluzet, F. and Guillet, C. and Narski, J. and Pace, P.},
title = {High-Order Sparse-{PIC} Methods: Analysis and Numerical Investigations},
journal = {SIAM Journal on Numerical Analysis},
volume = {63},
number = {3},
pages = {1281-1314},
year = {2025}}

@article{ricketson17,
	author = {Ricketson, L.F. and Cerfon, A.J.},
	journal = {Plasma Phys. Control. Fusion},
	number = {2},
	pages = {024002},
	title = {Sparse grid techniques for particle-in-cell schemes},
	volume = {59},
	year = {2017}}

@article{muralikrishnan21,
title = {Sparse grid-based adaptive noise reduction strategy for particle-in-cell schemes},
journal = {Journal of Computational Physics: X},
volume = {11},
pages = {100094},
year = {2021},
issn = {2590-0552},
author = {Muralikrishnan, S. and Cerfon, A.J. and Frey, M. and Ricketson, L.F. and Adelmann, A.}}

@article{driscoll90,
	author = {Driscoll, C. F. and Fine, K. S.},
	journal = {Physics of Fluids B: Plasma Physics},
	number = {6},
	pages = {1359--1366},
	publisher = {AIP Publishing},
	title = {Experiments on vortex dynamics in pure electron plasmas},
	volume = {2},
	year = {1990}}

@article{denton95,
  title={{$\delta$f} algorithm},
  author={Denton, R.E. and Kotschenreuther, M.},
  journal={Journal of Computational Physics},
  volume={119},
  number={2},
  pages={283--294},
  year={1995}}

@article{gassama07,
	author = {Gassama, S. and Sonnendrücker, E. and Schneider, K. and Farge, M. and Domingues, M. O.},
	title = {Wavelet denoising for postprocessing of a 2D Particle-In-Cell code},
	url= "https://doi.org/10.1051/proc:2007013",
	journal = {ESAIM: Proc.},
	year = 2007,
	volume = 16,
	pages = "195-210"}

@article{langdon70,
  title={Effects of the spatial grid in simulation plasmas},
  author={Langdon, A.B.},
  journal={Journal of Computational Physics},
  volume={6},
  number={2},
  pages={247--267},
  year={1970},
}

@article{scipy20,
  author  = {Virtanen, P. and Gommers, R. and Oliphant, T.E. and
            Haberland, M. and Reddy, T. and Cournapeau, D. and
            Burovski, E. and Peterson, P. and Weckesser, W. and
            Bright, J. and {van der Walt}, S. J. and
            Brett, M. and Wilson, J. and Millman, K. J. and
            Mayorov, N. and Nelson, A.R.J. and Jones, E. and
            Kern, R. and Larson, E. and Carey, C. J. and
            Polat, I. and Feng, Y. and Moore, E.W. and
            {VanderPlas}, J. and Laxalde, D. and Perktold, J. and
            Cimrman, R. and Henriksen, I. and Quintero, E. A. and
            Harris, C.R. and Archibald, A.M. and
            Ribeiro, A. H. and Pedregosa, F. and
            {van Mulbregt}, P. and {SciPy 1.0 Contributors}},
  title   = {{{SciPy} 1.0: Fundamental Algorithms for Scientific
            Computing in Python}},
  journal = {Nature Methods},
  year    = {2020},
  volume  = {17},
  pages   = {261--272}
}

@techreport{petsc-user-ref,
    title = {{PETSc/TAO} Users Manual},
	author = {Balay, S. and Abhyankar, S. and Adams, M.F. and Benson, S. and Brown, J. and Brune, P. and Buschelman, K. and Constantinescu, E. and Dalcin, L. and Dener, A. and Eijkhout, V. and Faibussowitsch, J. and Gropp, W.D. and Hapla, V. and Isaac, T. and Jolivet, P. and Karpeev, D. and Kaushik, D. and Knepley, M.G. and Kong, F. and Kruger, S. and May, D.A. and McInnes, L.C. and Mills, R.T. and Mitchell, L. and Munson, T. and Roman, J.E. and Rupp, K. and Sanan, P. and Sarich J. and Smith, B.F. and Zampini, S. and Zhang, H. and Zhang, J.},
	institution = {Argonne National Laboratory},
	number = {ANL-21/39 - Revision 3.21},
	year = {2024}}

@article{sydora99,
title = {Low-noise electromagnetic and relativistic particle-in-cell plasma simulation models},
journal = {Journal of Computational and Applied Mathematics},
volume = {109},
number = {1},
pages = {243-259},
year = {1999},
author = {R.D. Sydora}}

@article{aydemir94,
    author = {Aydemir, A. Y.},
    title = {A unified Monte Carlo interpretation of particle simulations and applications to non--neutral plasmas},
    journal = {Physics of Plasmas},
    volume = {1},
    number = {4},
    pages = {822-831},
    year = {1994}}

@article{crestetto18,
  TITLE = {{A particle micro-macro decomposition based numerical scheme for collisional kinetic equations in the diffusion scaling}},
  AUTHOR = {Crestetto, A. and Crouseilles, N. and Lemou, M.},
  JOURNAL = {{Communications in Mathematical Sciences}},
  PUBLISHER = {{International Press}},
  VOLUME = {16},
  NUMBER = {4},
  PAGES = {887-911},
  YEAR = {2018}}

@InProceedings{obersteiner21,
author = {Obersteiner, M. and Bungartz, H.-J.},
title="A Spatially Adaptive Sparse Grid Combination Technique for Numerical Quadrature",
booktitle="Sparse Grids and Applications - Munich 2018",
year="2021",
publisher="Springer International Publishing",
pages="161--185"}

@article{obersteiner21-1,
author = {Obersteiner, M. and Bungartz, H.-J.},
title = {A Generalized Spatially Adaptive Sparse Grid Combination Technique with Dimension-wise Refinement},
journal = {SIAM Journal on Scientific Computing},
volume = {43},
number = {4},
pages = {A2381-A2403},
year = {2021}}

@article{pfluger10,
title = {Spatially adaptive sparse grids for high-dimensional data-driven problems},
journal = {Journal of Complexity},
volume = {26},
number = {5},
author = {Pflüger, D. and Peherstorfer, B. and Bungartz, H.-J.},
pages = {508-522},
year = {2010}}

@article{tranquilli22,
title = {A deterministic verification strategy for electrostatic particle-in-cell algorithms in arbitrary spatial dimensions using the method of manufactured solutions},
journal = {Journal of Computational Physics},
volume = {448},
pages = {110751},
year = {2022},
author = {Tranquilli, P. and Ricketson, L. and Chac{\'o}n, L.}}

@article{bungartz91,
  title={An adaptive {Poisson} solver using hierarchical bases and sparse grids},
  author={H.-J. Bungartz},
  journal={Forschungsberichte, TU Munich},
  year={1991},
  volume={TUM I 9130},
  pages={1-21},
}

@article{bokanowski13,
	author = {Bokanowski, O. and Garcke, J. and Griebel, M. and Klompmaker, I.},
	doi = {10.1007/s10915-012-9648-x},
	isbn = {1573-7691},
	journal = {Journal of Scientific Computing},
	number = {3},
	pages = {575--605},
	title = {An Adaptive Sparse Grid Semi-Lagrangian Scheme for First Order {Hamilton}-{Jacobi} {Bellman} Equations},
	volume = {55},
	year = {2013}}

@InProceedings{griebel99,
author="Griebel, M. and Zumbusch, G.",
title="Adaptive Sparse Grids for Hyperbolic Conservation Laws",
booktitle="Hyperbolic Problems: Theory, Numerics, Applications",
year="1999",
publisher="Birkh{\"a}user Basel",
pages="411--422"}
\bibliographystyle{siamplain}

\end{document}